\documentclass[12pt]{amsart} 
\usepackage{amsmath, amssymb, amsfonts, amsbsy, amsthm, comment, latexsym, stmaryrd, mathtools, fullpage, enumerate, ytableau, shuffle, multirow, mathdots, tikz, bbm, xcolor, MnSymbol, mathtools, graphicx, textpos, rotating, youngtab, psfrag, float, hyperref}

\numberwithin{equation}{section}

\theoremstyle{plain}

\newtheorem{thm}{Theorem}
\newtheorem*{thm*}{Theorem}
\numberwithin{thm}{section}
\newtheorem*{thm1.1}{Theorem~1.1}
\newtheorem*{thm1.2}{Theorem~1.2}
\newtheorem*{thm1.3}{Theorem~1.3}

\newtheorem{lem}[thm]{Lemma}
\newtheorem*{lem*}{Lemma}

\newtheorem{prop}[thm]{Proposition}
\newtheorem*{prop*}{Proposition}

\newtheorem*{corollary*}{Corollary}

\newtheorem*{observation*}{Observation}

\theoremstyle{definition}

\newtheorem*{claim*}{Claim}

\newtheorem*{conjecture*}{Conjecture}

\newtheorem{defn}[thm]{Definition}
\newtheorem*{defn*}{Definition}

\newtheorem*{fact*}{Fact}

\newtheorem*{case*}{Case}

\newtheorem{example}[thm]{Example}
\newtheorem*{example*}{Example}

\newtheorem*{assumption*}{Assumption}

\newtheorem{approach*}{Approach}

\theoremstyle{remark}

\newtheorem{note}[thm]{Note}
\newtheorem*{note*}{Note}

\newtheorem{remark}[thm]{Remark}
\newtheorem*{remark*}{Remark}

\newcommand{\I}{\mathcal{I}}
\newcommand{\Z}{\mathbb{Z}}

\newcommand{\R}{\mathbb{R}}
\newcommand{\cal}[1]{\mathcal{#1}}
\newcommand{\diam}{\text{diam}}
\newcommand{\GT}{\text{GT}}
\newcommand{\Aut}{\text{Aut}}

\newcommand{\lam}{\lambda}
\newcommand{\gamlam}{\Gamma_\lambda}

\author[Yibo Gao]{Yibo Gao}
\address{Department of Mathematics, Massachusetts Institute of Technology, \mbox{Cambridge, MA 02139}}
\email{\href{mailto:gaoyibo@mit.edu}{{\tt gaoyibo@mit.edu}}}
\author[Benjamin Krakoff]{Benjamin Krakoff}
\address{Department of Mathematics, Yale University, \mbox{New Haven, CT 06520}}
\email{\href{mailto:benjamin.krakoff@yale.edu}{{\tt benjamin.krakoff@yale.edu}}}
\author[Lisa Yang]{Lisa Yang}
\address{Department of Mathematics, Massachusetts Institute of Technology, \mbox{Cambridge, MA 02139}}
\email{\href{mailto:lisayang@mit.edu}{{\tt lisayang@mit.edu}}}

\begin{document}

\title{Diameter and Automorphisms of Gelfand-Tsetlin Polytopes}
\date{\today}

\begin{abstract}
We determine the diameter of the 1-skeleton and the combinatorial automorphism group of any Gelfand-Tsetlin polytope $\GT_{\lambda}$ associated to an integer partition $\lambda.$
\end{abstract}

\maketitle

\section{Introduction and Statement of Results}

Gelfand-Tsetlin (GT) polytopes are compact convex polytopes defined by a set of linear inequalities depending on a partition $\lam$ as shown in Figure~\ref{fig:ineqdiagram}. The polytope $\GT_\lam$ corresponds to all fillings of this triangular array with real numbers such that all rows and columns are weakly increasing.

\begin{figure}[h]
\begin{equation*}
\begin{matrix}
\lambda_1 \\
\rotatebox[origin=c]{270}{$\leq$} \\
x_{2, 1} & \leq & \lambda_2 \\
\rotatebox[origin=c]{270}{$\leq$} & & \rotatebox[origin=c]{270}{$\leq$} \\
x_{3, 1} & \leq & x_{3, 2} & \leq & \lambda_3 \\
\rotatebox[origin=c]{270}{$\leq$} & & \rotatebox[origin=c]{270}{$\leq$} & & \rotatebox[origin=c]{270}{$\leq$}\\
x_{4, 1} & \leq & x_{4, 2} & \leq & x_{4, 3} & \leq & \lambda_4 \\
\vdots & & \vdots & & \vdots & & \rotatebox[origin=c]{-45}{$\ldots$} \\
x_{n, 1} & \leq & x_{n, 2} & \leq & \ldots & \leq & x_{n, n-1} & \leq & \lambda_n
\end{matrix}
\end{equation*}
\caption{Inequality constraints of GT polytopes.}
\label{fig:ineqdiagram}
\end{figure}

GT polytopes arise from the study of representations of $GL_n(\mathbb{C})$ and have connections to areas of representation theory and algebraic geometry (see for example \cite{kirichenko2012schubert}). For any integer partition $\lam = (\lam_1,\ldots,\lam_n)$, let $n$ be the length of $\lam$ and let $\GT_\lam$ denote the GT polytope associated to $\lam$. Then the integral points within $\GT_\lam$ are in bijection with the number of semi-standard Young tableaux of shape $\lam$ with tableaux entries bounded by $n$. Furthermore, the integral points of $\GT_\lam$ parametrize a \emph{Gelfand-Tsetlin basis} of the $GL_n$-module with highest weight $\lam$, so the number of integral points equals the dimension of this module. GT polytopes can also be viewed as the marked order polytope of a poset which is discussed in \cite{stanley1986two} and \cite{ardila2011gelfand}. 

This paper describes the diameter of the 1-skeleton and the combinatorial automorphism group of GT polytopes. Since our results are purely combinatorial, it suffices to consider partitions of the form $\lambda = (1^{a_1}, 2^{a_2}, \ldots, m^{a_m})$ where $a_i > 0$ for all $1 \le i \le m$. This is explained in Section~\ref{sec:preliminaries} and in  Remark~\ref{rmk:ai}.

The 1-skeleton of a polytope is the graph obtained by looking only at its vertices and edges (the $0$ and $1$ dimensional faces). The shortest path between two vertices is the path with the minimum number of edges. The diameter of a graph is the maximum over all pairs of vertices of the length of the shortest path between the pair of vertices. Our first result is a formula for the diameter of the 1-skeleton of $\GT_\lam$ which we denote by $\diam(\GT_\lam)$.

\noindent We adopt the following notational conventions:
\begin{itemize}
\item The Kronecker delta function is defined as $\delta_{x,y} :=
    \begin{cases}
            1, &         \text{if } x=y,\\
            0, &         \text{if } x\neq y.
    \end{cases}$
\item For any $a \in \Z{>0}$, let $S_a$ denote the symmetric group on a set of size $a$.
\item For $\lam = (1^{a_1}, 2^{a_2}, \ldots, m^{a_m})$, its reverse partition is $\lam' := (1^{a_m},2^{a_{m-1}},\ldots,m^{a_1})$. If $a_i = a_{m+1-i}$ for all $1 \le i \le m$, then we say $\lam = \lam'$ and call $\lam$ a \emph{reverse symmetric} partition. 
\end{itemize}

\begin{thm}[Diameter of 1-skeleton]
For any GT polytope $\GT_\lambda$, $\diam(\GT_\lambda) = 2m - 2 - \delta_{1, a_1} - \delta_{1, a_m}$.
\label{thm:diameter}
\end{thm}

A polytope's face poset (or face lattice) is formed by ordering its faces by inclusion. A combinatorial automorphim of a polytope is an automorphism of its face poset. Our second result is a description of the combinatorial automorphism group of $\GT_\lam$. 

\begin{thm}[$m = 2$ Automorphisms]
Suppose $\lam = (1^{a_1}, 2^{a_2})$ and $a_1,a_2 \ge 2$. 
If $a_1 = a_2 = 2$, then 
\begin{equation*}
\Aut(\GT_\lambda) \cong D_4 \times \Z_2.
\end{equation*}
Otherwise, 
\begin{equation*}
\Aut(\GT_\lambda) \cong D_4 \times \Z_2 \times \Z_2^{\delta_{a_1, a_2}},
\end{equation*}
where $D_4$ is the dihedral group of order 8 and $\Z_2$ is the cyclic group of order 2.
\label{thm:autom-m=2}
\end{thm}

\begin{thm}[$m \ge 3$ Automorphisms]
Suppose $\lam = (1^{a_1},\ldots, m^{a_m})$ and $m \ge 3$. Let $r_1$ be the number of $k$ such that $a_k, a_{k+1} \geq 2$. Let $r_2 = 1$ if $\lam = \lam'$ and let $r_2 = 0$ otherwise. Then 
\begin{equation*}
\Aut(\GT_\lam) \cong (S_{a_2}^{\delta_{1,a_1}} \times S_{a_{m-1}}^{\delta_{1,a_{m}}} \times \Z_2^{r_1+1}) \ltimes_\varphi \Z_2^{r_2}
\end{equation*}
where if $r_2 = 1$, then $\varphi: \Z_2 \to \Aut(S_{a_2}^{\delta_{1,a_1}} \times S_{a_{m-1}}^{\delta_{1,a_{m}}} \times \Z_2^{r_1+1})$ sends the nonidentity element of $\Z_2$ to the map sending $(\sigma_1, \sigma_2, z_1, \ldots, z_{r_1}, z_{r_1+1}) \mapsto (\sigma_2, \sigma_1, z_{r_1}, \ldots, z_1, z_{r_1+1}).$
\label{thm:autom-m>=3}
\end{thm}

\begin{example}
Figure~\ref{fig:gt123} shows the Gelfand-Tsetlin polytope $\GT_{(1,2,3)}$. It is clear that the diameter of $\GT_{(1,2,3)}$ is 2 and the automorphism group is $\Z_2\times\Z_2$. One generator is given by rotating 180$^\circ$ about the axis containing vertex $4$ and vertex $7$. The other generator is given by fixing vertex $1$ and interchanging vertex $2$ with vertex $3$, vertex $5$ with vertex $6$, and vertex $4$ with vertex $7$. These results are predicted by our theorems.
\begin{figure}[h!]
\includegraphics[scale=0.3]{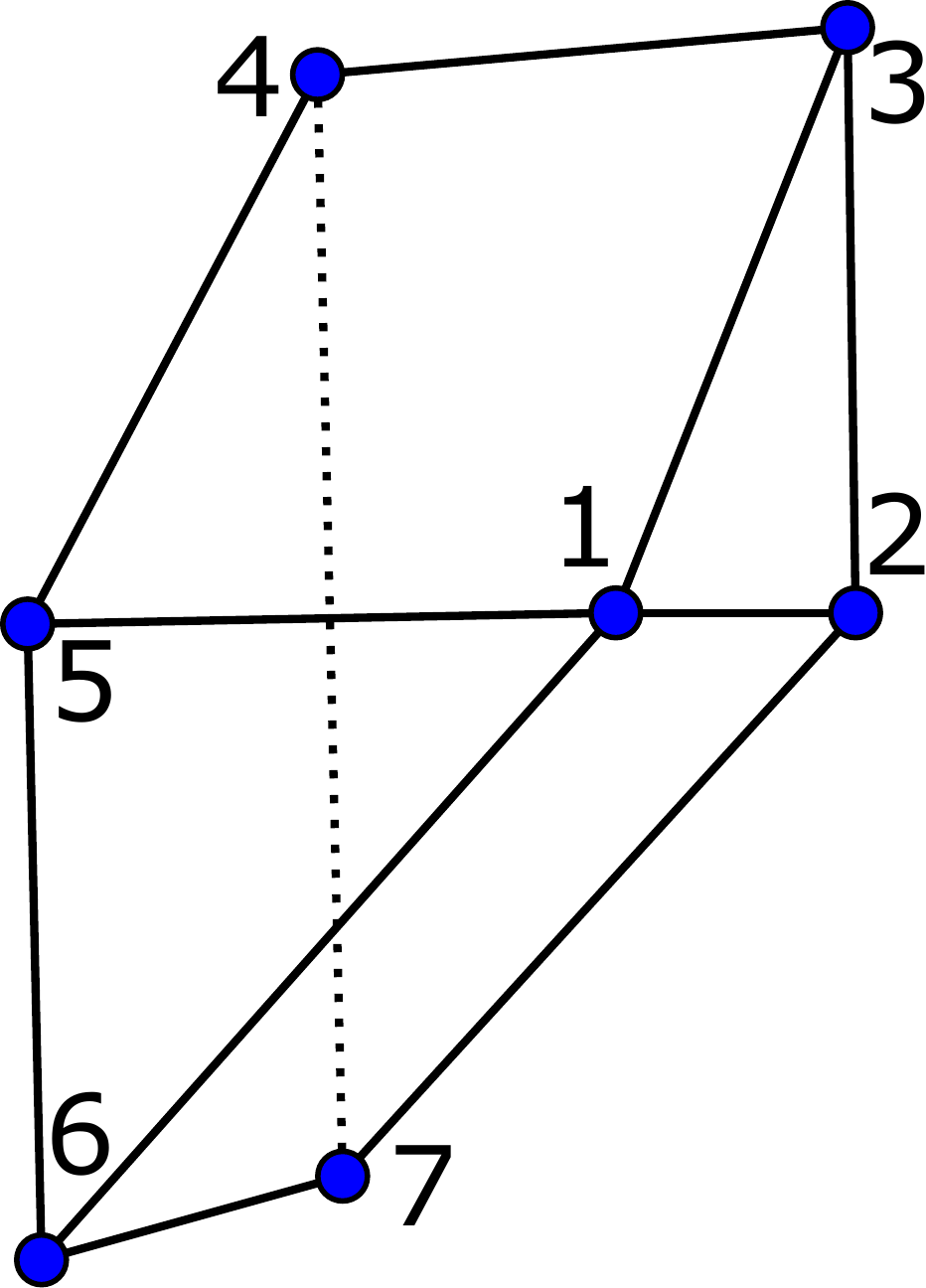}
\caption{Polytope $\GT_{(1,2,3)}$.}
\label{fig:gt123}
\end{figure}
\end{example}

In Section~\ref{sec:preliminaries}, we review background on GT polytopes including how to model their face poset using combinatorial objects called \emph{ladder diagrams}. Then we present a proof of Theorem~\ref{thm:diameter} in Section~\ref{sec:comb-diameter}. In Section~\ref{sec:comb-automorphisms}, we identify the generators of $\Aut(\GT_\lam)$ and present proofs of Theorems~\ref{thm:autom-m=2} and \ref{thm:autom-m>=3}.

\section{Preliminaries}\label{sec:preliminaries}
In this section, we formally define GT polytopes and describe how to view their faces as combinatorial objects called $\emph{ladder diagram}$. These diagrams are the primary objects we will use to model the face lattice of $\GT_\lam$.

\subsection{Gelfand-Tsetlin Polytopes}
A partition of $s$ is a sequence of weakly increasing positive integers $\lambda = (\lambda_1 \leq \lambda_2 \leq \ldots \leq \lambda_n)$ such that $\sum_{i=1}^n \lambda_i = s$. We will often use multiplicative notation for $\lambda$ and write $\lambda=(\lam_1^{a_1}, \lam_2^{a_2}, \ldots, \lam_m^{a_m})$ for $a_1, \ldots, a_m \in \mathbb{Z}_{\ge 0}$ to denote a partition with $a_1$ copies of $\lam_1$, $a_2$ copies of $\lam_2$, and so forth. We may omit writing the term $\lam_i^{a_i}$ if $a_i = 0$. 

\begin{defn}[GT Polytope]\label{def:GTpoly}
Given a partition $\lambda = (\lambda_1, \ldots, \lambda_n)$, the Gelfand-Tsetlin Polytope $\GT_\lambda$ is the set of points $\vec{x} = (x_{i, j})_{1 \leq j \leq i \leq n} \in \mathbb{R}^{n(n+1)/2}$ such that $x_{i,i}=\lambda_i$ for $1 \le i \le n$ and such that the following inequalities are satisfied:
\begin{enumerate}
\item $x_{i-1,j}\leq x_{i,j}\leq x_{i+1,j}$,
\item $x_{i,j-1}\leq x_{i,j}\leq x_{i,j+1}$.
\end{enumerate}

Suppose for some $i < j$ that $\lambda_i = \lambda_j$. Then for every $i\leq i',j'\leq j$, we are forced to have $x_{i', j'} = \lambda_i$. Whenever such a situation occurs, we say that the coordinate $x_{i', j'}$ is \emph{fixed}. In general, $\GT_\lam$ will be a polytope in $\mathbb{R}^d$ where $d$ is at most $n(n-1)/2$.

These constraints can be visualized in a triangular array as shown in Figure~\ref{fig:ineqdiagram}. 
\end{defn}

We adopt the following notational conventions:
\begin{itemize}
\item Let $n$ denote the length of $\lambda$. We specify a partition as $\lam = (\lam_1,\ldots,\lam_n)$. We typically only consider partitions of the form $\lam = (1^{a_1}, \ldots, m^{a_m})$ where $a_i > 0$ for all $1 \le i \le m$.
\item Let $m$ denote the number of distinct values of $\lambda$. This agrees with our notation for partitions of the form $\lam = (1^{a_1}, \ldots, m^{a_m})$ where $a_i > 0$ for all $1 \le i \le m$.
\item Let $d$ denote the dimension of $\GT_\lambda$. Note that $d = {n \choose 2} - \sum\limits_{i=1}^m{a_i \choose 2}$.
\item Let $\mathcal{F}(\GT_\lam)$ denote the face poset of $\GT_\lambda$ ordered by inclusion.
\item Let $\I_n=\{(i,j):1\leq j\leq i\leq n\}$ denote the triangular grid with shape shown in Figure~ \ref{fig:ineqdiagram}.
\end{itemize}

\subsection{Ladder Diagrams}\label{subsec:ladder-diagrams}
For every $\lam$, we define a graph $\gamlam$ such that faces of $\GT_\lam$ correspond to subgraphs of $\gamlam$ with certain restrictions. These subgraphs are the ladder diagrams introduced in \cite{ACK}.

Let $Q$ be the infinite graph corresponding to first quadrant of the Cartesian plane, i.e. let $Q$ have vertices $(i,j)$ for all $i,j \ge 0$ and edges $\{(i,j),(i+1,j)\}$ and $\{(i,j),(i,j+1)\}$.
For the sake of convenience, define $a_0 := 0$ and $s_j := \sum_{i=0}^j a_i$ for $0 \le j \le m$.

\begin{defn}[$\gamlam$ and Ladder Diagrams]
For $\lambda = (1^{a_1},\ldots,m^{a_m})$, the grid $\gamlam$ is an induced subgraph of $Q$ constructed as follows. We identify the vertex at $(0,0)$ as the \emph{origin} and the vertices $t_j = (s_j, n - s_j)$ for $0 \le j \le m$ as \emph{terminal vertices}. $\gamlam$ consists of all vertices and edges appearing on any North-East path between the origin and a terminal vertex.

A \emph{ladder diagram} is a subgraph of $\gamlam$ such that
\begin{enumerate}
\item the origin is connected to every terminal vertex by some North-East path.
\item every edge in the graph is on a North-East path from the origin to some terminal vertex.
\end{enumerate}
\end{defn}

An example of the grid $\gamlam$ and some of its ladder diagrams are shown in Figure~\ref{fig:ladder-diag}. All fixed coordinates are shaded. The terminal vertices lie along the \emph{main diagonal} of $\gamlam$. As standalone objects, the ladder diagrams of $\gamlam$ form a poset ordered by inclusion: given two ladder diagrams $\cal{L}_1, \cal{L}_2$, we have $\cal{L}_1 \le \cal{L}_2$ if $\cal{L}_1$ is a subgraph of $\cal{L}_2$.  

\begin{figure} [htp] 
\includegraphics[scale=0.12]{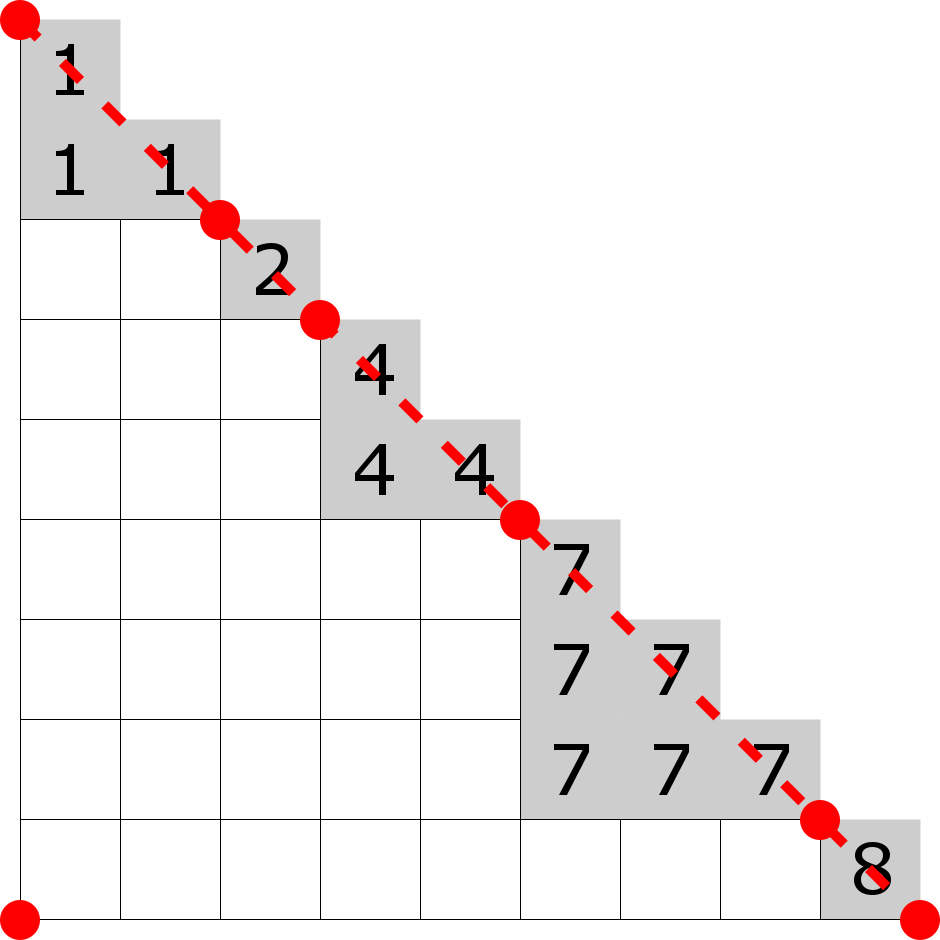}
\includegraphics[scale=0.12]{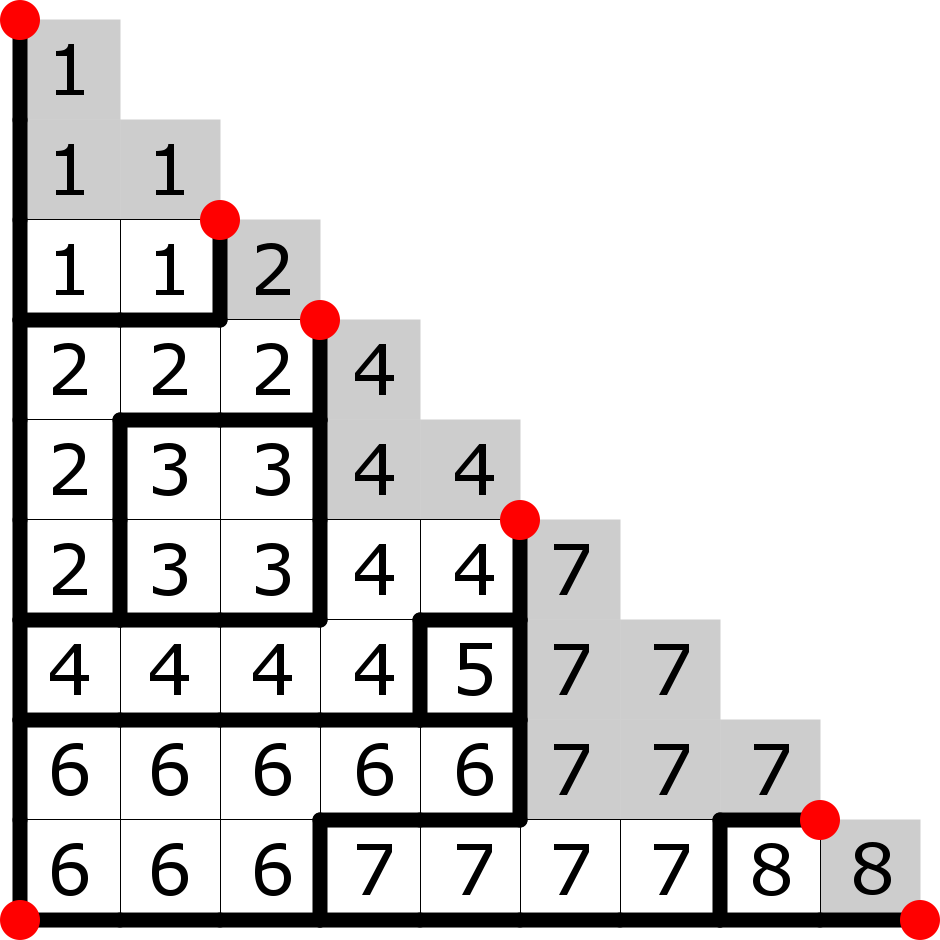}
\includegraphics[scale=0.12]{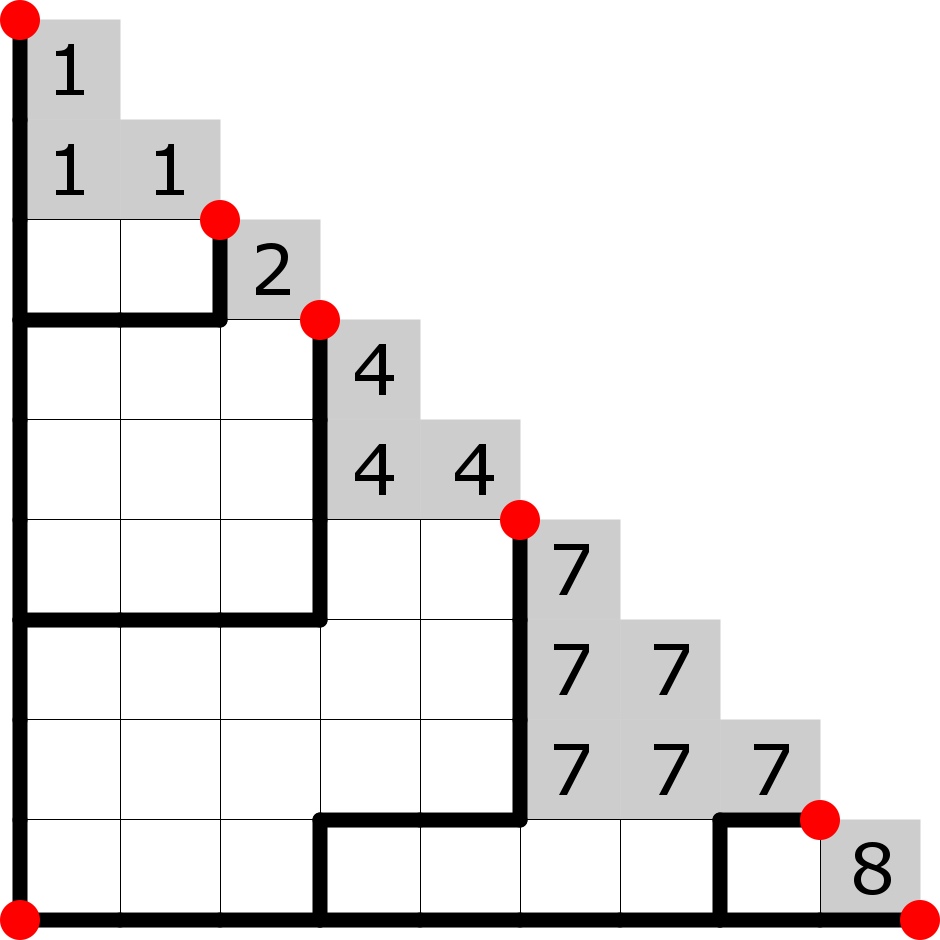}
\includegraphics[scale=0.12]{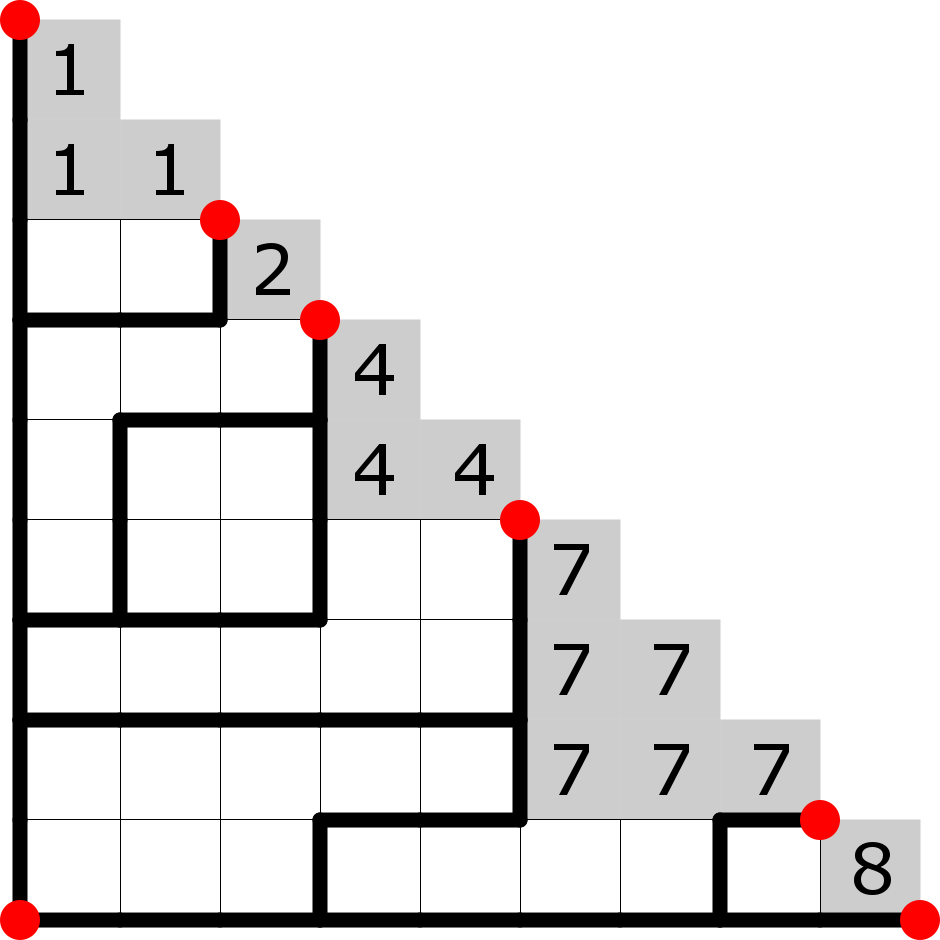}
\caption{Let $\lam = (1^2,2^1,4^2,7^3,8^1)$. From left to right: $\gamlam$ with origin and terminal vertices as dots and a dashed line indicating the main diagonal, ladder diagram for a point in $\GT_\lam$, ladder diagram for a $0$-dimensional face (vertex), and ladder diagram for a $2$-dimensional face.}
\label{fig:ladder-diag}
\end{figure}

\begin{defn}[Face Lattice of Ladder Diagrams]
Let $\cal{F}(\gamlam)$ denote the set of all ladder diagrams of $\gamlam$ ordered by inclusion. This may also be called the face lattice of $\gamlam$.
\end{defn}

\begin{thm}\cite[Theorem 1.9]{ACK}\label{thm:ladders}
$\mathcal{F}(\GT_\lambda) \cong \mathcal{F}(\gamlam)$.
\end{thm}
\begin{proof}
An isomorphism is given by taking a point in $\GT_\lambda$ and drawing lines around adjacent groups of $x_{i, j}$ with equal value will produce a face of $\Gamma_a$. For more details, see \cite{ACK}. We also mention that \cite{McAllister} proves an analogous relation between $\cal{F}(\GT_\lambda)$ and the poset of GT tilings, which is essentially equivalent to $\cal{F}(\gamlam)$.
\end{proof}

Note that $\mathcal{F}(\gamlam)$ is graded by the number of bounded regions where $k$-dimensional faces correspond to ladder diagrams with $k$ bounded regions. In fact, given a point $x \in \GT_\lam$, we can determine the dimension of the minimal face containing $x$ by mapping $x$ to its corresponding ladder diagram. Then the number of bounded regions will be the dimension of this minimal face.

\begin{remark}
Note that the poset $\mathcal{F}(\gamlam)$ only depends on the multiplicities $a_i$ and not on the values of $\lambda_i$. So when examining the purely combinatorial properties of $\GT_\lambda$, it suffices to consider partitions of the form $\lambda = (1^{a_1}, 2^{a_2}, \ldots, m^{a_m})$.
\label{rmk:ai}
\end{remark}

\section{Combinatorial Diameter}\label{sec:comb-diameter}
In this section, we present an exact formula for the diameter of the 1-skeleton of $\GT_\lambda$ which is denoted by $\diam(\GT_\lambda$). As explained in Remark~\ref{rmk:ai}, it suffices to consider $\lambda = (1^{a_1}, \ldots, m^{a_m})$ where $a_1, \ldots ,a_m \in \mathbb{Z}_{>0}$. Our proofs will indirectly work with vertices and edges of $\GT_\lam$ by working with their corresponding ladder diagrams.

In order to study the diameters of the 1-skeleton of $\GT_\lambda$, we need to first understand what a vertex is in terms of ladder diagrams and under what conditions two vertices are connected.

\begin{defn}
Two paths in a ladder diagram from the origin to terminal vertices are \emph{noncrossing} if they do not meet again after their first separation.
\end{defn}

In particular, vertices of $\GT_\lam$ have ladder diagrams consisting of $m+1$ noncrossing paths. Two vertices are connected by an edge if the union of their ladder diagrams has exactly one bounded region, as shown in Figure~\ref{fig:adj-vertices}. We think of traveling from one vertex to another as moving a \emph{subpath} of the first vertex's ladder diagram. Note that such a move can alter more than one of the $m+1$ noncrossing paths.

\begin{figure} [htp]
\includegraphics[scale=0.10]{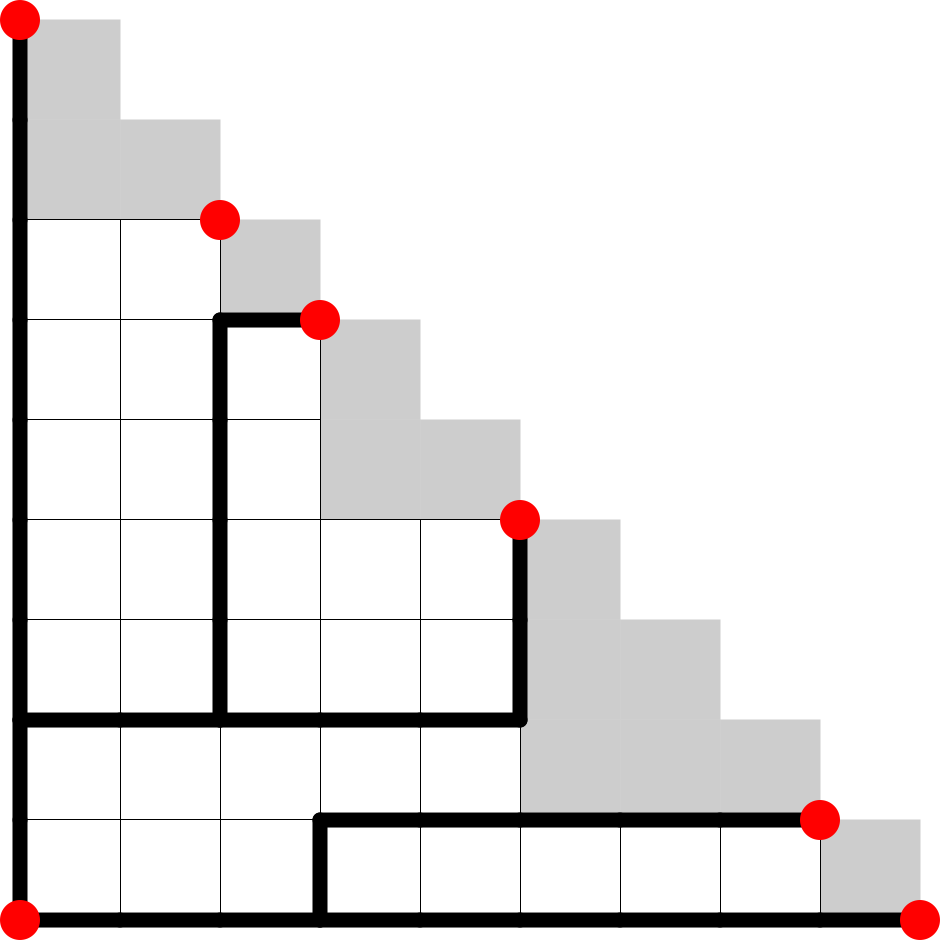}
\qquad
\includegraphics[scale=0.10]{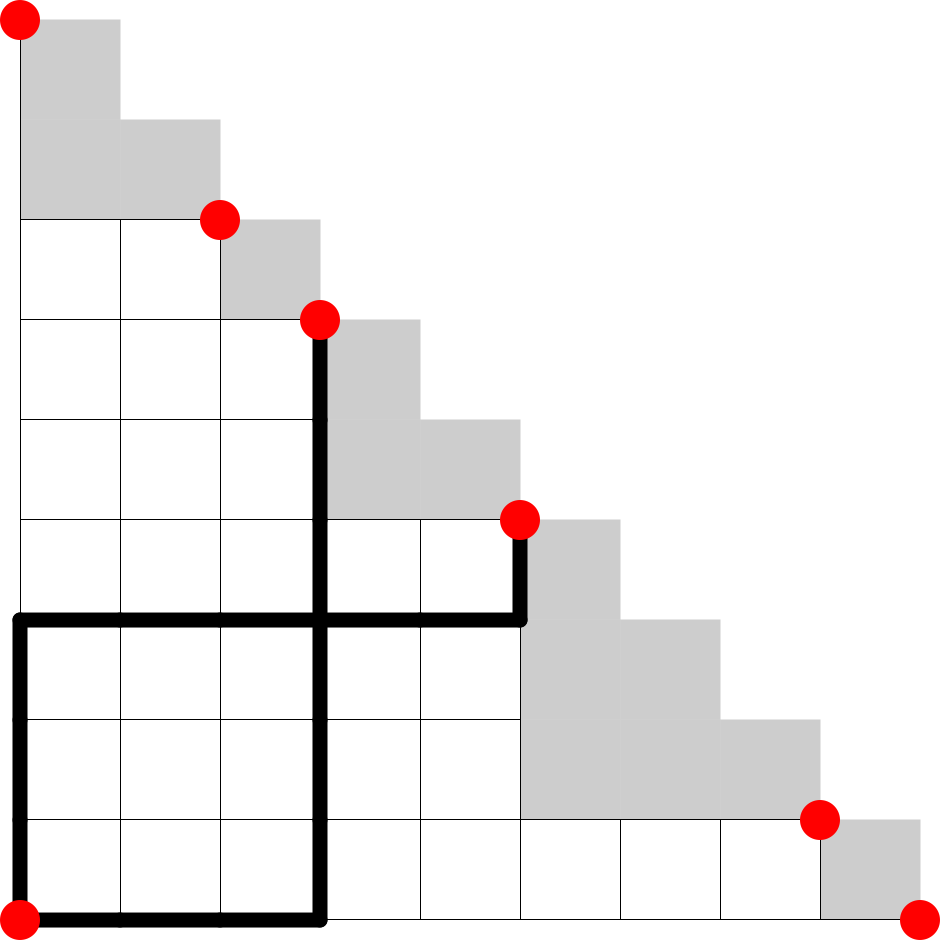}
\qquad
\includegraphics[scale=0.10]{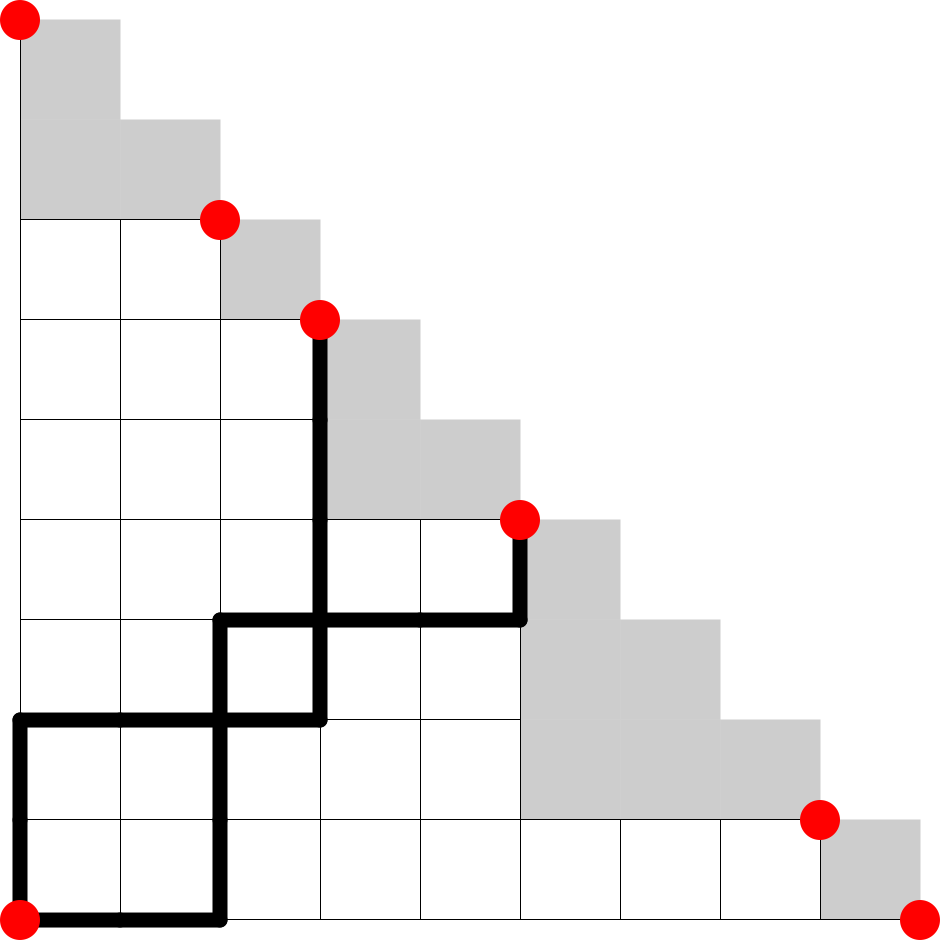}
\caption{Left: 5 non-crossing paths. Middle and Right: 2 Crossing paths}
\label{fig:intersecting-paths}
\end{figure}

\begin{figure} [htp]
\includegraphics[scale = 0.10]{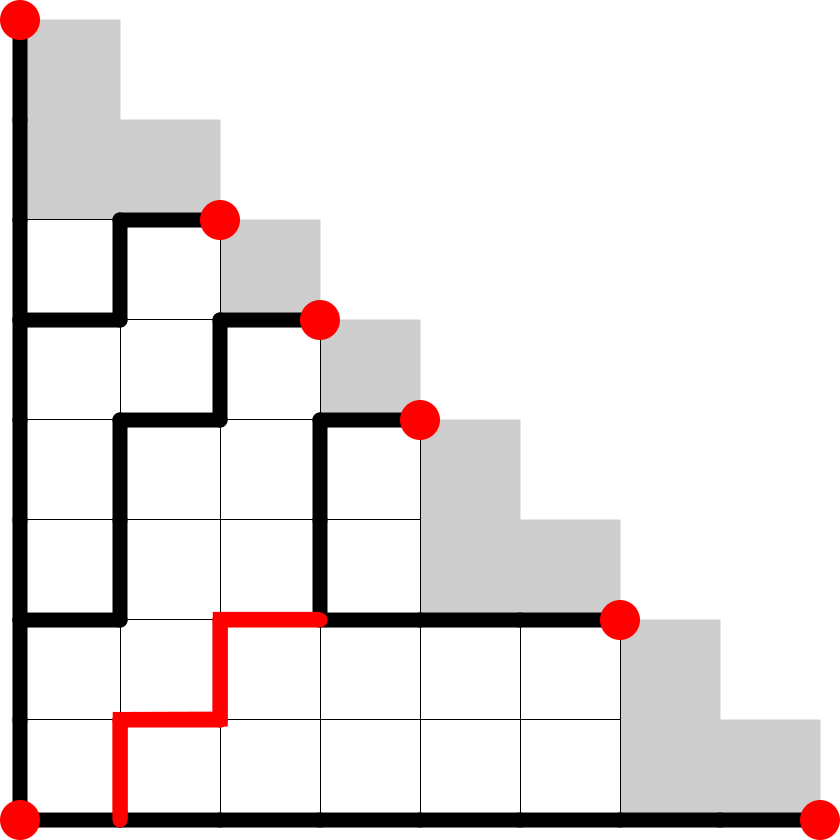}
\qquad
\includegraphics[scale = 0.10]{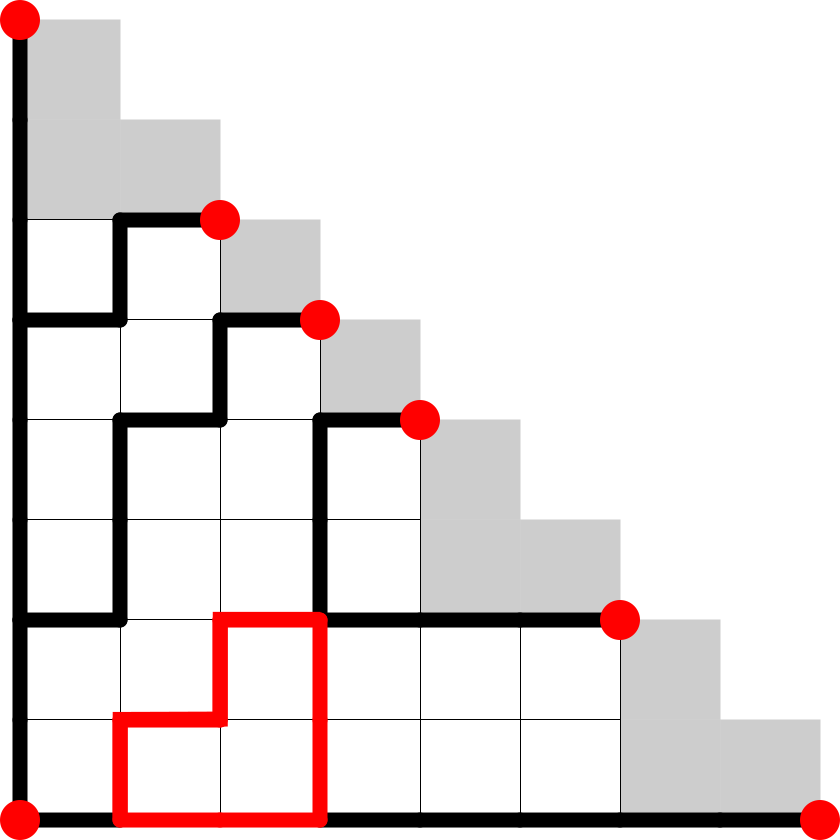}
\qquad
\includegraphics[scale = 0.10]{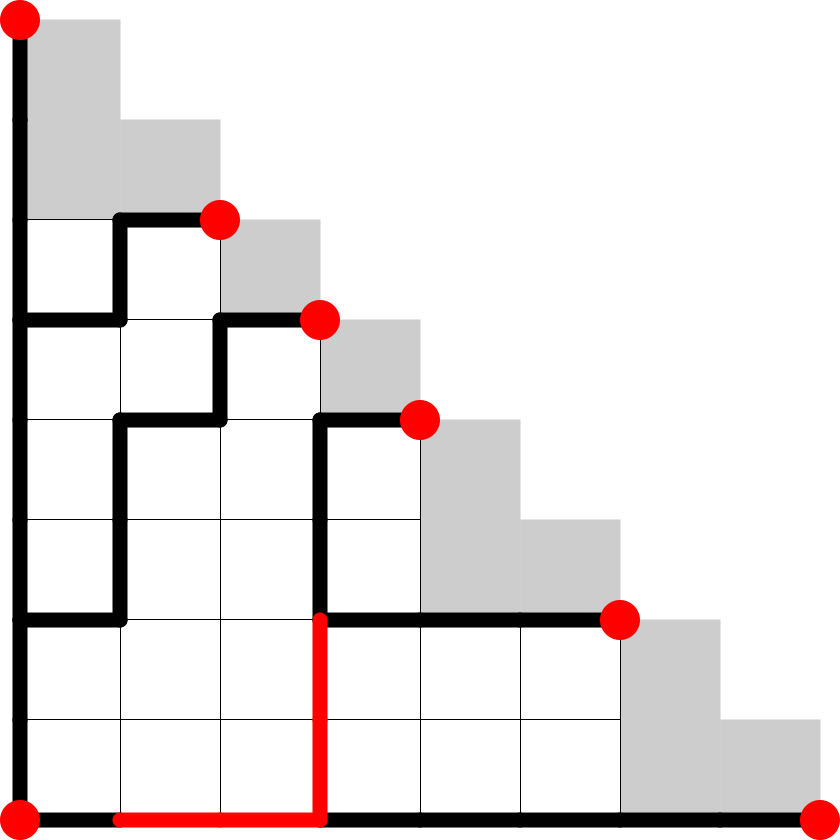}
\caption{Two vertices and the edge between them.}
\label{fig:adj-vertices}
\end{figure}

\begin{lem}\label{lem:diamUB}
Any two vertices $v$ and $w$ of $\GT_\lam$ are separated by at most $2m - 2- \delta_{1, a_1} - \delta_{1, a_m}$ edges.
\end{lem}
\begin{proof}
We give an algorithm to find a path between $v$ and $w$ of length at most $2m- \delta_{1, a_1} - \delta_{1, a_m}$. Assume that in the ladder diagram representation, vertex $v$ corresponds to noncrossing paths $v_1,\ldots,v_{m-1}$ where $v_j$ connects the origin $(0,0)$ to terminal vertex $t_j=(s_j,n-s_j)$ where $s_j=\sum_{i=0}^{j}a_i$. Similarly denote the noncrossing paths corresponding to $w$ as $w_1,\ldots,w_{m-1}$. 

Essentially, we want to change $v_1$ to $w_1$, $v_2$ to $w_2$, $\ldots$, $v_{m-1}$ to $w_{m_1}$, making sure that the $m-1$ paths we have are always noncrossing, and that the common refinement before and after changing some paths has exactly one bounded region. This ensures we are always traveling along edges in $\GT_\lam$. Note that the two paths ending at $t_0$ and $t_m$ are the same for every ladder diagram so we ignore them here.

\textbf{Phase 1:} If $a_1=1$, then $v_1,w_1$ are paths that go from $(0,0)$ to $(1,n-1)$. Therefore, there exists a unique index $r_v$ such that path $v_1$ passes through both $(0,r_v)$ and $(1,r_v)$. In other words, $r_v$ is the vertical index for $v_1$ to go from column 0 to column 1. Similarly we can define $r_w$. WLOG, assume that $r_v\geq r_w$. Because of this inequality, we know that path $w_1$ is contained inside of $v_1$ and therefore, the ladder diagram consisting of $w_1,v_1,v_2,\ldots,v_{m-1}$ has exactly 1 bounded region and it is thus an edge $e$ of the GT polytope containing $v$. Let $v'=(v_1',\ldots,v_{m-1}')$ be the other side of this edge (one side is $v$). Notice that the inner edge $\{(0,r_w),(1,r_w)\}$ is in the ladder diagram of $e$ but not in the ladder diagram of $v'$. Therefore, $\{(0,r_w),(1,r_w)\}$ must be in the ladder diagram of $v'$, which means $v_1'=w_1$, since $a_1=1$. Similarly, we can use one move to make $v_{m-1}$ and $w_{m-1}$ equal if $a_{m}=1$.

\begin{figure}[htp]
\includegraphics[scale=0.10]{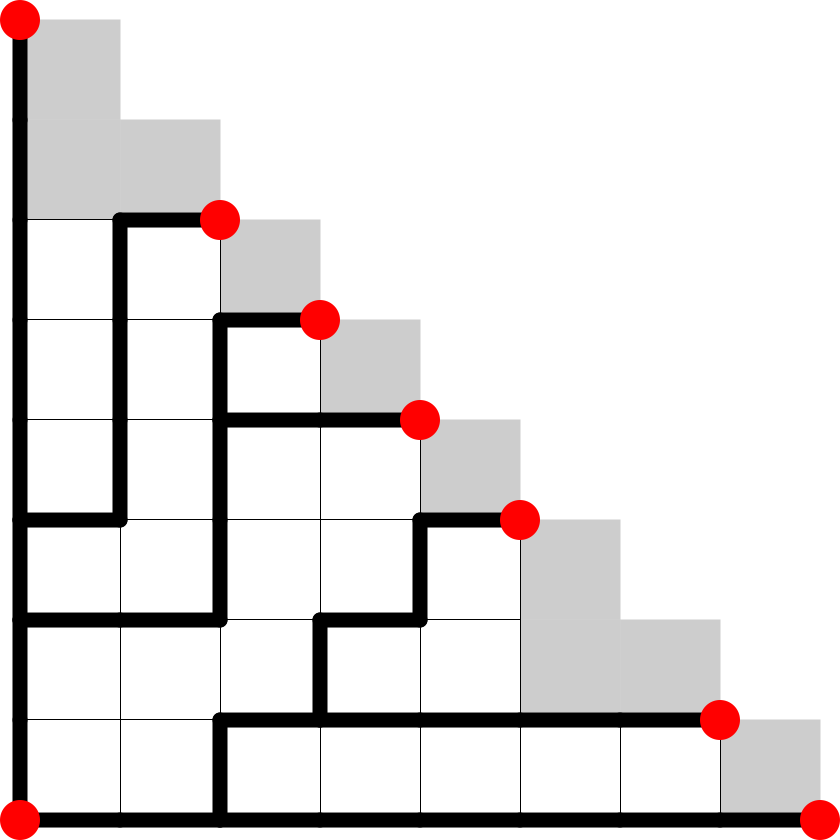}
\includegraphics[scale=0.10]{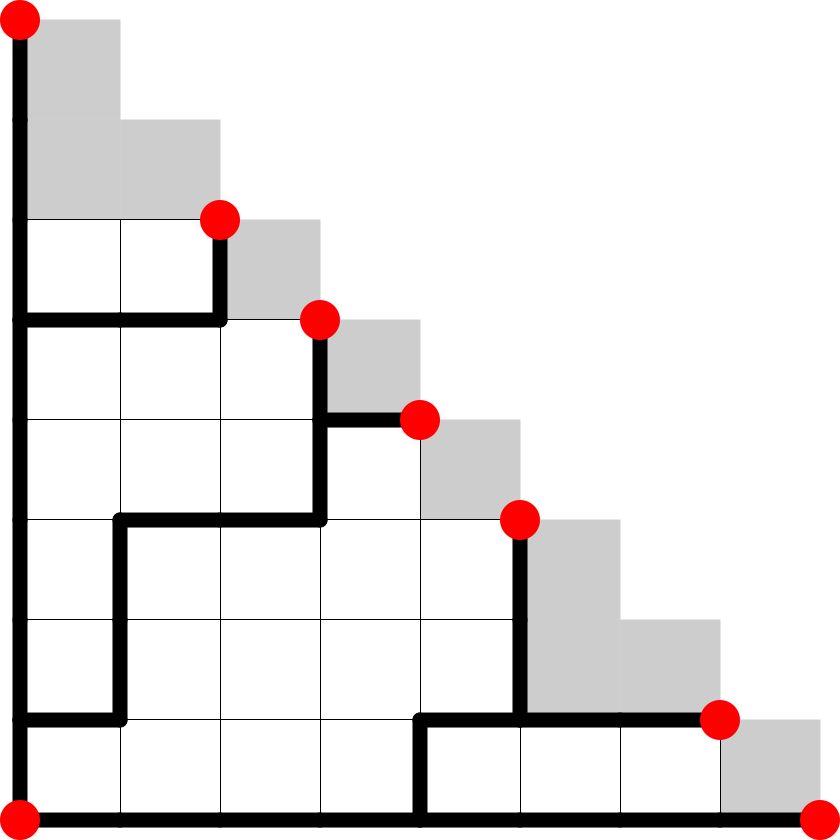}
\qquad\qquad
\includegraphics[scale=0.10]{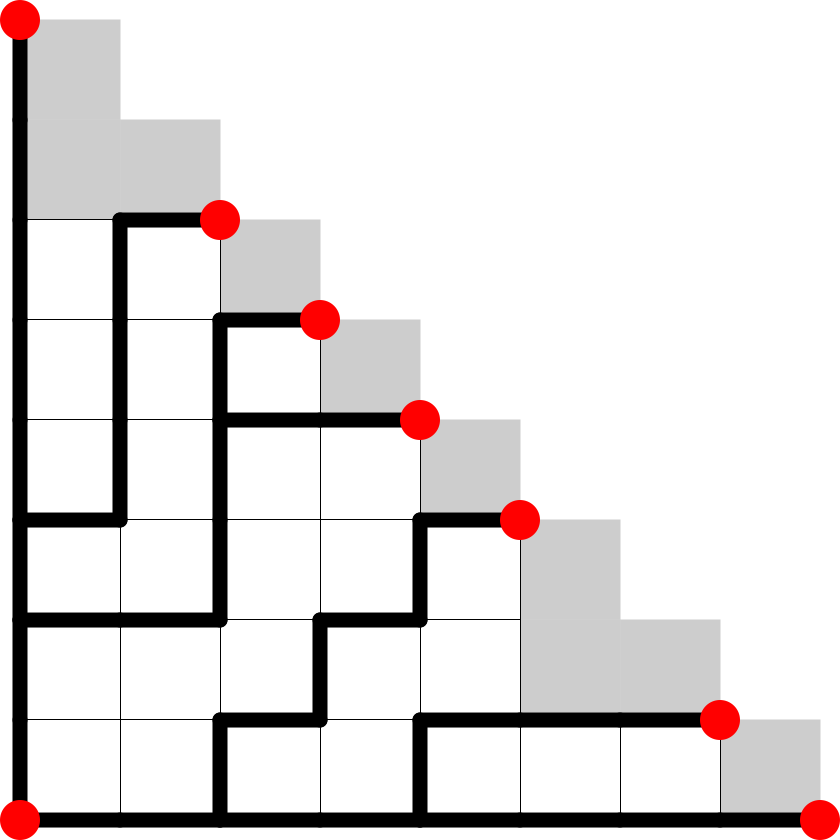}
\includegraphics[scale=0.10]{diamex10.png}
\caption{Phase 1 of the algorithm for Lemma \ref{lem:diamUB}. From left to right: the ladder diagram for $v$, $w$, $v'$, $w'(=w)$.}
\label{fig:diamalg1}
\end{figure}
\textbf{Phase 2:} Now we describe an algorithm that takes $v'$ to some vertex $u$ in at most $m-1-\delta_{1,a_1}-\delta_{1,a_m}$ steps. The algorithm works as follows: for each $i=1+\delta_{1,a_1},\ldots,m-1-\delta_{1,a_m}$, change path $v_i$ so that it starts at terminal vertex $t_i$, goes horizontally to the left until it meets and merges with path $v_{i-1}$. First, the ladder diagram after this change is clearly a vertex. Also, if we take the common refinement of the two ladder diagrams before and after the change, or equivalently, start with the old ladder diagram and add a new path $v_i'$ described above, then this new path simply cuts the tile bounded by $v_{i-1}$ and $v_i$ into two parts and thus there exists an edge between these two vertices. Figure \ref{fig:diamalg2} shows an example of this algorithm.

\begin{figure}[htp]
\includegraphics[scale=0.10]{diamex02.png}
\includegraphics[scale=0.10]{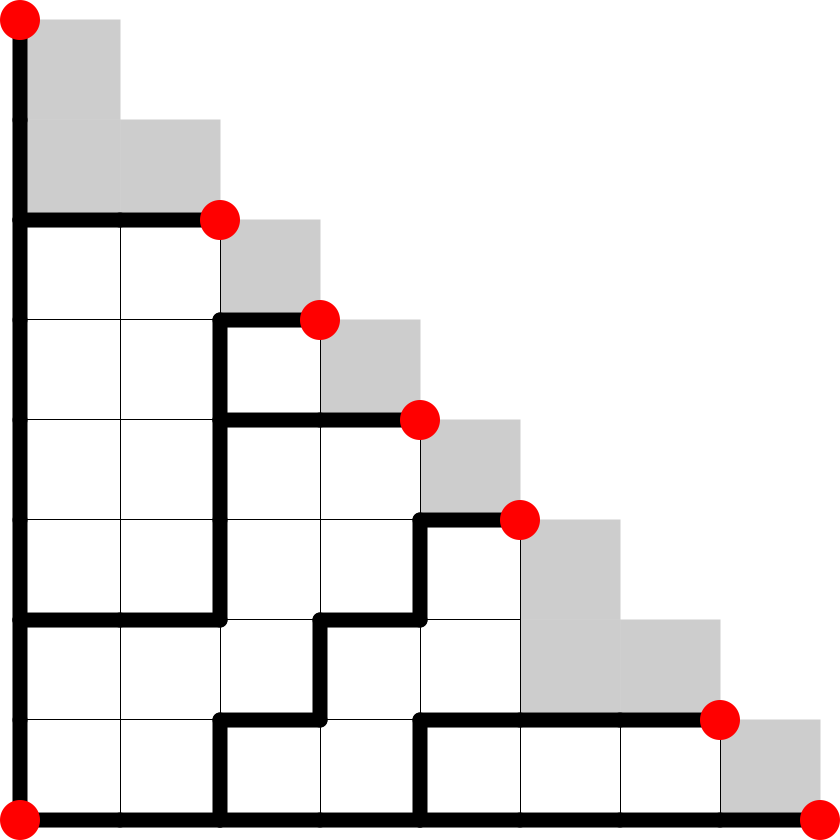}
\includegraphics[scale=0.10]{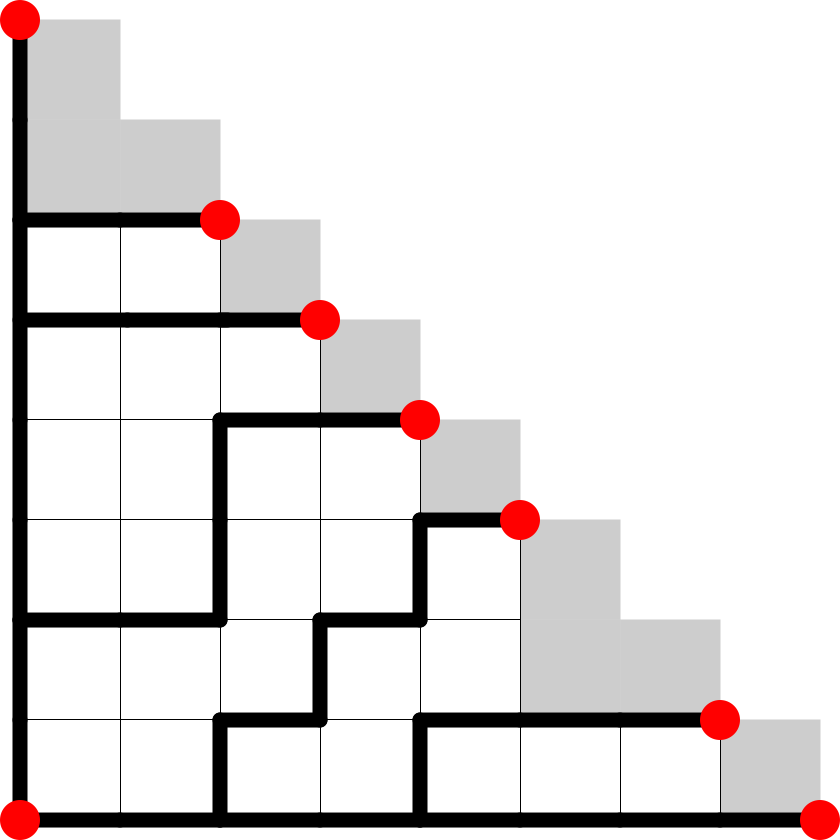}
\includegraphics[scale=0.10]{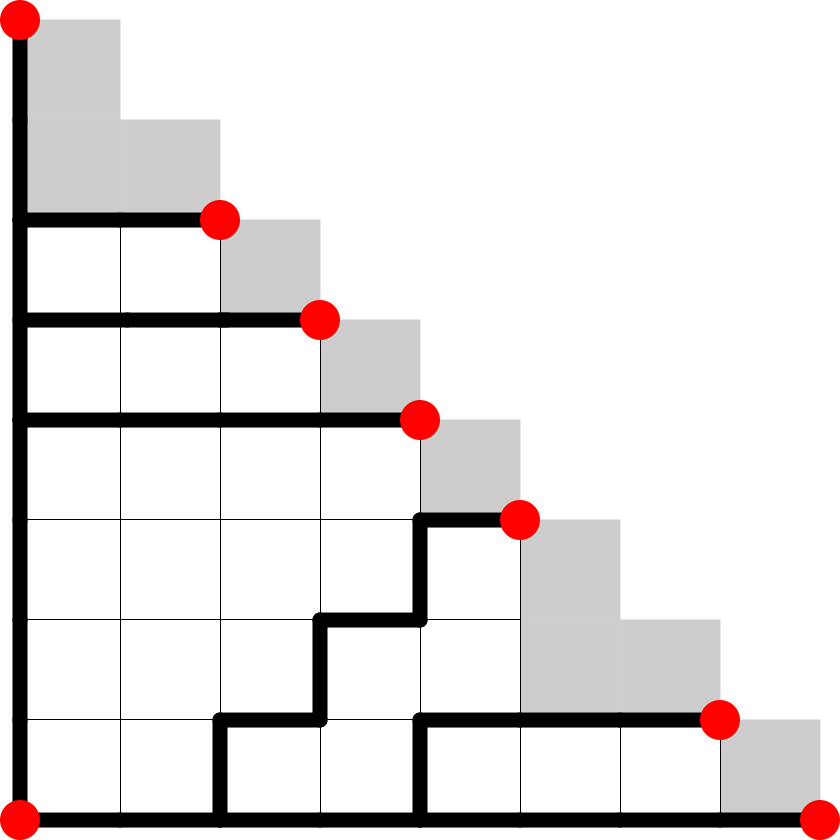}
\includegraphics[scale=0.10]{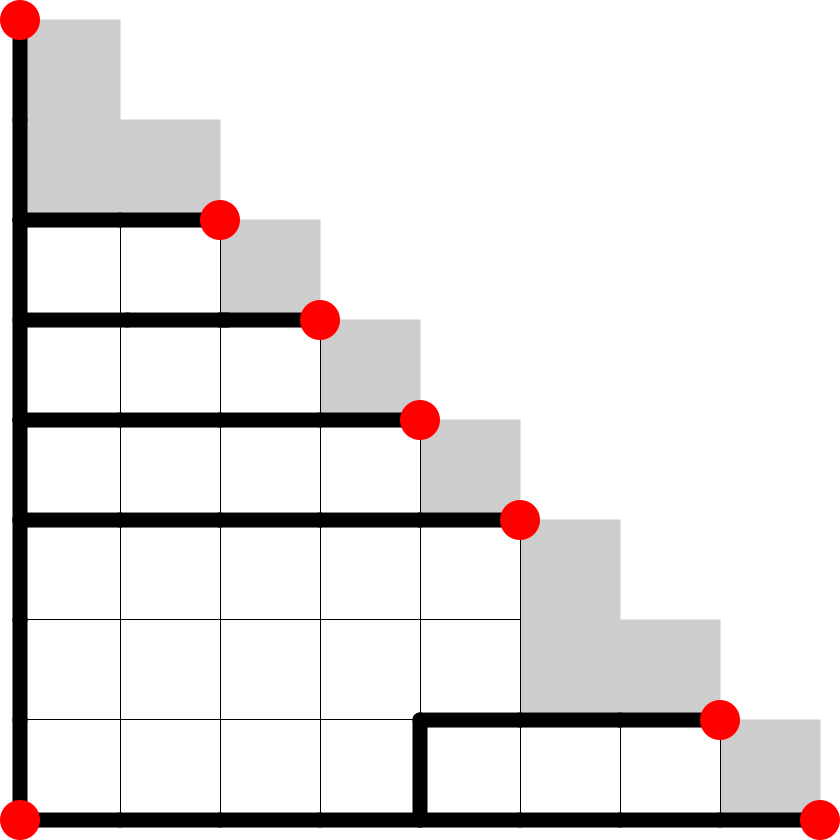}

\

\includegraphics[scale=0.10]{diamex10.png}
\includegraphics[scale=0.10]{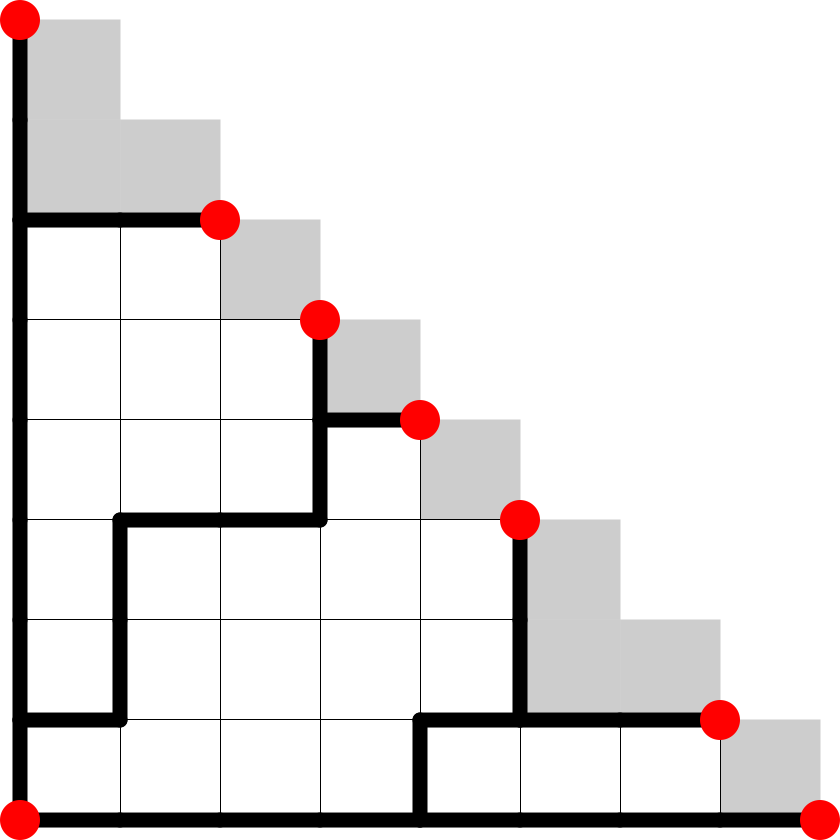}
\includegraphics[scale=0.10]{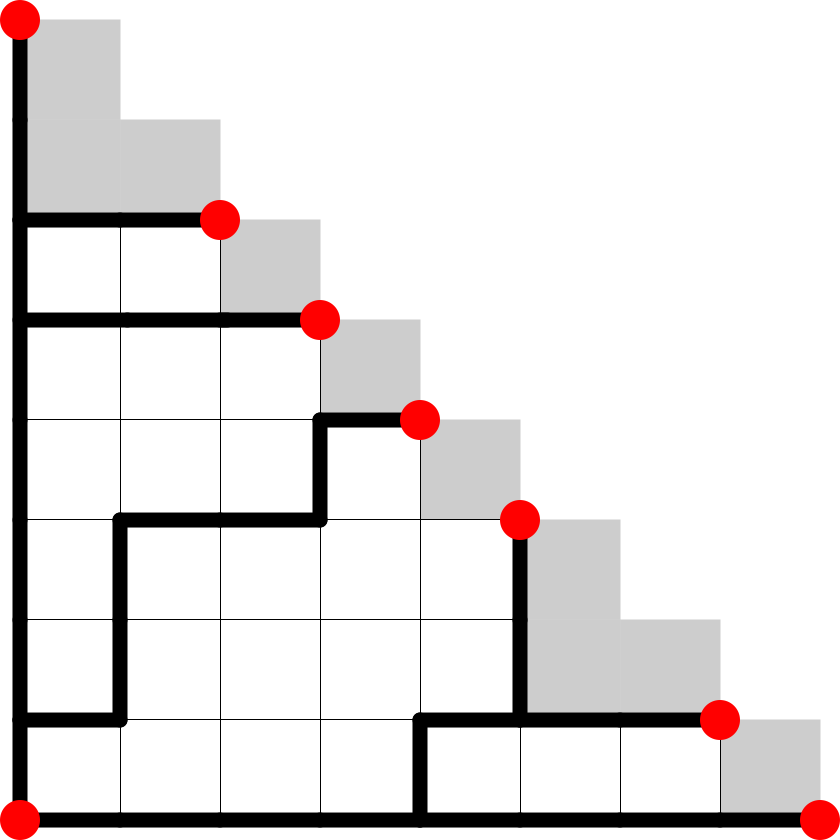}
\includegraphics[scale=0.10]{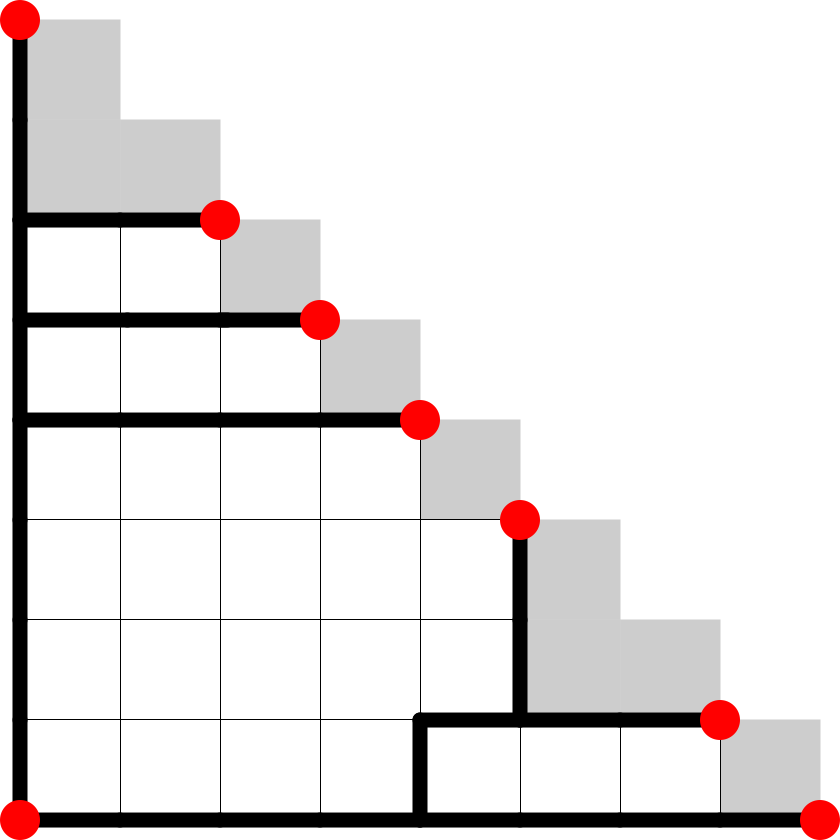}
\includegraphics[scale=0.10]{diamex06.png}
\caption{Phase 2 of the algorithm for Lemma \ref{lem:diamUB}. The first line shows steps that move from $v'$ to $u$; the second line shows steps that move from $w'$ to $u$.}
\label{fig:diamalg2}
\end{figure}

Similarly, we can apply the same algorithm to $w'$ to get to the same vertex $u$ in the same number of steps. The total number of steps is at most
$$(\delta_{1,a_1}+\delta_{1,a_m})+2(m-1-\delta_{1,a_1}-\delta_{1,a_m})=2m-2-\delta_{1,a_1}-\delta_{1,a_m}.$$
\end{proof}

\begin{lem}\label{lem:diamLB}
There exist two vertices separated by at least $2m -2-\delta_{1,a_1}-\delta_{1,a_m}$ edges.
\end{lem}
In order to prove this lemma, we need to first construct vertices in terms of ladder diagrams. 

\begin{defn}[Zigzag lattice path]\label{def:zigzag}
We construct two vertices $z_h$ and $z_v$ that will be used in the proof for Lemma \ref{lem:diamLB}. Let $\lambda=(1^{a_1},\ldots,m^{a_m})$. If $a_1>1$, meaning that $(1,n-1)$ is not a terminal vertex, we call $(1,n-1)$ a \emph{virtual terminal vertex}. Similarly, if $(n-1,1)$ is not an actual terminal vertex, meaning that $a_m>1$, we call it a virtual terminal vertex.

We will consider the ladder diagram for a vertex of $\GT_{\lambda}$ as $m-1$ southwest lattice paths from terminal vertices to the origin. For $j=1,\ldots,m-1$, define a horizontal zigzag path, $h_j$, to be the path that starts at terminal vertex $t_j$, goes horizontally left until reaching a column where there exists a terminal vertex or a virtual terminal vertex on it, then goes vertically down until reaching a row where there exists a terminal vertex or a virtual terminal vertex, and so on and so forth until the path reaches column 0 or row 0. Similarly define a vertical zigzag path, $v_j$, with the only difference that it will start vertically instead of horizontally. Finally, let $z_h$ be the vertex of $\GT_{\lambda}$ represented by the ladder diagram $(h_1,\ldots,h_{m-1})$ and let $z_v$ be the vertex of $\GT_{\lambda}$ represented by $(v_1,\ldots,v_{m-1})$. Figure~\ref{fig:zigzag} shows the construction of an example.

\begin{figure}[htp]
\includegraphics[scale=0.12]{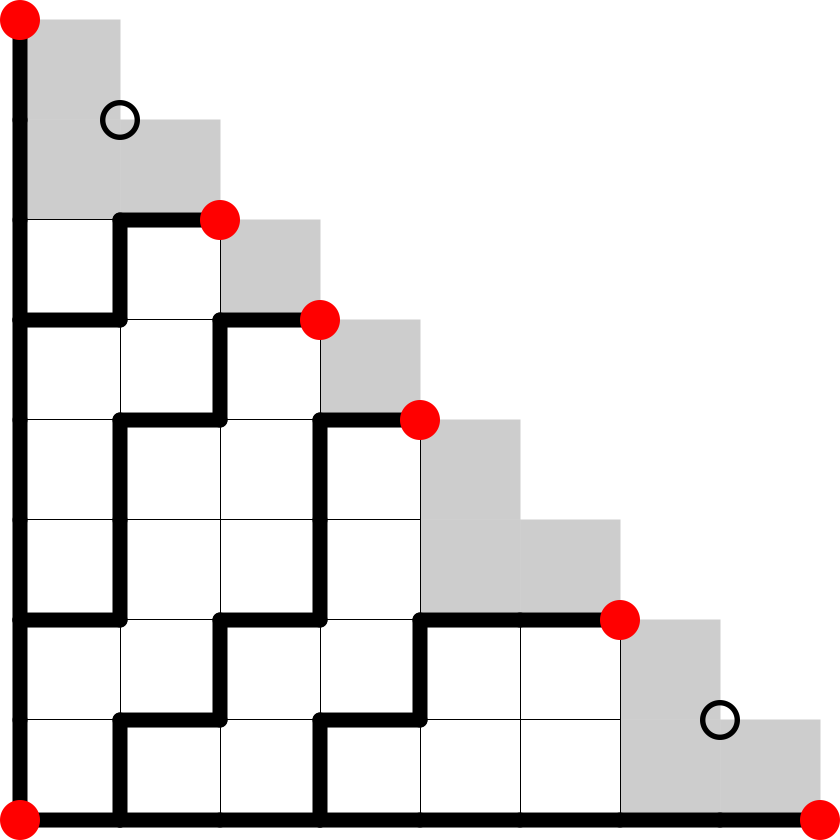}
\qquad
\includegraphics[scale=0.12]{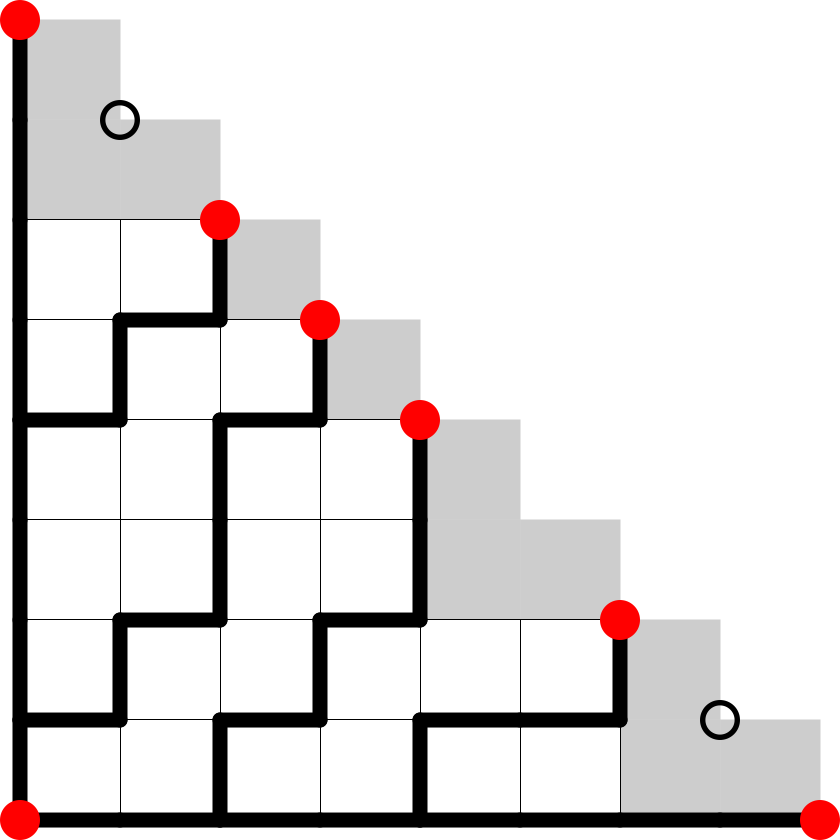}
\caption{Vertices $z_h$ (left) and $z_v$ (right) of $GT_{\lambda}$ with $\lambda=(1^2,2^1,3^1,4^2,5^2)$. Virtual terminal vertices are labeled as open circles.}
\label{fig:zigzag}
\end{figure}
\end{defn}

\begin{proof}[Proof of Lemma~\ref{lem:diamLB}]
We will first consider the case where $a_1,a_m\geq2$ so that the idea of the proof can be shown clearly. Afterward, we will deal with the details coming from either of them being 1.

Consider the vertices given in Definition~\ref{def:zigzag}. Assume that there is a sequence of vertices $z_h=y_0, y_1,\ldots,y_{\ell}=z_v$ in $\GT_{\lambda}$ such that $y_k$ and $y_{k+1}$ are connected by an edge. Since we can uniquely represent each vertex in this sequence as a union of $m-1$ lattice paths from the original to each terminal vertex, each step $y_k \rightarrow y_{k+1}$ can be thought of as simply changing these paths, such that the union of $y_k$ and $y_{k+1}$ has at most 1 bounded region. For the sake of convenience, for each $i=1,\ldots,m-1$, let $p_i$ be the path that goes from the origin to the terminal vertex $i$ generically for all $y_k$'s. For $k=1,\ldots,\ell$, define $X_k=\{i:p_i\text{ changes as we go from }y_{k-1}\text{ to }y_{k}\}.$ Since in each step we can have at most 1 bounded region between $y_{k-1}$ and $y_k$, it is clear that $X_k$ is of the form $\{i,i+1,\ldots,j\}$ for some $i\leq j$. 

We will now show that $X_1,\ldots,X_{\ell}$ must satisfy the following conditions:
\begin{enumerate}
\item For each $s\in[m-1]$, $s$ appears in at least two of the sets $X_{i}$'s. In other words, for each $s\in[m-1]$, there exists $1\leq i<j\leq \ell$ such that $s\in X_i$, $s\in X_j$. 
\item For each $s\neq s'\in[m-1]$, the last time $s$ appears in any of the sets (which is the unique index $b_s$ such that $s\in X_{b_s}$ and $k\notin X_{i}$ for all $i>b_k$) is different from the last time that $k'$ appears.
\item If $X_{k}=\{i,i+1,\ldots,j\}$, then at least $j-i$ of $i,i+1,\ldots,j$ must appear in some $X_{k'}$ for $k'<k$
\item If $X_{k}=\{i,i+1,\ldots,j\}$ and it is the last time that $s$ appears, then each one of $\{i,i+1,\ldots,j\}\backslash \{s\}$ must appear in (possibly different) $X_{k'}$ for $k' < k$. 
\end{enumerate}

Notice that if the last time $s$ appears is in some $X_k$, then $p_s$ has already become $v_s$ (Definition \ref{def:zigzag}) in $y_{k}$, and stays the same for the rest of the steps. We will then explain these conditions in detail one by one.

\textbf{Condition (1)}. If for some $s\in[m-1]$, it appears in only one set $X_k$, then it means when we go from vertex $y_{k-1}$ to $y_k$, the path $p_s$ is changed from $v_s$ to $h_s$ in exactly one step. But by construction, superimposing $h_s$ and $v_s$ will create at least 2 bounded regions, instead of one. Therefore, $y_{k-1}$ and $y_k$ are not connected by an edge, a contradiction. Therefore, each $s\in[m-1]$ appears in at least two sets.

\textbf{Condition (2)}. If $X_k$ is the last time that both $s$ and $s'$ appears, then it means that from $y_{k-1}$ to $y_k$, paths $p_{s'}$ and $p_s$ are changed to $v_{s'}$ and $v_{s}$ simultaneously. However, $v_{s'}$ and $v_s$ do not have any intersection in the interior of $\I_n$. So if we want to go back from $y_k$ to $y_{k-1}$, we have to change $v_s'$ and $v_s$ simultaneously, which will create two bounded regions when we superimpose $y_{k-1}$ and $y_{k}$, which is a contradiction.

\textbf{Condition (3)}. Since initially, $h_i$ and $h_j$ do not have any intersection in the interior of $\I_n$, in order to change multiple paths at the same time, we need to merge these paths first. Specifically, if we want to change paths $i,{i+1},\ldots,j$ simultaneously, we need to modify at least all but one of these paths to join them together.

\textbf{Condition (4)}. This condition is crucial and it justifies our choices for $z_h$ and $z_v$. Assume that $X_k=\{i,i+1,\ldots,j\}$ is the last time that $s$ appears, meaning that when we go from vertex $y_{k-1}$ to $y_k$, we change $p_s$ to $v_s$. If there exists $s'\neq s\in X_k$, such that path $s'$ hasn't changed before, we know $p_{s'}=h_{s'}$. Notice that since we change paths $i,i+1,\ldots,j$ simultaneously, all of these paths before and after the change will have an interior vertex in $\I_n$ in common. Therefore, $v_s$ and $h_{s'}$ must have a common interior vertex. Since $s\neq s'$, we must have $s'=s+1$. As we have assumed $a_1,a_m\geq2$, superimposing the ladder diagram of $y_{k-1}$ and $y_k$ will create at least $2$ bounded regions, because of the definition of $v_s$ and $h_{s+1}$. Therefore, we have a contradiction and thus all $s'\neq s$ must have already appeared at least once.

As we have proved that $X_1,\ldots,X_{\ell}$ satisfy all these conditions, we will jump out of the general setting of GT polytopes and look at any sequence of sets $X_1,\ldots,X_{\ell}$ that satisfies all these four conditions. We will show that any such sequence will have length $\ell\geq 2m-2$.

To do this, for $i=\ell,\ldots,1$, we look at the first set $X_i$ that is not a singleton. Say that it is $X_k$. If $X_k$ is the last time that some $s\in[m-1]$ appears, then we claim that changing $X_k$ to $\{s\}$ will still satisfy all four conditions. According to Condition (4), for each $s'\neq s$ that is in $X_k$, $s'$ must have appeared before. According to Condition (2), for each $s'\neq s$, this is not the last time that $s'$ appears so $s'$ will appear sometime later. Therefore, condition (1) still holds after changing $X_k$ to $\{s\}$. Condition (2) also holds because this change does not modify the indices of the sets where each $s'\in[m-1]$ appears last. Condition (3) and (4) hold trivially because we have less non-singleton sets to worry about. Another case is that $X_k$ is not the last time that any of the paths appears last. According to condition (3), there exists $s\in X_k$ such that each one of $X_k\backslash\{s\}$ has appeared before. Similarly, we claim that all these four conditions will hold after changing $X_k$ to $\{s\}$. For each one of $s'\in X_k\backslash\{s\}$, as it appears before and $X_k$ is not the last time that it appears, we know that $s'$ will appear at least twice even after this change. The number of appearances of $s$ does not change. Therefore, condition (1) is satisfied. Condition (2) holds because each one of $s'\in X_k$ will appear sometime later. Condition (3) and (4) hold trivially because similarly we have less non-singleton sets to worry about.

Continuing this procedure inductively, we will eventually end up with a sequence of sets $Y_1,\ldots,Y_{\ell}$ where each one is a singleton. As Condition (1) still holds, we conclude that $\ell\geq 2m-2$ as desired.

Now we consider the cases where $a_1$ and $a_m$ may be 1. We claim that if $a_1=1$, then $\{1\}$ must appear as a singleton as one of the terms in $X_1,\ldots,X_{\ell}$ and if $a_m=1$, then $\{m-1\}$ must appear as a singleton as one of the terms in the sequence. Notice the reasons are slightly different. If $a_1=1$, path $h_1$ has only 1 interior edge, namely $(0,n-1)-(1,n-1)$. In order for it to merge with other paths, it must change on its own first, meaning that $\{1\}$ will appear, because $h_1$ and $p_2,\ldots,p_{m-1}$ will never have any interior intersection. If it never merges with other paths, then it must appear as $\{1\}$ at some point so that $p_1$ can be actually changed from $h_1$ to $v_1$. If $a_m=1$, at some point, path $p_{m-1}$ must be changed to $v_{m-1}$, which has only 1 interior edge. This change is recorded as $\{m-1\}$ because $p_1,\ldots,p_{m-2}$ cannot have any interior intersection with $v_{m-1}$.

Therefore, if $a_1=1$ (or/and $a_m=1$), in the sequence $X_1,\ldots,X_{\ell}$, we can take out the singleton $\{1\}$ (or/and $\{m-1\}$) and delete all other $1$'s (or/and $m-1$'s) in the sets. The remaining sequence will satisfy the four conditions mentioned above. We have taken out at least $\delta_{1,a_1}+\delta_{1,a_m}$ singletons of $\{1\}$ and $\{m-1\}$ and there are $m-1-\delta_{1,a_1}-\delta_{1,a_m}$ paths remaining. So by the same argument, the total length of the sequence
$$\ell\geq (\delta_{1,a_1}+\delta_{1,a_m})+2(m-1-\delta_{1,a_1}-\delta_{1,a_m})=2m-2-\delta_{1,a_1}-\delta_{1,a_m}$$
as desired.
\end{proof}

Lemma~\ref{lem:diamUB} and Lemma~\ref{lem:diamLB} together prove our first result, Theorem~\ref{thm:diameter}.

\begin{thm1.1}[Diameter of 1-skeleton]
$\diam(\GT_\lambda)$ = $2m-2-\delta_{1, a_1}-\delta_{1, a_m}$.
\end{thm1.1}

\section{Combinatorial Automorphisms}\label{sec:comb-automorphisms}

In this section, we completely describe the combinatorial automorphism group of $\GT_\lambda$ which is denoted by $\Aut(\GT_\lambda)$. A combinatorial automorphism of a polytope is a permutation of its faces that preserves inclusion. Such a permutation can be thought of as an automorphism of the face lattice. For the sake of brevity, we will refer to combinatorial automorphisms simply as automorphisms. In Section~\ref{subsec:sym} we present the generators of the automorphism group and the relations between them. In Section~\ref{subsec:facets} we discuss \emph{facet chains} which are then used in Section~\ref{subsec:proving-autom-gp} to show that our generators form the entire automorphism group. As explained in Remark~\ref{rmk:ai}, it suffices to consider $\lambda = (1^{a_1}, \ldots, m^{a_m})$ where $a_1, \ldots ,a_m \in \mathbb{Z}_{>0}$. 

\subsection{Automorphisms}\label{subsec:sym} 
We begin by identifying the generators of $\Aut(\GT_\lambda)$. By Theorem~\ref{thm:ladders}, to show that a map is an automorphism of $\GT_{\lambda}$, it suffices to show that the map is an automorphism of $\mathcal{F}(\gamlam)$. 

\begin{prop}[The Corner Symmetry]
For any $\lambda$, there is an order 2 automorphism $\mu$ of $\mathcal{F}(\gamlam)$ which acts by exchanging two pairs of edges near $(0,0)$: if a ladder diagram contains $\{(0, 0), (1, 0)\}$ and $\{(1, 0), (1, 1)\}$, then $\mu$ acts on this diagram by replacing these edges with $\{(0, 0), (0, 1)\}$ and $\{(0, 1), (1, 1)\}$ and vice versa. 
\end{prop}
\begin{proof}
Let $\mu:\R^{\frac{n(n+1)}{2}} \rightarrow \R^{\frac{n(n+1)}{2}}$ be the linear map that sends $x_{n, 1} \mapsto x_{n, 2} + x_{n-1, 1} - x_{n, 1}$ and acts as the identity on all other $x_{i, j}$. 
Since $x_{n-1, 1} \leq x_{n-1, 1} + x_{n, 2} - x_{n, 1} \leq x_{n, 2}$, all necessary inequalities are satisfied so $\mu(\GT_\lambda) \subset \GT_\lambda$. Since $\mu^2 = Id$, the previous inclusion is actually an equality. Combining this fact with the fact that $\mu$ is an affine transformation, we see $\mu$ induces a combinatorial automorphism on $\GT_\lambda$, which we shall abuse notation and call $\mu$. Examining what $\mu$ does to equalities shows it acts in the described manner on ladder diagrams.

\end{proof}

\begin{figure}[htp]\label{fig:cornersym}
\includegraphics[scale = 0.12]{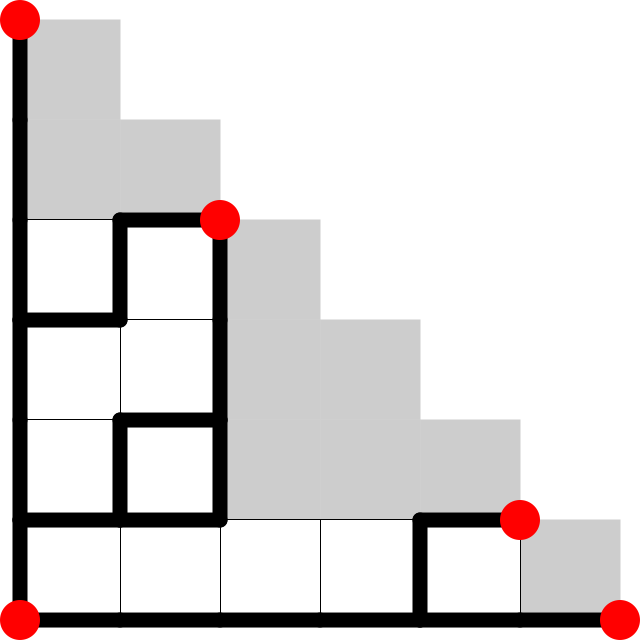}
\qquad
\includegraphics[scale = 0.12]{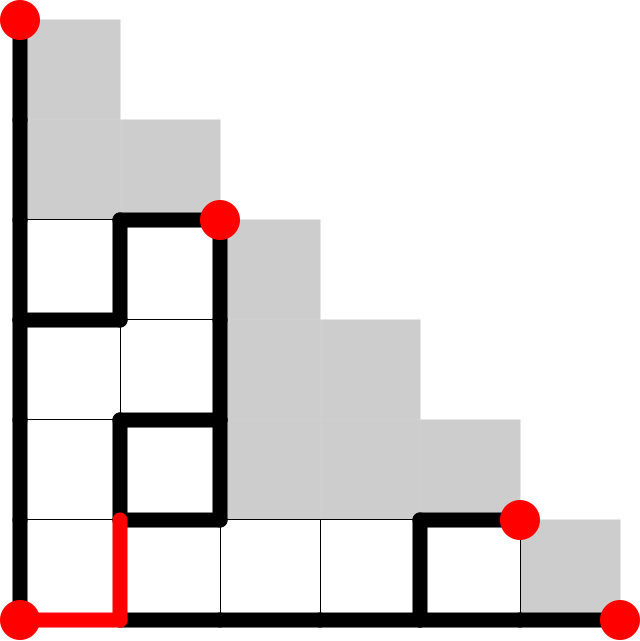}
\caption{Action of $\mu$ on a ladder diagram.}
\end{figure}

The next map is very similar to the Corner Symmetry $\mu$ except it occurs at the $k$th terminal vertex instead of at the origin. The argument is very similar to Prop 4.1, so we will omit some details.

\begin{prop}[The $k$-Corner Symmetry]
Let the $k$th terminal vertex be $t_k = (n-i, i)$ and suppose that $a_k, a_{k+1} \geq 2$. Then there is an order 2 automorphism $\mu_k$ which acts by exchanging two pairs of edges near $t_k$: if a ladder diagram contains $\{(n-i, i),(n-i, i-1)\}$ and $\{(n-i, i-1), (n-i-1, i-1)\}$, then $\mu_k$ acts on this diagram by replacing these edges with $\{(n-i, i),(i-1, i)\}$ and $\{(n-i-1, i), (n-i-1, i-1)\}$.
\end{prop}
\begin{proof}
Let $x_{i', j'}$ be the coordinate in Figure~\ref{fig:ineqdiagram} which is immediately adjacent to the values $\lambda_k$ and $\lambda_{k+1}$. Let $\mu_k: \R^{\frac{n(n+1)}{2}} \rightarrow \R^{\frac{n(n+1)}{2}}$ be the linear map that sends $x_{i', j'} \mapsto x_{i', j'-1} + x_{i'+1, j'} - x_{i', j'}$ and acts as the identity on all other $x_{i, j}$. Since $a_k, a_{k+1} \geq 2$, we have 
$$\lambda_k \leq x_{i', j'-1} \leq x_{i', j'-1} +x_{i'+1, j'} - x_{i', j'} \leq x_{i'+1, j'} \leq \lambda_{k+1}$$
so all 4 constraints on the coordinate $x_{i', j'}$ are satisfied. Since $\tau^2 = 1$, the same argument as in Prop. 4.1 gives an automorphism $\tau$ with the desired properties. 

\end{proof}

\begin{figure}[htp]\label{fig:kcornersym}
\includegraphics[scale = 0.12]{sym-mus-sigma-before.png}
\qquad
\includegraphics[scale = 0.12]{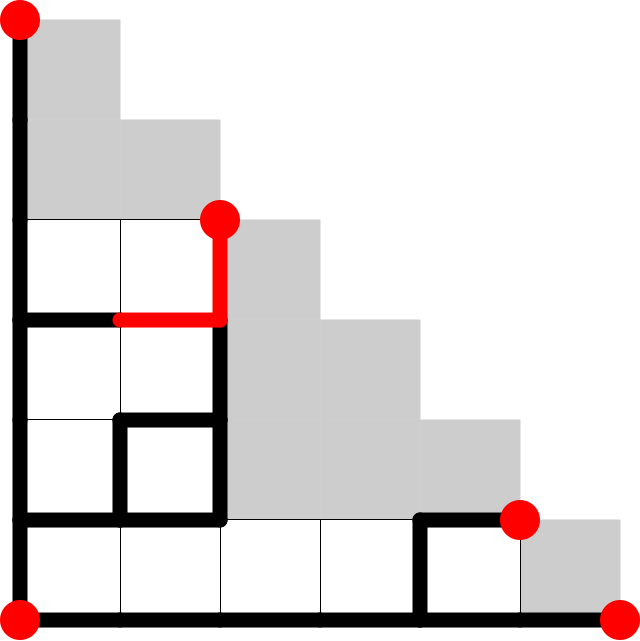}
\caption{Action of $\mu_k$ on a ladder diagram.}
\end{figure}

\begin{prop}[Symmetric Group Symmetry]
If $a_1 = 1$, then $S_{a_2} \subset \Aut(\GT_\lambda)$. Similarly if $a_m = 1$, then $S_{a_{m-1}} \subset \Aut(\GT_\lambda)$.
\end{prop}

\begin{proof}
Number the topmost $a_2$ horizontal edges in the first column of $\gamlam$ from $1$ through $a_2$. Let $\sigma$ be any element in $S_{a_2}$. Then $\sigma$ acts on each ladder diagram by permuting the horizontal edges in the first column.  Specifically, for a given ladder diagram, if there is a path passing through the $i$th horizontal edge and $\sigma(i) = j$, then $\sigma$ acts on this ladder diagram by changing this path to go through the $j$th horizontal edge. Note that there is only one way to change the vertical edges of the modified path to yield a valid ladder diagram. Again, $\sigma$'s action on $\cal{F}(\GT_\lam)$ has finite order so $\sigma$ is a bijection on the underlying set of $\cal{F}(\GT_\lam)$. Furthermore, $\sigma$ preserves inclusions of ladder diagrams so $\sigma$ is an automorphism of $\cal{F}(\GT_\lam)$.
\end{proof}

\begin{figure}[htp]\label{fig:sigmasym}
\includegraphics[scale = 0.12]{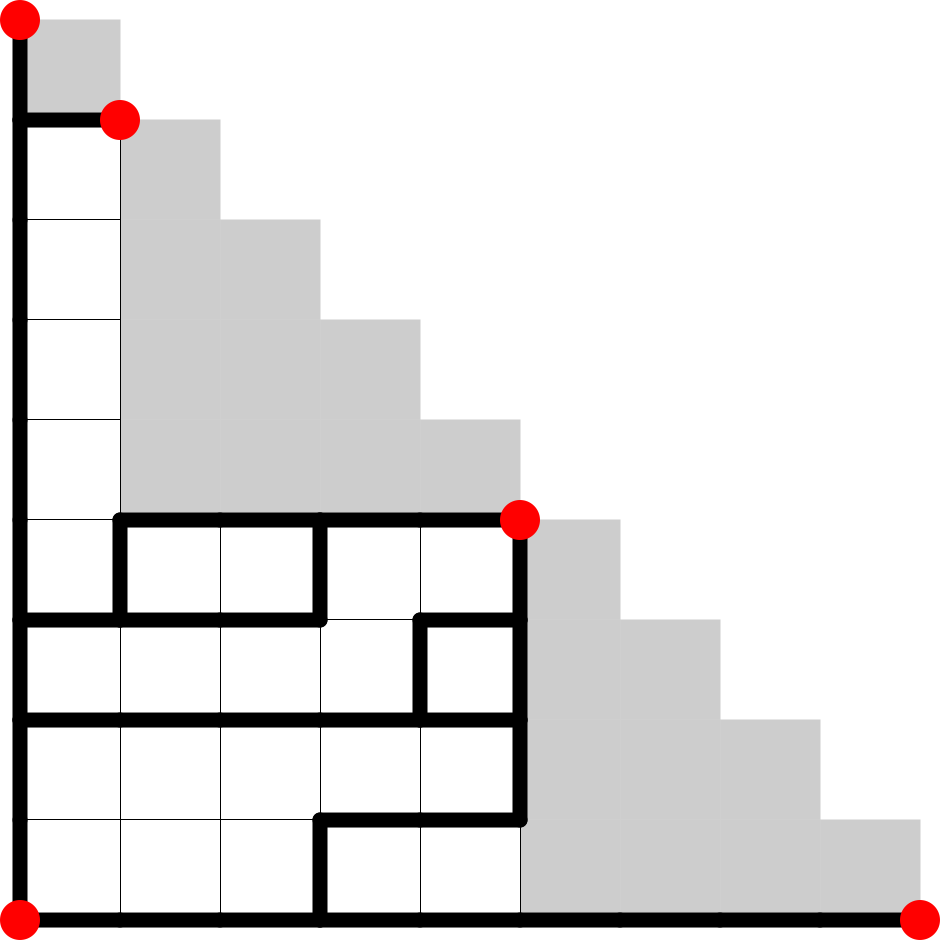}
\quad
\includegraphics[scale = 0.12]{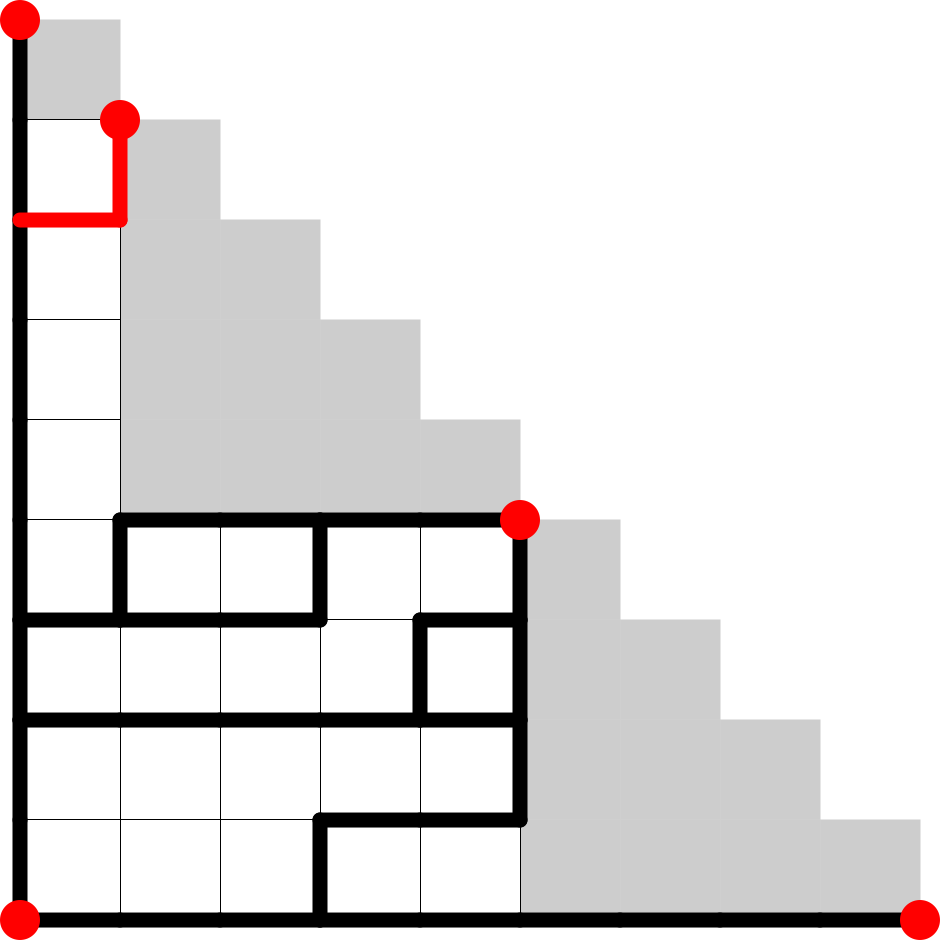}
\quad
\includegraphics[scale = 0.12]{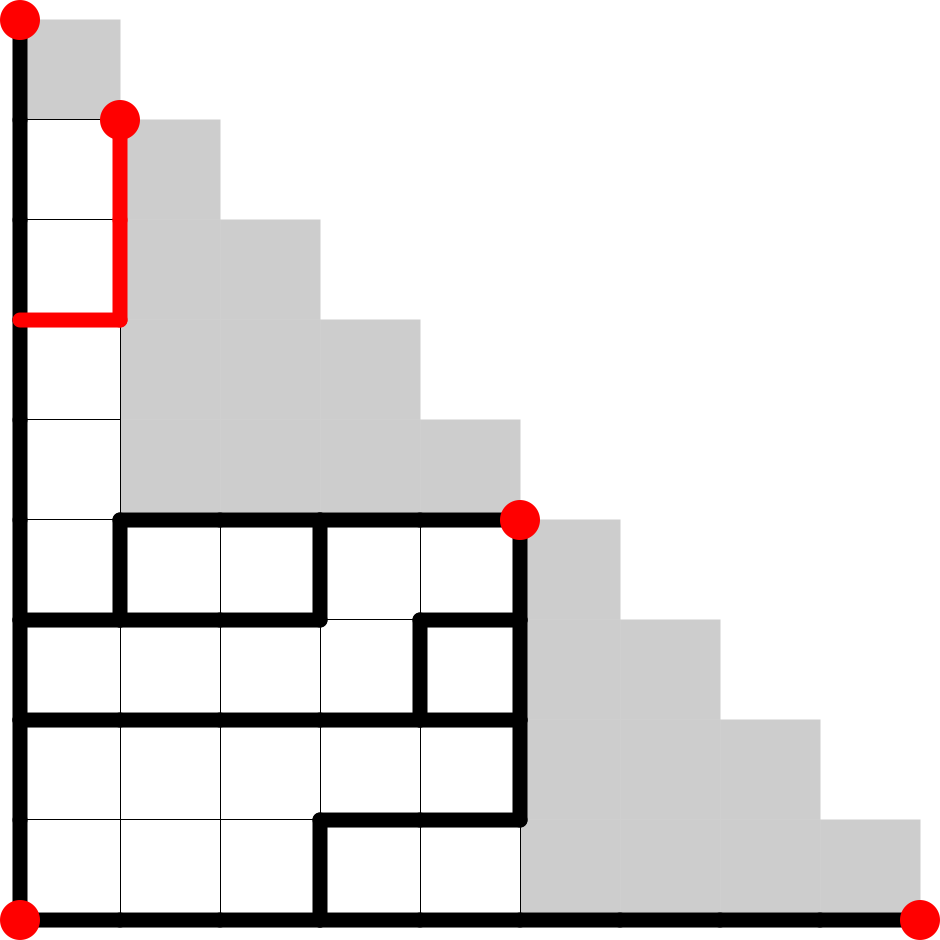}
\caption{Action of $(1 2 3) \in S_4$ on a ladder diagram.}
\end{figure}

\begin{prop}[The Flip Symmetry]
Suppose that $\lambda = \lambda'$. Then there is an order 2 automorphism $\rho$ of $\cal{F}(GT_\lambda)$ which reflects ladder diagrams over the line $y = x$.
\label{prop:flip-sym}
\end{prop}

\begin{proof}
Recall that if $\lam = (1^{a_1}, 2^{a_2}, \ldots, m^{a_m})$, then $\lam' := (1^{a_m}, 2^{a_{m-1}}, \ldots, m^{a_1})$. So if $\lam = \lam'$, we have $a_i = a_{m-i+1}$ for $1 \le i \le m$. In other words, the terminal vertices of $\gamlam$ are symmetric about the line $y=x$. Thus $\rho$ sends ladder diagrams to valid ladder diagrams. Since $\rho^2 = 1$ and $\rho$ preserves inclusions of ladder diagrams, $\rho$ is an automorphism of $\mathcal{F}(\gamlam)$.
\end{proof}

\begin{figure}[htp]\label{fig:sym2}
\includegraphics[scale = 0.12]{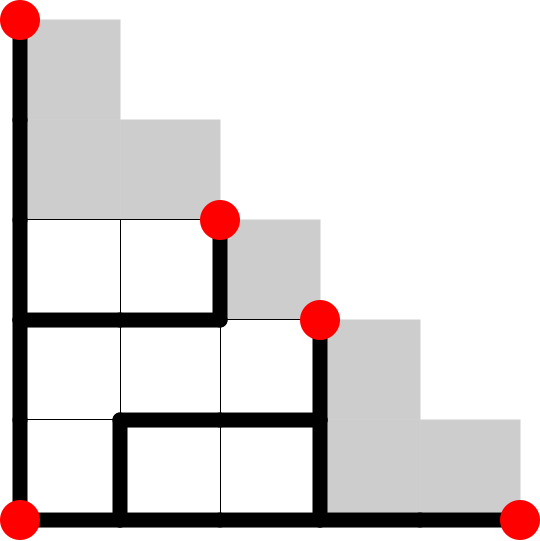}
\qquad
\includegraphics[scale = 0.12]{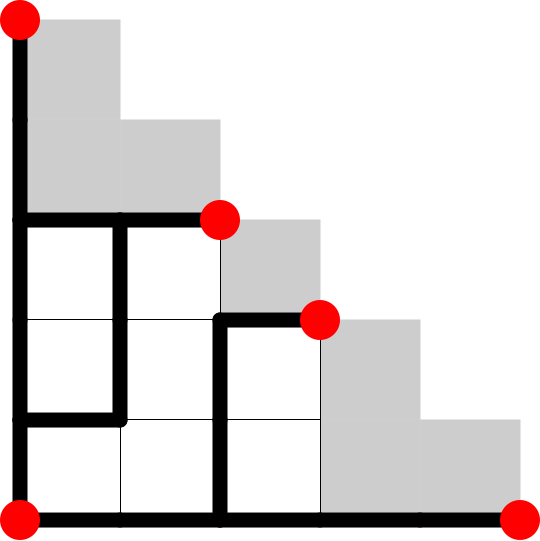}
\caption{Action of $\rho$ on ladder diagram.}
\end{figure}

\begin{remark*}
When $\lambda= (1^{a_1}, 2^{a_2}, \ldots, m^{a_m})$, this symmetry is actually affine: it can be realized as the map $f(x) = -Px + m \cdot \boldsymbol 1$, where $P$ is a permutation matrix, and $\boldsymbol 1$ is the vector with each entry equal to 1. However, for $\lambda$ of a different form, this is no longer true, even in small cases. For example when $\lambda = (1, 2, 4)$, a straightforward computation shows that $\rho$ is no longer affine.
\end{remark*}

\begin{prop}[The $m = 2$ Rotation Symmetry]
Suppose that $m = 2$. There is an order 2 automorphism $\tau$ on $\mathcal{F}(\GT_\lambda)$ that rotates all paths from $(0, 0)$ to $t_1$ by $180^\circ$ into paths from $t_1$ to $(0, 0)$.
\end{prop}

\begin{proof} 
The map $\tau$ takes ladder diagrams to valid ladder diagrams. 
Since $\tau^2 = 1$ and $\tau$ preserves inclusions of ladder diagrams, $\tau$ is an automorphism of $\mathcal{F}(\GT_\lam)$.
\end{proof}

\begin{figure}[htp]\label{fig:twosymmetry}
\includegraphics[scale = 0.12]{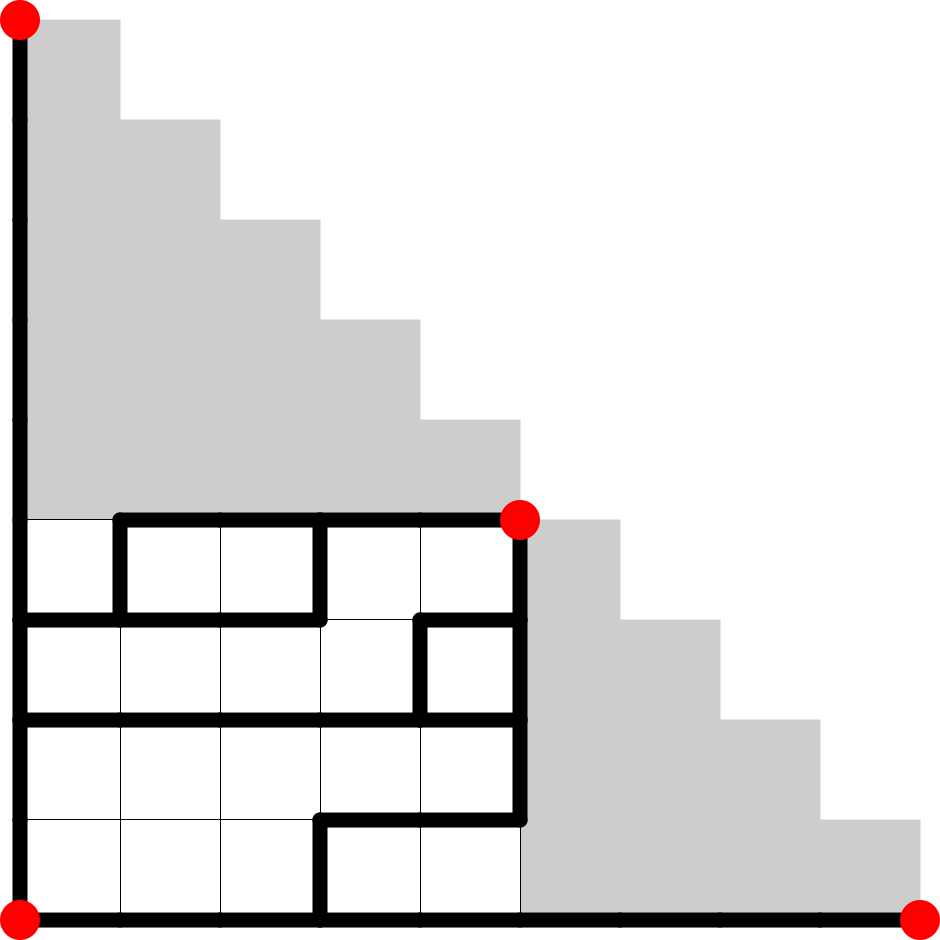}
\qquad
\includegraphics[scale = 0.12]{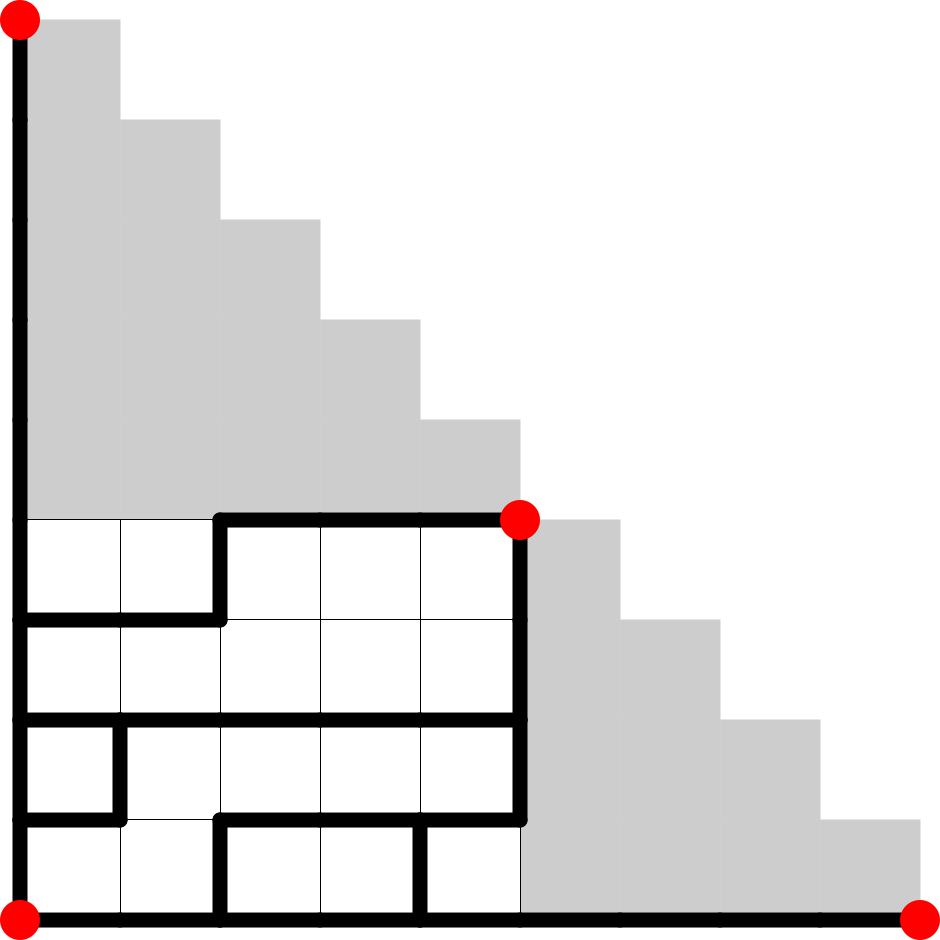}
\caption{Action of $\tau$ on a ladder diagram.}
\end{figure}

\begin{prop}[The $m=2$ Vertex Symmetry]
When $m = 2$, there are two paths $p_1$ and $p_2$ to the terminal vertex $t_1$ that turn exactly once. The map $\alpha$ sending these two paths to each other is an order 2 automorphism of $\GT_\lambda$.
\end{prop}

\begin{proof}
Similar to the proofs above, $\alpha^2=1$ and it clearly preserves inclusion. Therefore, it is an automorphism of $\mathcal{F}(\GT_\lam)$.
\end{proof}

\begin{figure}[htp]\label{fig:twosymmetry2}
\includegraphics[scale = 0.12]{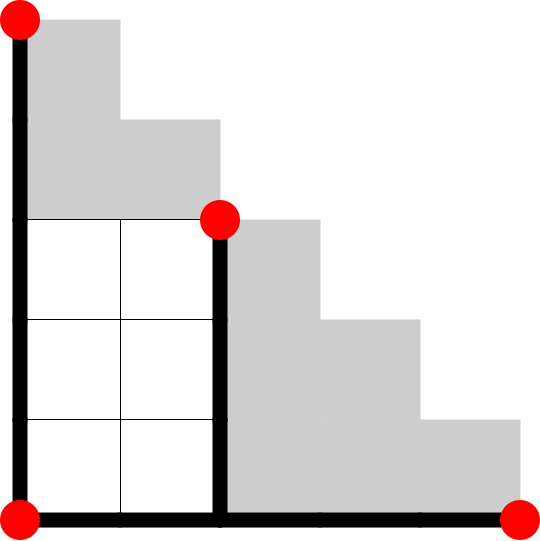}
\qquad
\includegraphics[scale = 0.12]{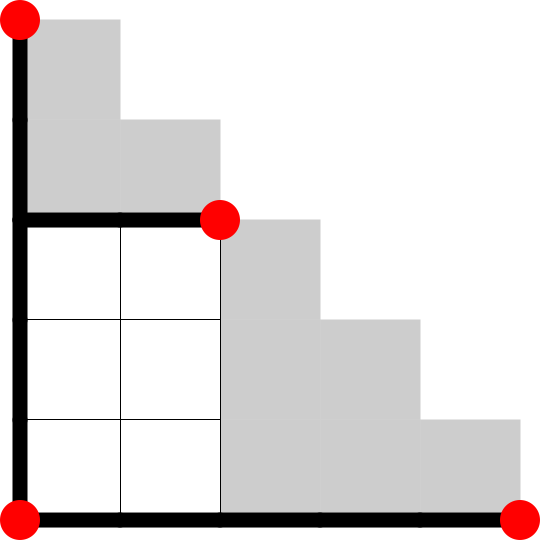}
\caption{Paths $p_1$ and $p_2$.}
\end{figure}

The ladder diagram formed by $p_1$ is a vertex of $\GT_\lam$ that is connected to every other vertex by an edge, and similarly for $p_2$. The map $\alpha$ simply exchanges these two vertices.

It is clear that the group formed by these possible generators is contained in $\Aut(\GT_\lam)$. Note that if $m = 1$, then the polytope is a single point with only the trivial automorphism. If $m = 2$ and either $\lam = (1,2^{a_2})$ or $\lam = (1^{a_1}, 2)$, then the polytope is a simplex and its automorphism group is the symmetric group. In all other cases, either $m = 2$ and $a_1,a_2 \ge 2$, or $m \ge 3$. For both of these cases, we will describe the group generated by the automorphisms listed above and finish proving that this group is the entire automorphism group in Section~\ref{subsec:proving-autom-gp}. 

\begin{thm1.2}[$m$=2 Automorphisms]
Suppose $\lambda = (1^{a_1}, 2^{a_2})$ and $a_1,a_2 \ge 2$. If $a_1 = a_2 = 2$, then 
\begin{equation*}
\Aut(\GT_\lambda) \cong D_4 \times \Z_2.
\end{equation*}
Otherwise, 
\begin{equation*}
\Aut(\GT_\lambda) \cong D_4 \times \Z_2 \times \Z_2^{\delta_{a_1, a_2}}.
\end{equation*}
\end{thm1.2}

\begin{proof}[Proof of inclusion.]
We show that if $a_1 = a_2 = 2$, then $D_4 \times \Z_2 \subseteq \Aut(\GT_\lambda)$ and otherwise $D_4 \times \Z_2 \times \Z_2^{\delta_{a_1, a_2}} \subseteq \Aut(\GT_\lambda)$.

Suppose $a_1 \neq a_2$. Collecting the generators applicable in this case from previous propositions, we have the subgroup of automorphisms $\langle \mu, \mu_1, \tau, \alpha \rangle$, and note these generators satisfy the following relations $\mu^2 = \mu_1^2 = \tau^2 = \alpha^2 = 1, \mu \tau = \tau \mu_1$ and all other variables commute. The subgroup $\langle \tau \mu , \mu \rangle$ is $D_4$. The generator $\alpha$ commutes with all other generators, so the resulting group is isomorphic to $D_4 \times \Z_2$. 

Assuming $a_1 = a_2 \geq 3$, we have the subgroup described in the previous case $a_1 \neq a_2$, but with the additional generator $\rho$. Note $\rho$ commutes with all of these generators from the previous case, so the resulting group is $D_4 \times \Z_2^2$.
When $a_1 = a_2 = 2$, note that $\rho = \tau$, resulting in a smaller subgroup.
\end{proof}

We take care in specifying how to write down the composition of elements of $\Aut(\GT_\lam)$ as a tuple. This is not as straightforward as in the $m=2$ case because the Flip Symmetry does not act locally and does not commute with the other symmetries.

In the statement of Theorem~\ref{thm:autom-m>=3}, elements of $\Aut(\GT_\lam)$ are tuples $(\sigma_1,\sigma_2,z_1,\ldots,z_{r_1},z_{r_1+1},z_{r_1+2})$. Here $\sigma_1 \in S_{a_2}^{\delta_{1,a_1}}$ and $\sigma_2 \in S_{a_{m-1}}^{\delta_{1,a_{m}}}$ correspond to the Symmetric Group Symmetry and $z_1,\ldots,z_{r_1} \in \Z_2$ correspond to the $k$-Corner Symmetries. Finally $z_{r_1+1} \in \Z_2$ corresponds to the Corner Symmetry, $z_{r_1+2} \in \Z_2$ corresponds to the Flip Symmetry. Let $g \in S_{a_2}^{\delta_{1,a_1}} \times S_{a_{m-1}}^{\delta_{1,a_{m}}} \times \Z_2^{r_1+1}$ be such that $(\sigma_1,\sigma_2,z_1,\ldots,z_{r_1},z_{r_1+1},z_{r_1+2}) =: (g,z')$. Then $(g,z_{r_1+2}) \cdot (g',z_{r_1+2}') = (g\varphi(z_{r_1+2})(g'),z_{r_1+2}z_{r_1+2}')$ where $\varphi(0)$ is the identity map on $S_{a_2}^{\delta_{1,a_1}} \times S_{a_{m-1}}^{\delta_{1,a_{m}}}$ and $\varphi(1)$ is the map sending $(\sigma_1, \sigma_2, z_1, \ldots, z_{r_1}, z_{r_1+1}) \mapsto (\sigma_2, \sigma_1, z_{r_1}, \ldots, z_1, z_{r_1+1})$. This is formally stated in Theorem~\ref{thm:autom-m>=3}.

\begin{thm1.3}[$m \geq 3$ Automorphisms]
Suppose $\lam = 1^{a_1}\ldots m^{a_m}$ and $m \ge 3$. Let $r_1$ be the number of $k$ such that $a_k, a_{k+1} \geq 2$. Let $r_2 = 1$ if $\lam = \lam'$ and let $r_2 = 0$ otherwise. Then 
\begin{equation*}
\Aut(\GT_\lam) \cong (S_{a_2}^{\delta_{1,a_1}} \times S_{a_{m-1}}^{\delta_{1,a_{m}}} \times \Z_2^{r_1+1}) \ltimes_\varphi \Z_2^{r_2}
\end{equation*}
where if $r_2 = 1$, then $\varphi: \Z_2 \to \Aut(S_{a_2}^{\delta_{1,a_1}} \times S_{a_{m-1}}^{\delta_{1,a_{m}}} \times \Z_2^{r_1+1})$ sends the nonidentity element of $\Z_2$ to the map sending $(\sigma_1, \sigma_2, z_1, \ldots, z_{r_1}, z_{r_1+1}) \mapsto (\sigma_2, \sigma_1, z_{r_1}, \ldots, z_1, z_{r_1+1}).$
\end{thm1.3}

\begin{proof}[Proof of inclusion.]
We show that $(S_{a_2}^{\delta_{1,a_1}} \times S_{a_{m-1}}^{\delta_{1,a_{m}}} \times \Z_2^{r_1+1}) \ltimes_\varphi \Z_2^{r_2} \subseteq \Aut(\GT_\lam)$.

By the previous propositions, our generators are $\mu, \mu_1, \ldots, \mu_{m-1}, S_{a_2}, S_{a_{m-1}}$ with $S_{a_2}$ or $S_{a_{m-1}}$ possibly omitted. Whichever symmetries are present commute since they act on disjoint sets of edges in the ladder diagrams. 

If $\lambda$ is reverse symmetric, we also have the generator $\rho$. Note that since $\rho$ flips every ladder diagram about $y = x$, $\rho$ satisfies the following commutation relations: $\rho \mu = \mu \rho$, $\rho \mu_i = \mu_{m-i} \rho$, and for $\sigma\in S_{a_2}, S_{a_{m-1}}$ ($a_2$ and $a_{m-1}$ must be equal for $\lambda$ to be reverse symmetric),$\rho \sigma = \sigma\rho$. These relations are enough to give a subgroup of the stated form.
\end{proof}

In the following sections, we finish the proofs of Theorems~\ref{thm:autom-m=2} and \ref{thm:autom-m>=3} by showing the group formed by our generators is the entire automorphism group by bounding the size of $\Aut(\GT_\lambda)$ by the size of the group formed by our generators. We will examine the action of any combinatorial automorphism on $\cal{F}(\gamlam)$ and apply the Orbit-Stabilizer theorem. In order to do so, we first develop ways to classify and partition the facets of $\GT_\lambda$.

\subsection{Chains of Facets}\label{subsec:facets}
In this section, we present a useful model for the facets of $\GT_\lam$, define the \emph{facet chains} of $\gamlam$, and discuss how an automorphism can act on facet chains. This line of reasoning is motivated by the following observation.

\begin{lem}\label{lem:facets}
An automorphism of $\GT_\lambda$ is determined by where it sends the facets of $\GT_\lambda$ or equivalently, where it sends the ladder diagrams of facets.
\end{lem}
\begin{proof}
This follows from the fact that for a general polytope $P$, every face of $P$ can be written as an intersection of the facets of $P$.  Thus specifying the image of each facet suffices to specify the image of any face.
\end{proof}

\begin{defn}
We define the \emph{interior edges} of $\gamlam$ to be all edges of the form $\{(s_j, n - s_{j+1}), (s_j, n - s_{j+1} + 1)\}$ or $\{(s_j, n - s_{j+1}), (s_j+1, n - s_{j+1})\}$ and all edges lying inside $\gamlam$. All other edges of $\gamlam$ are considered \emph{boundary edges}.
\label{defn:interior-edges}
\end{defn}

\begin{figure} [htp]
\includegraphics[scale=0.12]{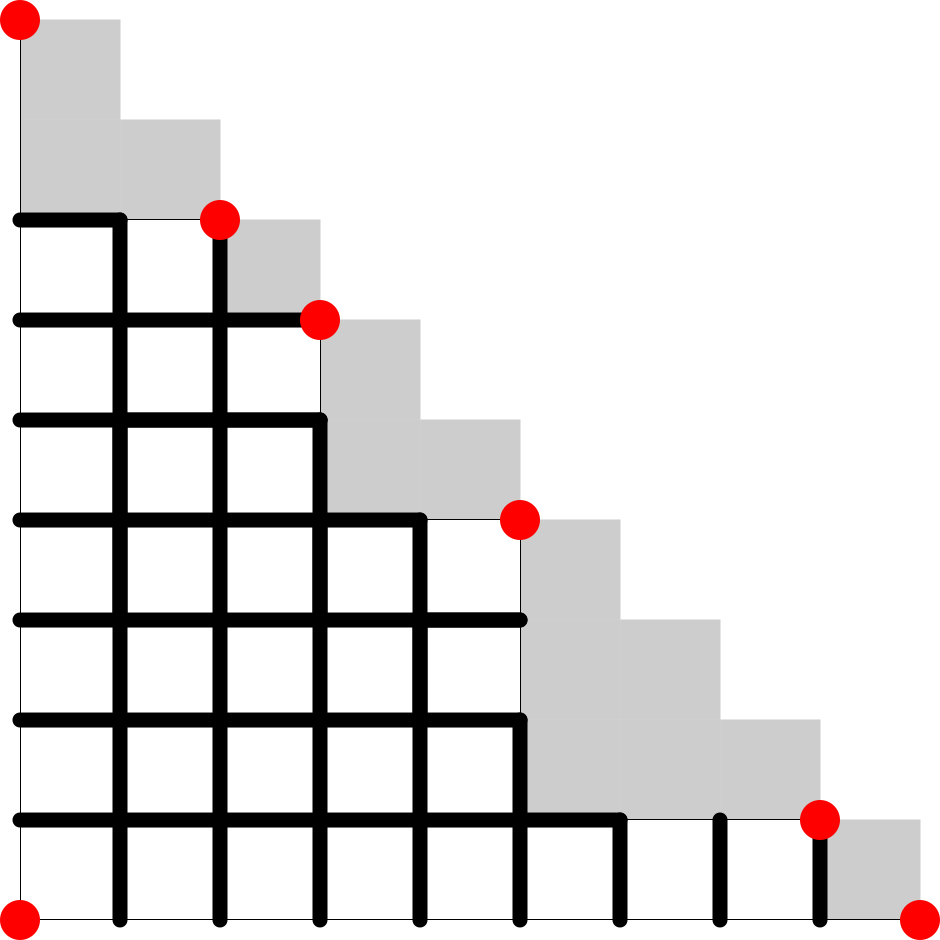}
\qquad
\includegraphics[scale=0.12]{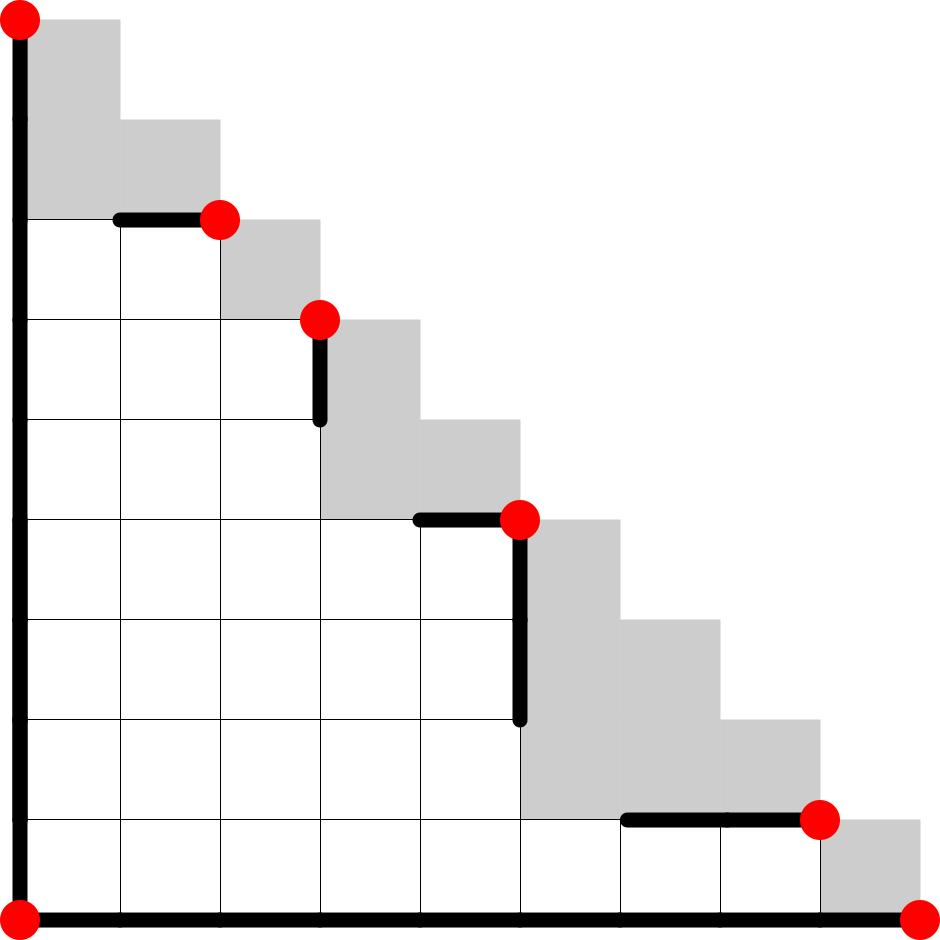}
\caption{Left: interior edges of $\gamlam$. Right: boundary edges of $\gamlam$.}
\label{fig:interior-edges}
\end{figure}

\begin{prop}
\label{fact:facets-int-edges}
The facets of $\GT_\lambda$ are in bijection with the interior edges of $\gamlam$. 
\end{prop}

\begin{proof}
Since ladder diagrams in $\mathcal{F}(\gamlam)$ are graded by number of bounded regions, any facet will correspond to a ladder diagram with all possible edges except one which will be an interior edge. The bijection is given by mapping a facet $F$ to the single interior edge not contained in its associated ladder diagram. For the other direction, note that any ladder diagram missing a single interior edge is a valid ladder diagram and represents a facet.
\end{proof}

We will often represent a facet $F \in \GT_\lam$ by its corresponding interior edge which we denote by $e(F)$. A face is not contained in facet $F$ iff its ladder diagram contains edge $e(F)$.

\begin{figure}[htp]
\includegraphics[scale = 0.12]{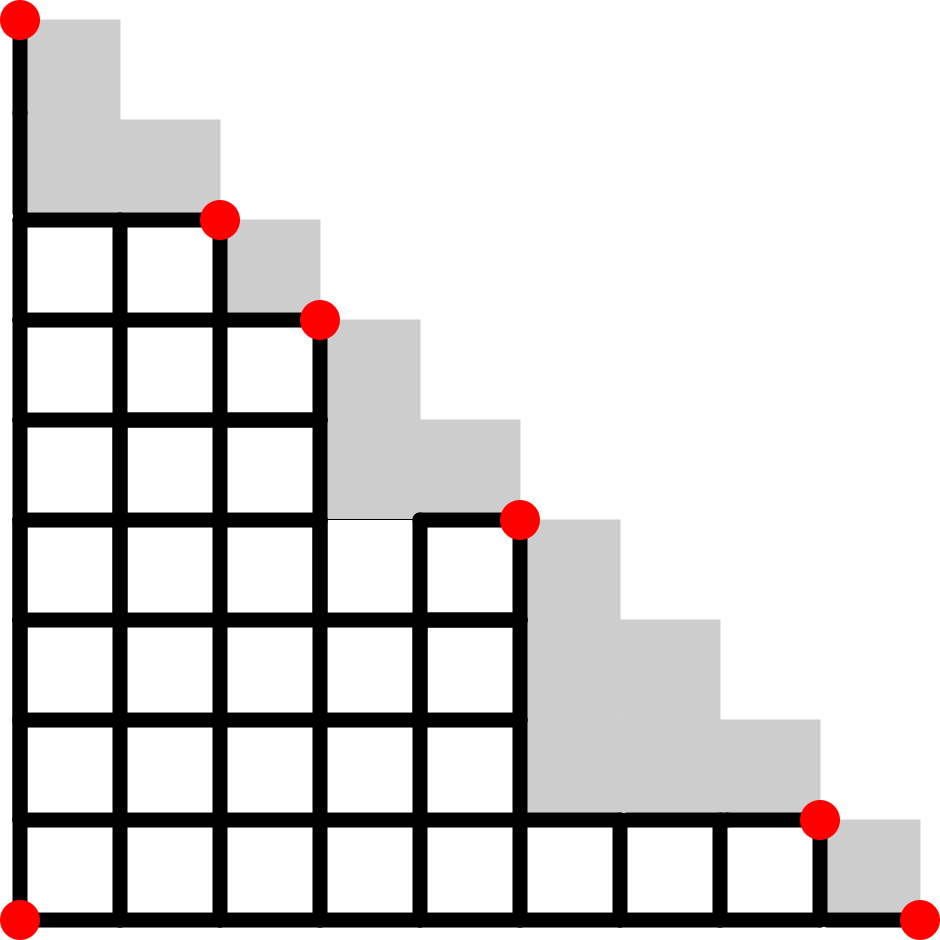}
\qquad
\includegraphics[scale = 0.12]{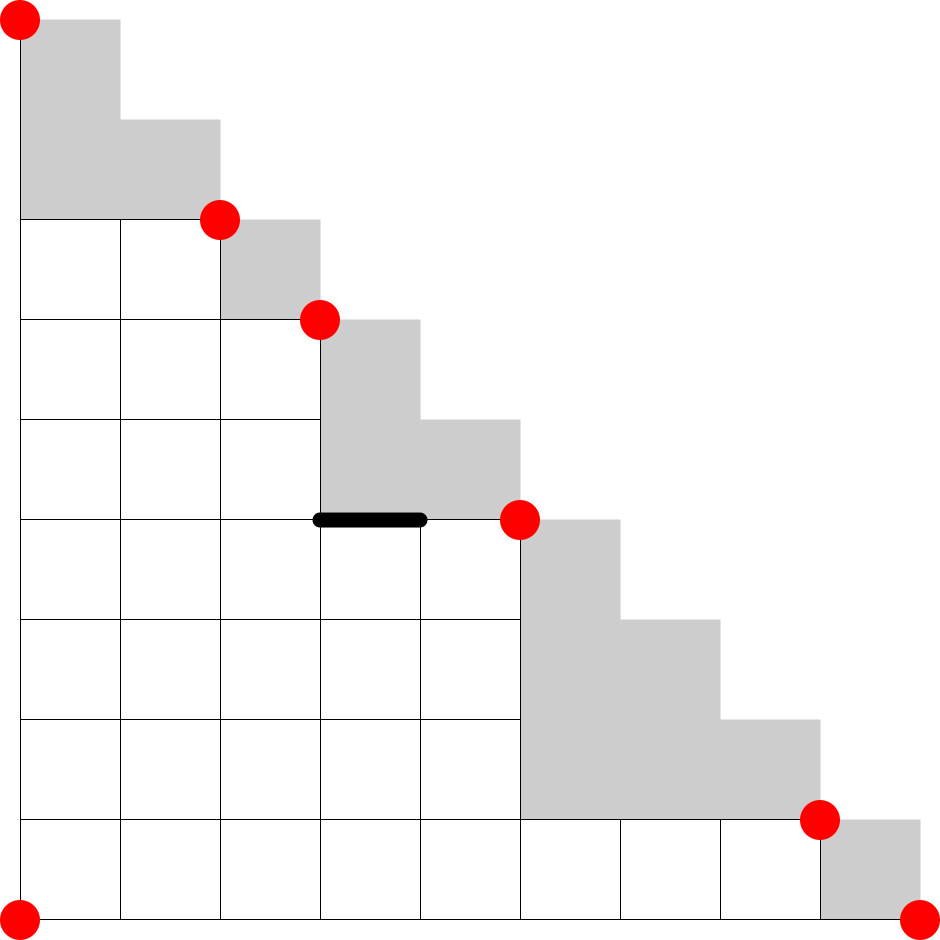}
\caption{Left: ladder diagram of a facet. Right: complement of ladder diagram.}
\label{fig:facet-diag}
\end{figure}

\begin{defn}
Two facets are called \emph{dependent} if their intersection is a $d-3$ dimensional face.
\end{defn}

Given two facets $F_1$ and $F_2$, a necessary condition for $F_1$ and $F_2$ to be dependent is for them to be arranged as shown in Figure~\ref{fig:d-3dim}. As shown in Figure~\ref{fig:nondependent}, this is not always sufficient if there are already forced equalities amongst the $4$ shaded coordinates.

\begin{figure}[h!]
\includegraphics[scale = 0.14]{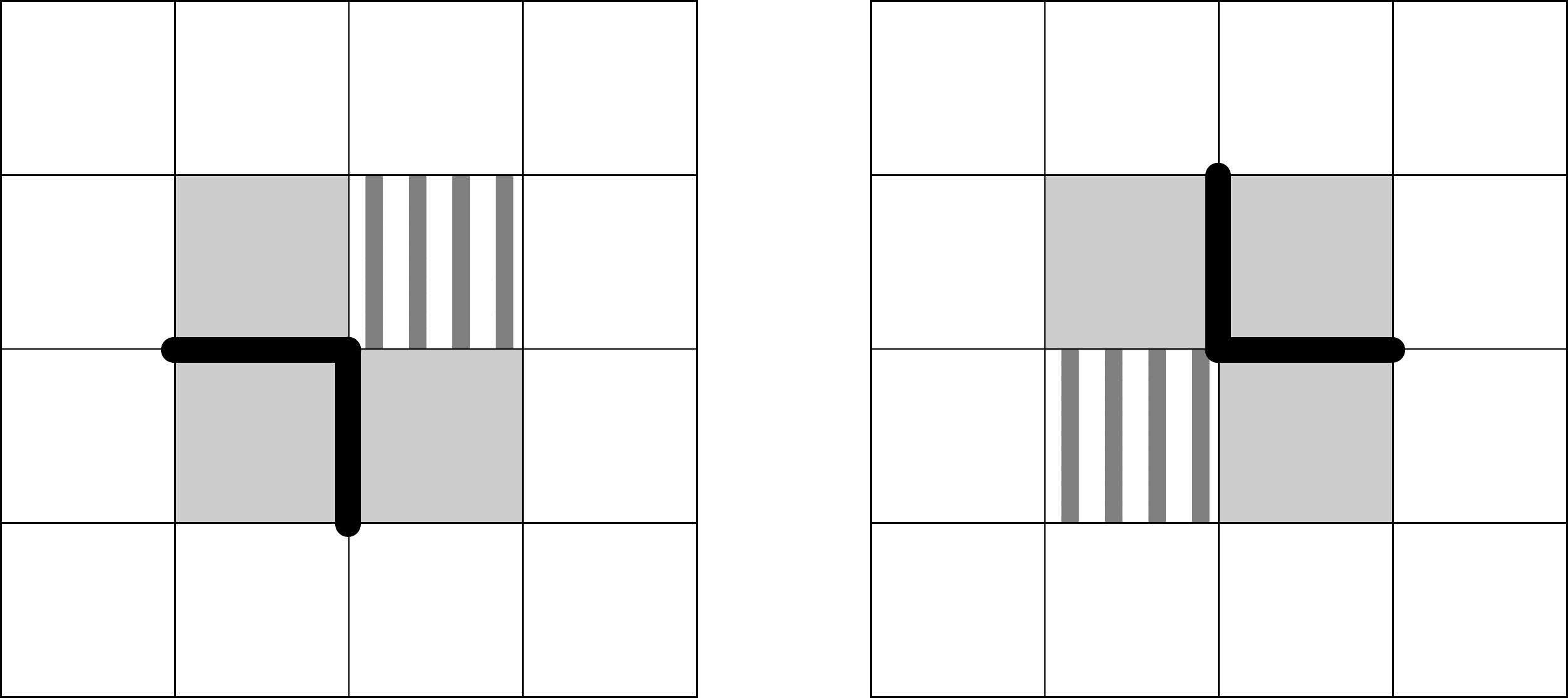}
\caption{The gray boxes indicate coordinates $x_{i,j}$ that are equal on each facet. The dashed box indicates the coordinate forced to be equal to the other three.}
\label{fig:d-3dim}
\end{figure}

\begin{figure}[h!]
\includegraphics[scale = 0.12]{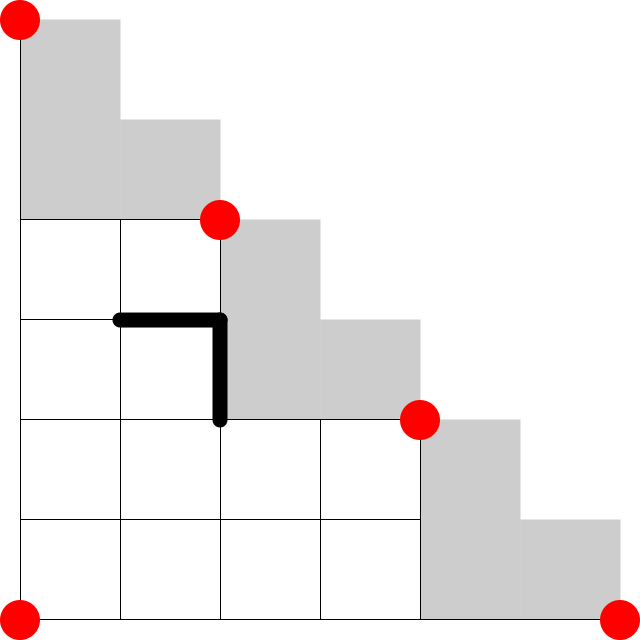}
\qquad
\includegraphics[scale = 0.12]{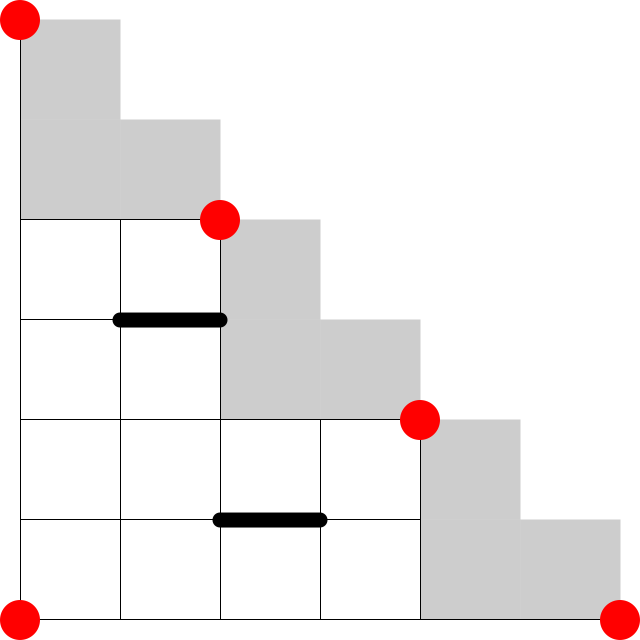}
\caption{Examples of facets that are not dependent.}
\label{fig:nondependent}
\end{figure}

Facet dependencies naturally partitions the facets of $\GT_\lambda$ into \emph{facet chains} which are easy to represent visually.

\begin{defn}[Facet Chains]
A \emph{facet chain} $C = (F_1,\ldots, F_l)$ of length $l$ is an ordered list of facets $F_1,\ldots,F_l$ such that $F_j$ is dependent on $F_{j-1}$ and facets $F_1,F_l$ are not dependent on any facets not in $C$. Visually, a chain is a set of edges $e(F_j)$ of $\gamlam$ forming a zig-zag pattern. The facets $F_1,\ldots,F_l$ are ordered such that $e(F_1)$ has the smallest $x$-coordinate.
\end{defn}

Let $\mathcal{C}$ denote the set of facet chains of $\gamlam$. These chains partition the interior edges of $\gamlam$. After we have the partition, it is then natural to study the relations between these facet chains. We define the notion of adjacency between facet chains as follows.

\begin{figure}[h!]
\includegraphics[scale = 0.12]{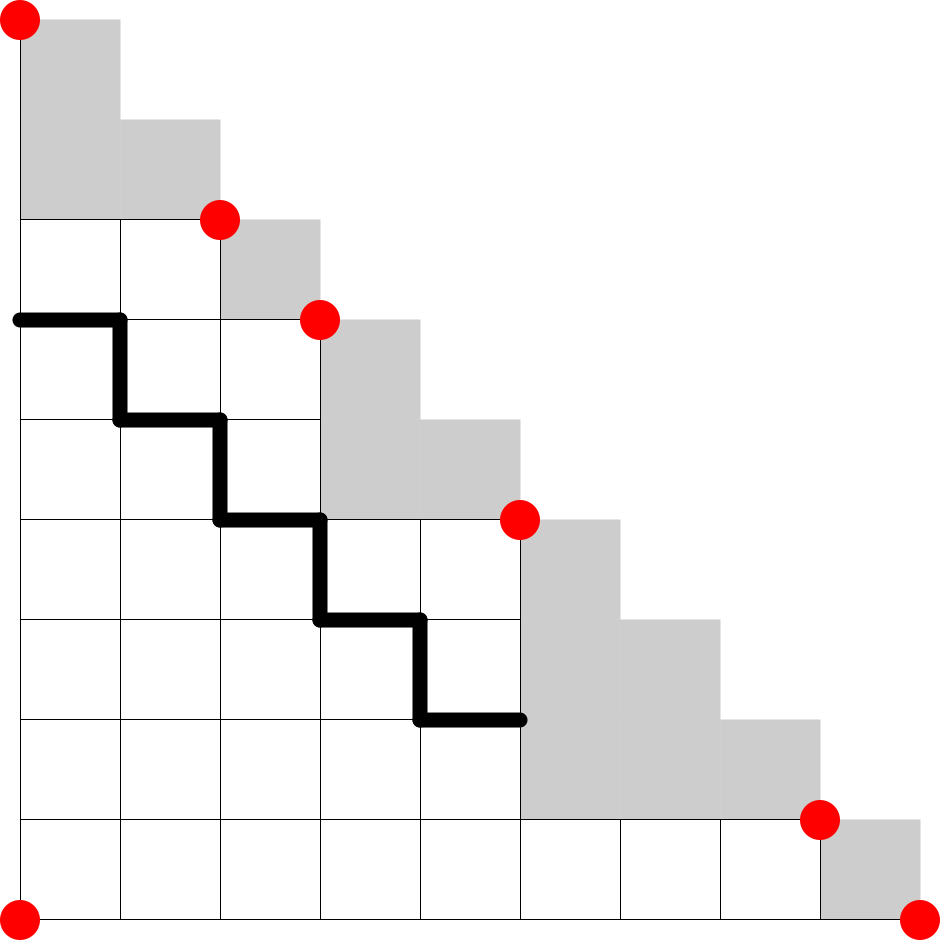}
\qquad
\includegraphics[scale = 0.12]{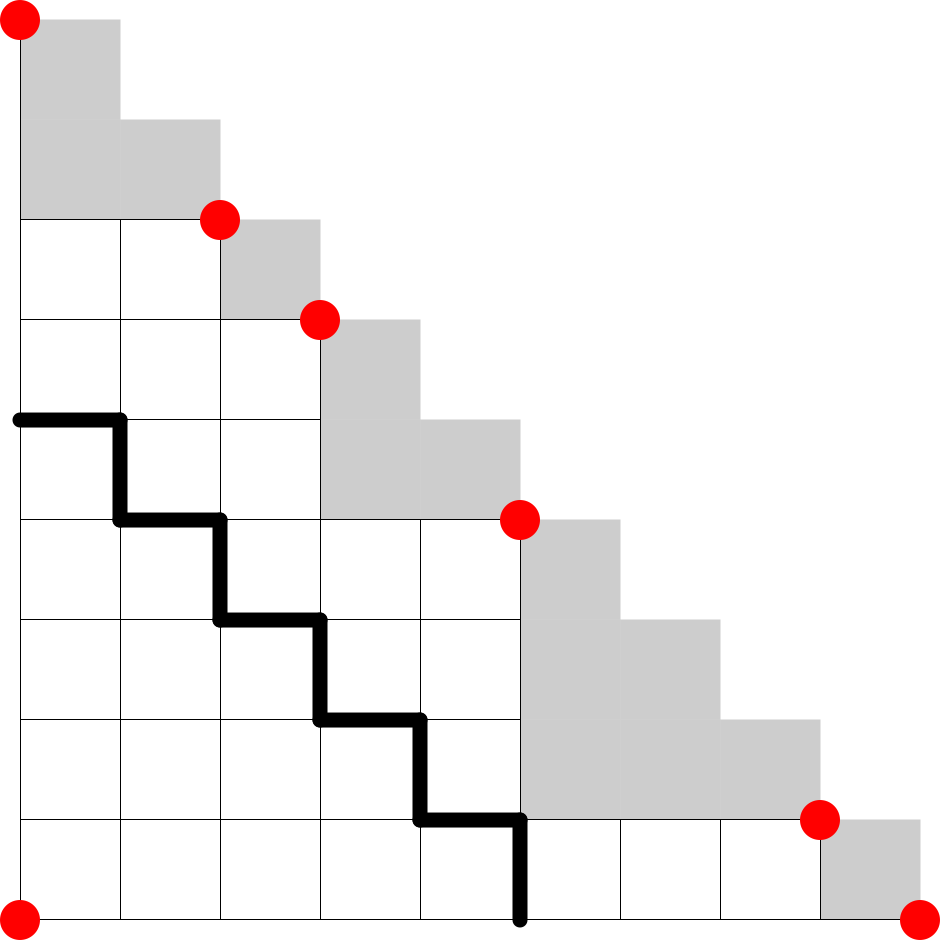}
\caption{Two adjacent facet chains.}
\label{fig:facet-chains}
\end{figure}

\begin{defn}[Adjacent Chains]
Two different chains $C, C' \in \mathcal{C}$ are \emph{adjacent} if there exists $F_1, F_2 \in C$ and $J_1, J_2 \in C'$ such that $F_1 \bigcap F_2 = J_1 \bigcap J_2$. We say $C$ and $C'$ are adjacent at $k$ points if there are $k$ distinct sets of facets $F_1,F_2,J_1,J_2$ with $F_1 \bigcap F_2 = J_1 \bigcap J_2$.
\end{defn}

Visually, two chains are adjacent iff one chain sits directly to the North-East of the other chain as shown in Figure~\ref{fig:facet-chains}.

We are now ready to study how each automorphism $\phi\in\Aut(\GT_{\lambda})$ acts on facet chains. Specifically, the action of an automorphism $\phi \in \Aut(\GT_\lam)$ can be extended to sets of facets. For any sets of facets $X_1$ and $X_2$, we say $\phi(X_1) = X_2$ if the restriction of $\phi$ to $X_1$ is a bijection between sets $X_1$ and $X_2$. In particular, we will often abuse notation and write $\phi(C_1) = C_2$, thinking of chains $C_1$ and $C_2$ as sets of facets. We now state a few simple lemmas that will be essential to our proofs of Theorems~\ref{thm:autom-m=2} and \ref{thm:autom-m>=3}.

\begin{lem}\label{lem:fcmapstofc}
Let $C$ be a facet chain and $\phi\in\Aut(\GT_{\lambda})$. Then $\phi(C)$ is a facet chain of the same length as $C$.
\end{lem}

\begin{lem}\label{lem:admapstoad}
Let $C$ and $C'$ be adjacent facet chains and $\phi\in\Aut(\GT_\lambda)$. Then $\phi(C)$ and $\phi(C')$ are also adjacent facet chains.
\end{lem}

\begin{proof}[Proof of Lemma~\ref{lem:fcmapstofc} and \ref{lem:admapstoad}]
Dependency between facets and adjacency between facet chains can be realized as properties of coatoms in the face lattice of the GT polytope. Thus, the lemmas followed directly from the definition of an automorphism.
\end{proof}

Given a notion of adjacency between facet chains, it is natural to examine the adjacency graph and how an automorphism can act on it. 

\begin{defn}(Adjacency Graph)
For any $\lam$, let $\cal{G}_\lam$ be the adjacency graph of the chains of $\gamlam$ with length $\ge 2$. More precisely, the nodes of $\cal{G}_\lam$ are the chains in $\cal{C}$ of length $\ge 2$ and there is an edge between two nodes iff their corresponding chains are adjacent. Note that $\cal{G}_\lam$ is connected. We make $\cal{G}_\lam$ a rooted planar tree by letting its root be $C_0$ and giving its nodes the ordering from their corresponding chains in $\gamlam$.
\label{def:adj-graph}
\end{defn}

\begin{figure}[h!]
\includegraphics[scale=0.1]{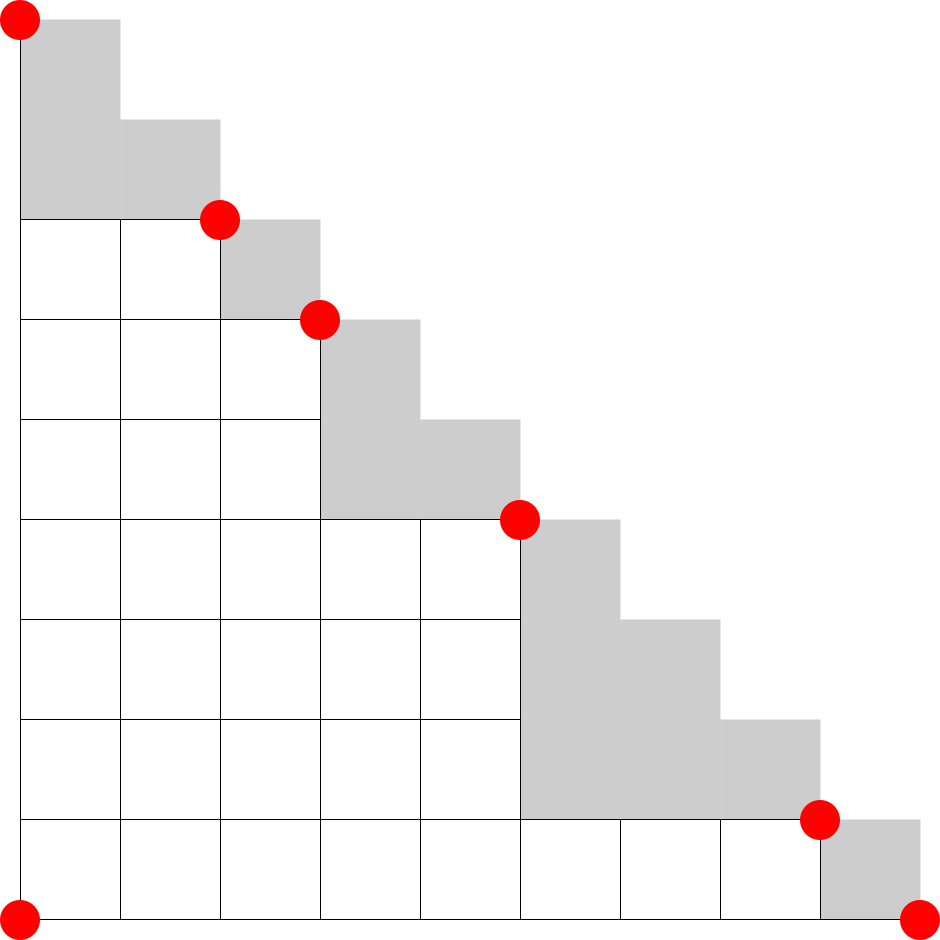}
\includegraphics[scale=0.045]{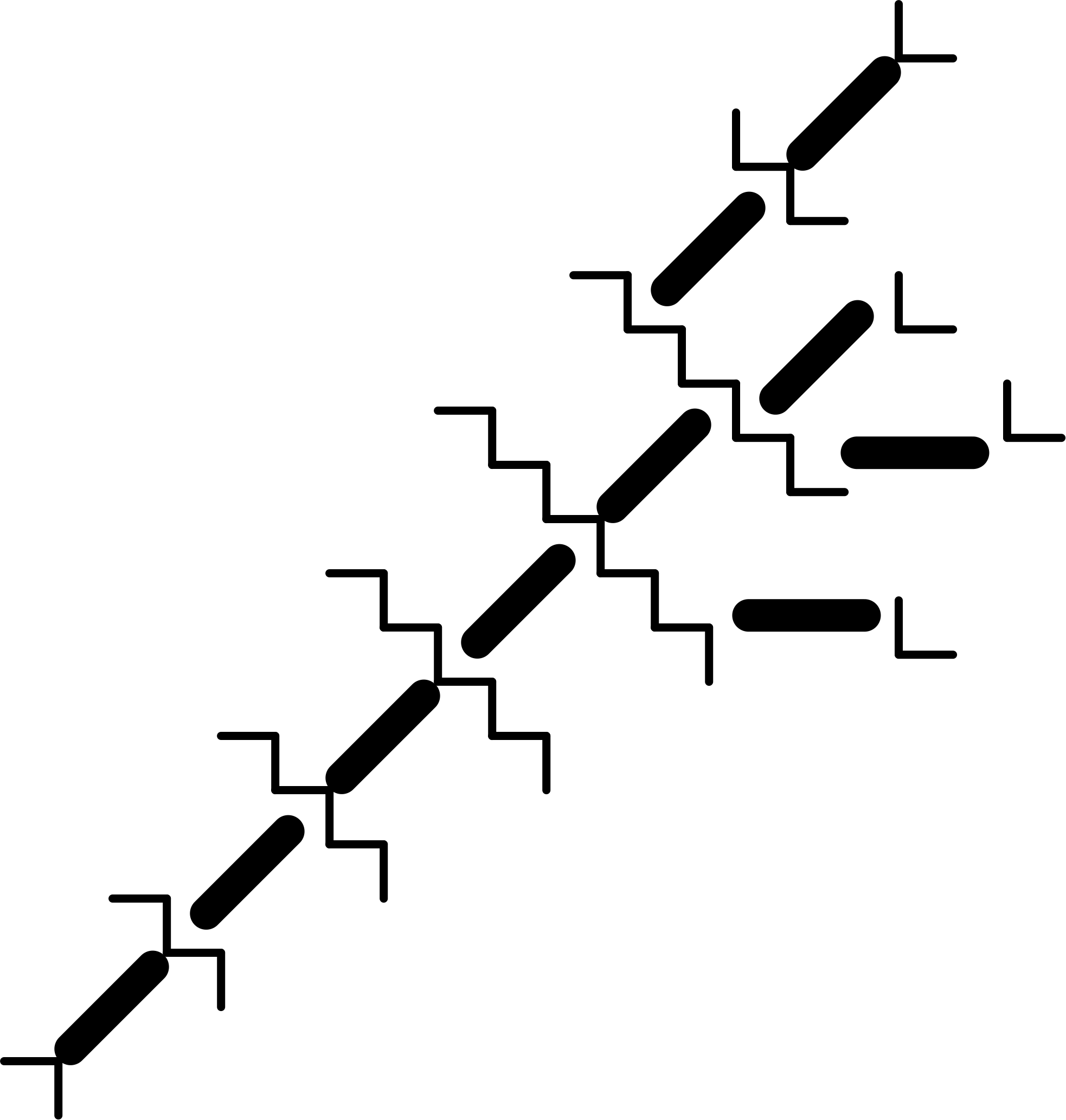}
\qquad
\includegraphics[scale=0.1]{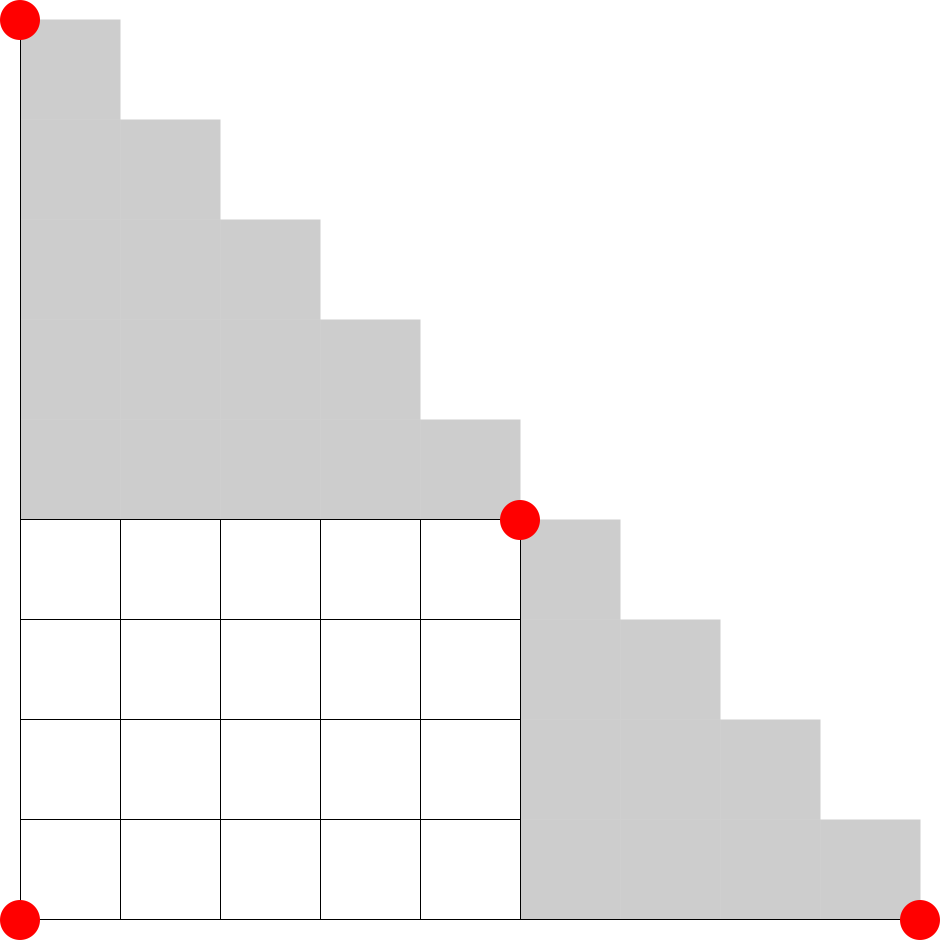}
\includegraphics[scale=0.05]{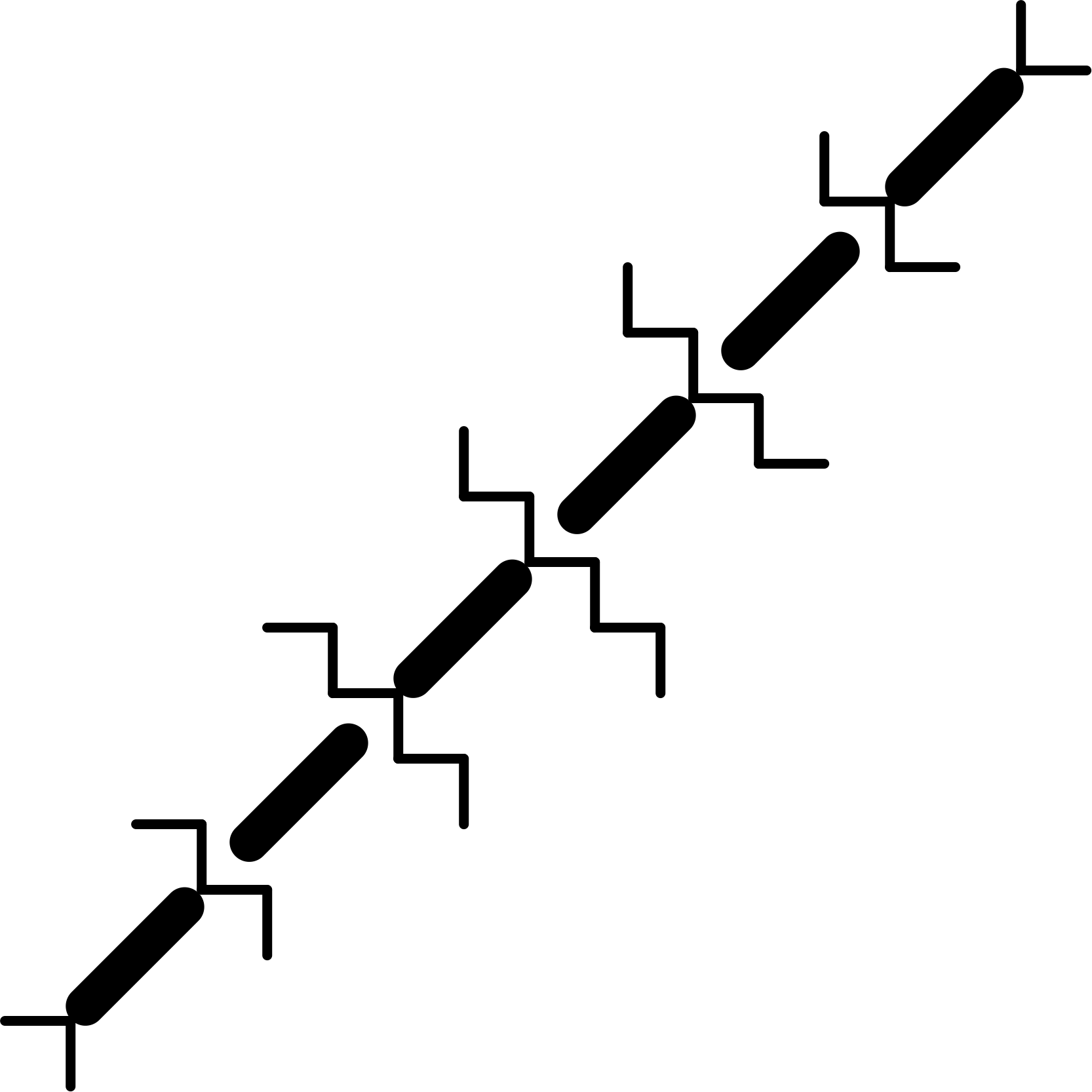}
\caption{Left: $\lambda=(1^2,2^1,3^2,4^3,5^1)$ and its adjacency graph. Right: $\lambda=(1^5,2^4)$ and its adjacency graph.}
\label{fig:adjacency-graph-3}
\end{figure}

Note that the length $2$ chains are exactly the chains with less than two neighbors in $\cal{G}$ since longer chains will have an adjacent chain above or below. Thus, the length $2$ chains are the leaves of $\cal{G}_\lam$. 

Unless specified otherwise, when talking about the action of an automorphism on nodes of $\cal{G}_\lam$, we are treating chains as sets of facets so we will not specify how the automorphism acts on the facets of chains. So if we say that an automorphism fixes a node of $\cal{G}_\lam$, we are not specifying whether the automorphism flips the chain.

\begin{lem}
Given facet chains $C = (F_1,\ldots,F_k)$ and $C' = (F_1',\ldots,F_k')$ and $\phi\in\Aut(\GT_\lam)$ such that $\phi(C)=C'$, either $\phi(F_i) = F_i'$ for all $1 \le i \le k$ or $\phi(F_i) = F_{k+1-i}'$ for all $1 \le i \le k$.
\label{lem:chain-flip}
\end{lem}

\begin{proof}
If $\phi(F_1) = F_1'$, the chain of dependencies of the $F_i$ will determine the image of each $F_i$. More specifically,  $\phi(F_2)$ must be dependent on $\phi(F_1)=F_1'$, so we must have $\phi(F_2) = F_2'$, and so forth. 
Else if $\phi(F_1) \neq F_1'$, we must have $\phi(F_1) = F_k'$ since $F_1$ and $F_k$ are the only facets dependent on exactly one other facet.
Again the chain of dependencies implies that $\phi(F_2)$ must be dependent on $\phi(F_1)=F_k'$, so we must have $\phi(F_2) = F_{k-1}'$, and so forth.
\end{proof}

Under the assumptions of Lemma~\ref{lem:chain-flip}, if $\phi(F_i) = F_{k+1-i}'$ for all $1 \le i \le k$, then we say $C$ is mapped to the \emph{flip} of $C'$. 
In particular if $\phi(C) = C$, then $\phi$ either flips or does not flip $C$. We call this the \emph{orientation} of chain $C$ under $\phi$.

\begin{lem}
Suppose $C$ and $C'$ are adjacent facet chains with distinct facets $F_1, F_2, F_3, F_4 \in C$ and $J_1, J_2, J_3, J_4 \in C'$ such that $F_1\cap F_2=J_1\cap J_2$ and $F_3\cap F_4=J_3\cap J_4$. Equivalently, $C$ and $C'$ are adjacent at $\ge 2$ points. Given an automorphism $\phi \in \Aut(\GT_\lambda)$ such that $\phi(C)=C$ and $\phi(C')=C'$, $C$ is flipped under $\phi$ iff $C'$ is flipped under $\phi$.
\label{lem:two-points-intersection-same-orientation}
\end{lem}

\begin{proof}
Suppose $\phi$ does not flip $C_i$, so $\phi(F_1) = F_1$ and $\phi(F_2) = F_2$. Since $F_1 \bigcap F_2 = J_1 \bigcap J_2$, and the intersection of any pair of facets in a chain is unique, we must have $\phi(J_1), \phi(J_2) \in \{J_1, J_2\}$. Suppose for contradiction that $\phi(J_1) = J_2$ implying $C_j$ is flipped. Note flipping a chain can preserve at most one adjacent pair of facets in that chain, so we must have $\phi(\{J_3,J_4\}) \neq \{J_3, J_4\}$. Again the intersection of two facets in a chain is unique, implying $\phi(J_3) \bigcap \phi(J_4) \neq J_3 \bigcap J_4$, but $\phi(J_3) \bigcap \phi(J_4) = \phi(F_3) \bigcap \phi(F_4) = F_3 \bigcap F_4 = J_3 \bigcap J_4$, a contradiction. The reverse direction is similar.
\end{proof}

\begin{lem}
If an automorphism $\phi$ fixes every node of $\cal{G}_\lam$, then all nodes corresponding to chains of length $>2$ have the same orientation under $\phi$.
\label{lem:adj-chains-same-orientation}
\end{lem}

\begin{proof}
By Lemma~\ref{lem:two-points-intersection-same-orientation}, it only remains to prove that if two adjacent chains of length $> 2$ are adjacent at one point, then they have the same orientation under $\phi$. Notice that here when we talk about a point at which two adjacent chains are adjacent, we actually mean a $d-3$ dimensional face as the intersection. See Figure~\ref{fig:pointd-3} for the natural correspondence. 

\begin{figure}[h!]
\includegraphics[scale=0.12]{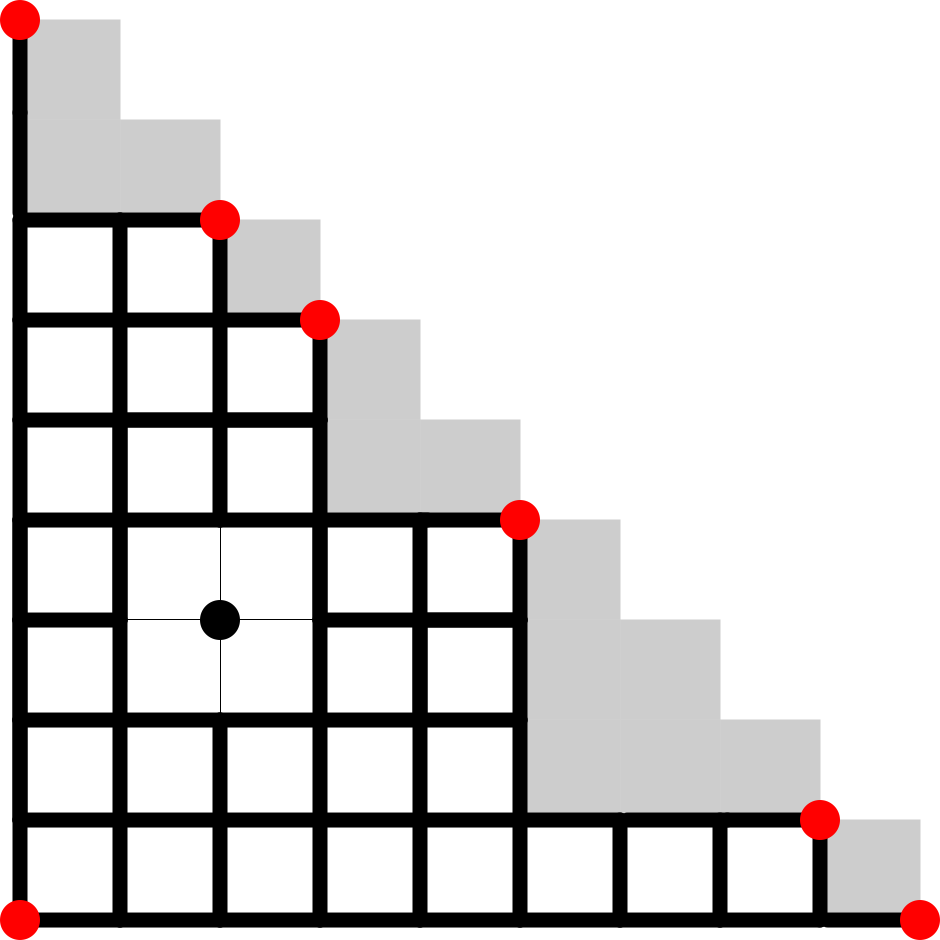}
\caption{The correspondence between an interior point of $\gamlam$ and the ladder diagram for a $d-3$ dimensional face}
\label{fig:pointd-3}
\end{figure}

Consider the set of points at which two chains with length greater than 2 can be adjacent.  These are the points inside $\gamlam$, excluding the corner points near terminal vertices, whose nearby parts $a_i,a_{i+1}$ have size at least 2. We partition these points into sets as follows: two points $p_1, p_2$ are in the same set if they lie along the same diagonal and if all of the points lying on this diagonal and lying between $p_1,p_2$ are inside $\gamlam$. See Figure~\ref{fig:interior-point} for an example.

\begin{figure}[h!]
\includegraphics[scale=0.12]{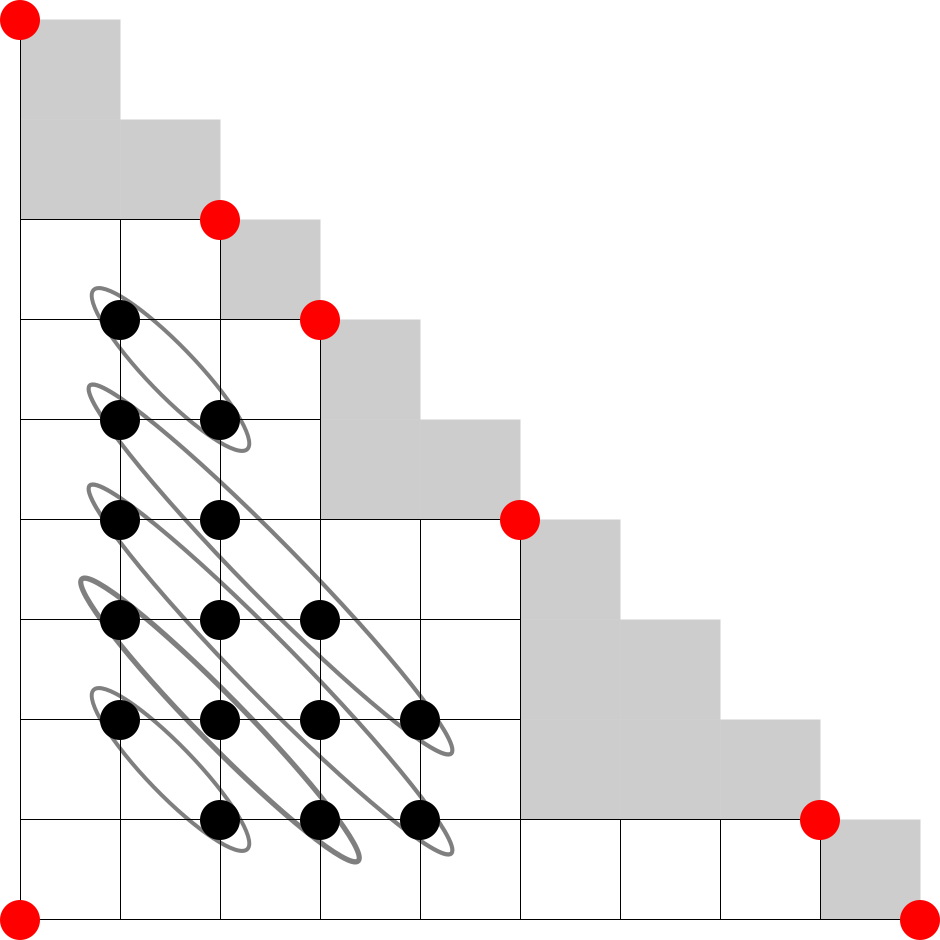}
\caption{Partition of interior points}
\label{fig:interior-point}

\end{figure}

If two chains are adjacent, then the points at which they are adjacent correspond to one of these sets.
So we are only concerned with adjacencies corresponding to sets of size $1$. There only exist sets of size $1$ corresponding to adjacencies between chains of length $> 2$ if either $a_1 = 2$ or $a_m = 2$. See Figure~\ref{fig:interior-points2} for an example.

\begin{figure}[h!]
\includegraphics[scale=0.12]{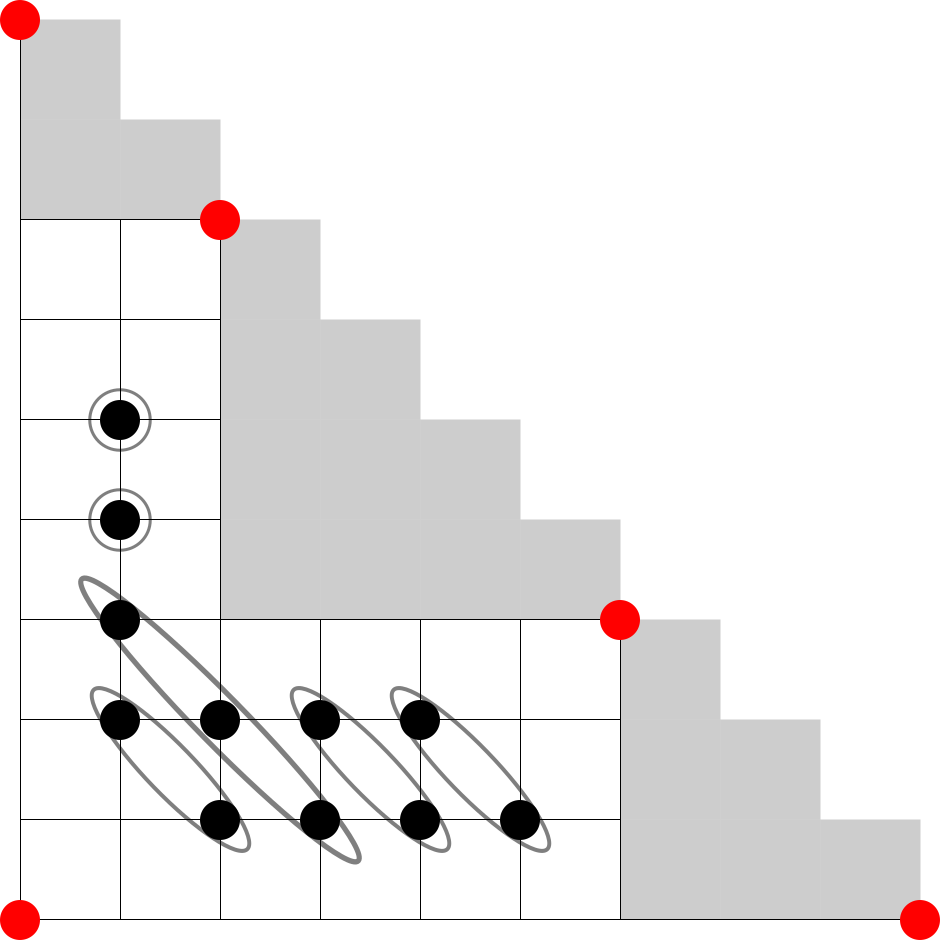}
\qquad
\includegraphics[scale=0.12]{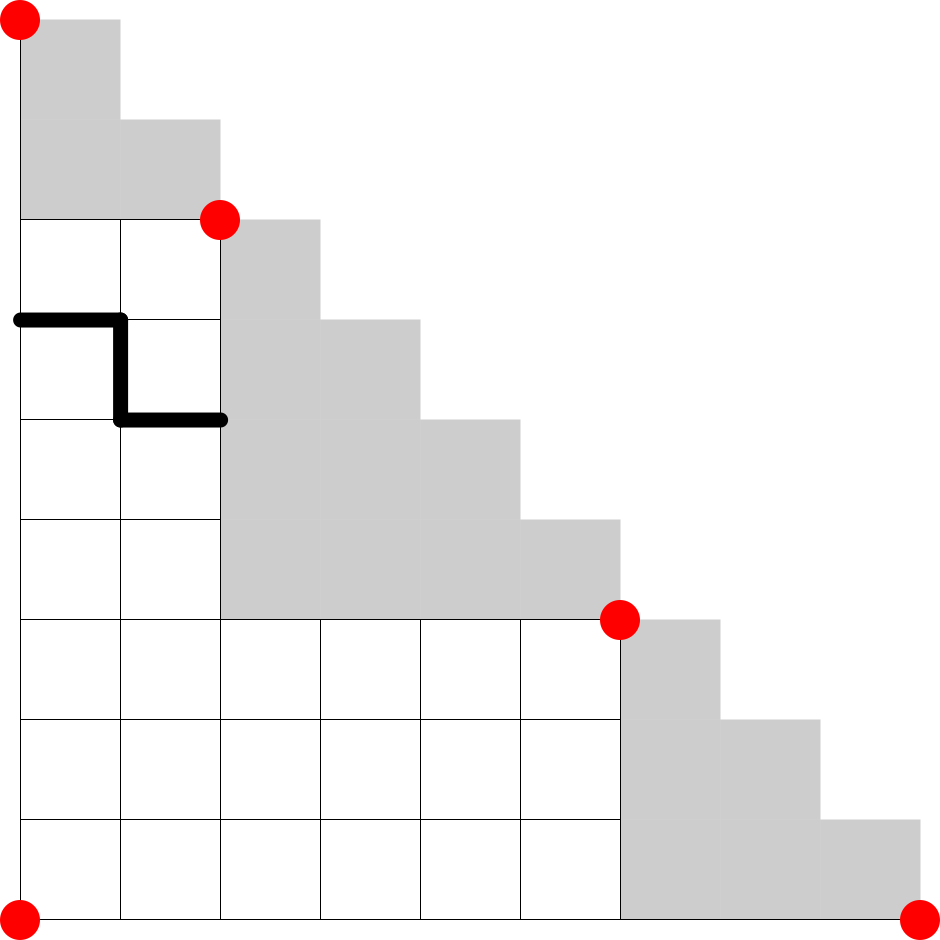}
\caption{Left: partition of interior points when $a_1=2$. Right: example of length 3 chain $D_1$.}
\label{fig:interior-points2}
\end{figure}

WLOG assume $a_1 = 2$. Let $D_1, \ldots, D_{a_2-2}$ denote the length $3$ chains such that $D_1$ is adjacent to the length 2 chain near $t_1$ and $D_k$ is adjacent to $D_{k+1}$ for $1 \le k \le a_2-3$. Once the facets in $D_1$ are fixed, the intersection of the two leftmost facets in $D_2$ is fixed since this equals the intersection of the two rightmost facets in $D_1$. So $D_1$ could not have been flipped, and $D_2$ is also not flipped. Similarly, we can argue that $D_3, \ldots, D_{a_2-2}$ are not flipped. If $D_{a_2-2}$ is adjacent to a chain of length $>2$, repeating this argument shows that this chain cannot be flipped. In particular, all of these chains have the same orientation.
\end{proof}

\subsection{The Automorphism Group}\label{subsec:proving-autom-gp}
In this section, we finish the proofs of Theorems~\ref{thm:autom-m=2} and \ref{thm:autom-m>=3}. Recall that in Section~\ref{subsec:sym}, we showed that the desired groups of automorphisms are contained in $\Aut(\GT_\lam)$. We will now show equality by bounding the size of $\Aut(\GT_\lam)$. We will begin with a few preliminaries.

For group $G$ acting on the set of facets of $\GT_\lam$, we adopt the following conventions:
\begin{itemize}
\item For any set of facets $X$, we denote the stabilizer of $X$ by $G(X) := \{g \in G : g \cdot x = x \ \forall x \in X \}$. We denote the orbit of $X$ with respect to a subgroup $H \subset G$ by $O_H(X)$. 
\item Let $|G|$ denote the order of a group $G$.
\item For any set of facets $X$, let $X^c$ denote the set of facets not in $X$.
\end{itemize}

We will analyze the action of an automorphism on the facets of $\GT_\lam$ and apply the Orbit-Stabilizer theorem, which states that $|G| = |O_G(X)||G(X)|$. Below we finish the proof of Theorem~\ref{thm:autom-m=2}.
\begin{proof}[Proof of Thm. \ref{thm:autom-m=2} cont.]
In Section~\ref{subsec:sym}, we showed that $D_4 \times \Z_2 \times \Z_2^{\delta_{a_1, a_2}} \subseteq \Aut(\GT_\lambda)$. Now we show that the order of $\Aut(\GT_\lambda)$ is at most the order of this group. 

Let $\phi \in \Aut(\GT_\lambda)$. By Lemma~\ref{lem:fcmapstofc}, $\phi$ preserves lengths of chains, so we may consider the action of $\phi$ on chains of the same length. There are two facets in length $1$ chains. Let $X_{1}$ denote the set of these two facets. Let $C_0 = (F_{0,0}, F_{0,1})$ denote the length $2$ chain near the origin and let $C_2 = (F_{2,0}, F_{2,1})$ denote the length $2$ chain near $t_1$. Let $X_{2,0} := \{F_{0,0},F_{0,1}\}$ and $X_{2,2} := X_{2,0} \cup \{F_{2,0},F_{2,1}\}$.

To preserve lengths of chains, $\phi$ must send $F_{0,0}$ to a facet in a length $2$ chain so $|O_G(F_{0,0})| \le 4$. If $\phi$ fixes $F_{0,0}$, then $\phi$ fixes $F_{0,1}$ and we have $|O_{G(X_{2,0})}(F_{2,0})| \le 2$ and $|O_{G(X_{2,0} \cup \{F_{2,0}\})}(F_{2,1})| = 1$. Any $\phi \in G(X_{2,2})$ can either fix or exchange the facets in $X_0$ so $|O_{G(X_{2,2})}(X_0)| \le 2$.

Recalling Definition~\ref{def:adj-graph}, if $\phi$ fixes all chains of length $\le 2$, then $\phi$ fixes every node in $\cal{G}_\lam$. By Lemma~\ref{lem:chain-flip}, the orientation of all non-leaves must be the same. Now, we must show that $\phi \in G(X_0 \cup X_{2,2})$ can only flip all non-leaves if $\lam = \lam'$. Note that the corresponding tree $\mathcal{G}_\lambda$ is a path. If we have such a $\phi$ flipping each chain, we can show that every chain is symmetric about the line $y = x$ by inducting on the length from $C_0$. Every chain being symmetric about $y=x$ implies the entire diagram is symmetric about $y=x$, implying $\lambda = \lambda'$. 

Therefore, $|O_{G(X_0 \cup X_{2,2})}(F')| \le 2^{\delta_{a_1,a_2}}$ and $|G(X_0 \cup X_{2,2} \cup F')| = 1$.

So we have 
\begin{align*}
|G| &= |O_G(X_{2,0})||O_{G({X_{2,0}})}(X_{2,2} \backslash X_{2,0})||O_{G(X_{2,2})}(X_0)||O_{G(X_0 \cup X_{2,2})}(F')||G(X_0 \cup X_{2,2} \cup \{F'\})| \\
&\le 4 \cdot 2 \cdot 2 \cdot 2^{\delta_{a_1,a_2}} \cdot 1.
\end{align*}
We conclude $\Aut(\GT_\lambda) \cong D_4 \times \Z_2 \times \Z_2^{\delta_{a_1, a_2}}$.
\end{proof}

\begin{remark*}
It is natural to ask how the automorphism group of $\GT_\lam$ is related to the automorphism group of the 1-skeleton of $\GT_\lam$. Clearly $\Aut(\GT_\lam)$ is contained within the automorphism group of the 1-skeleton, but we do not always have equality. If $\lam = (3,3)$, then $\Aut(\GT_\lam)$ has size $16$ while the automorphim group of the 1-skeleton has size $32$.  Numerical computations suggest that we may have equality in all other cases when $m=2$.
\end{remark*}

The proof of Theorem~\ref{thm:autom-m>=3} uses similar arguments. We begin by establishing some useful lemmas about the image of chains of length $1$ or $2$ under any automorphism.

We label the length $1$ chains in $\gamlam$ as follows: let $\cal{C}_1$ be the set of length $1$ chains occurring to the left of $t_1$ and let $\cal{C}_2$ be the set of length $1$ chains occurring to the right of $t_{m-1}$. Let $X_1$ denote the set of facets in chains $\cal{C}_1 \cup \cal{C}_2$. Figure~\ref{fig:length1} shows all the length 1 chains visually. Let $D_1 \in \cal{C}_1$ and $D_2 \in \cal{C}_2$ be the length $1$ chains with smallest y-coordinate and x-coordinate respectively. 

\begin{figure}[h!]
\includegraphics[scale=0.12]{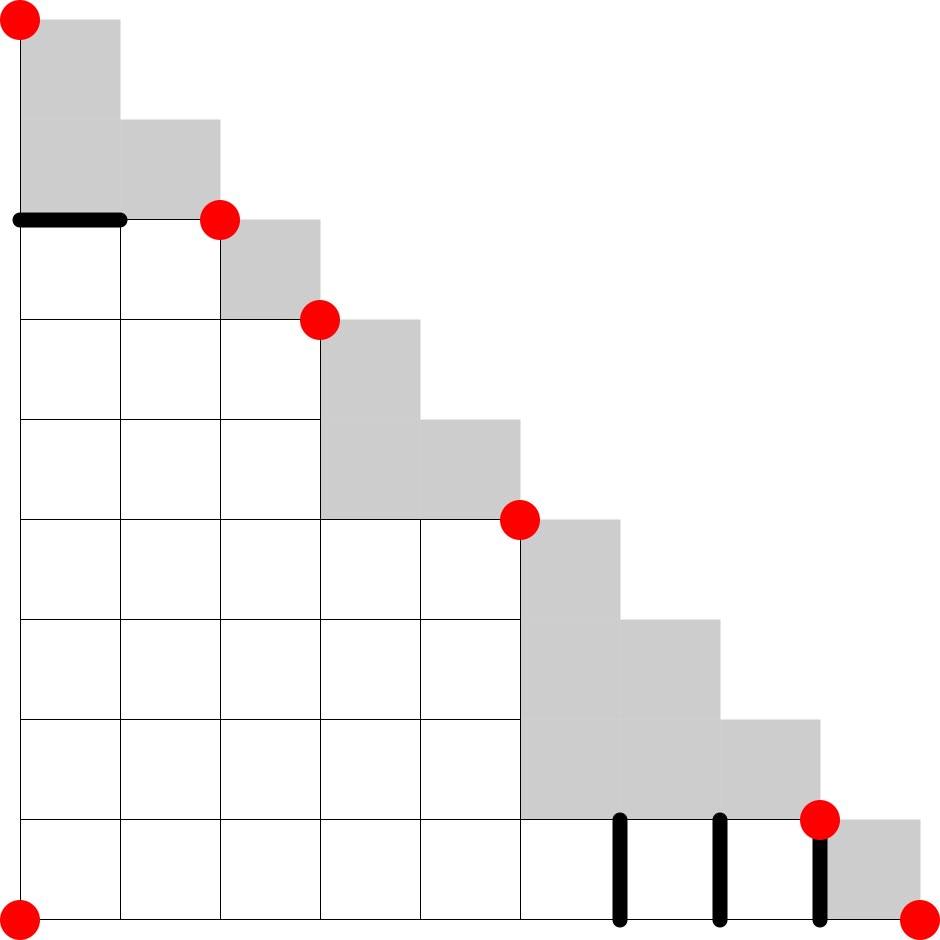}
\caption{All length 1 chains.}
\label{fig:length1}
\end{figure}

We label the length $2$ chains in $\gamlam$ as follows: let $C_0$ denote the length $2$ chain near the origin. Let $C_2, C_3, \ldots, C_{2m-2}$ denote the length $2$ chains that occur along the border of $\gamlam$ where each even index chain $C_{2k}$ (which may not exist) occurs near terminal vertex $t_k$ and each odd index chain $C_{2k-1}$ occurs at the corner of the $k$th shaded triangular subgrid corresponding to the coordinates that are fixed by $k^{a_k}$ in $\lam$. We call all odd index chains $C_{2k-1}$, $k=2,\ldots,m-1$, the \emph{type A} chains and all even index chains $C_{2k}$, $k=1,\ldots,m-1$ (excluding $C_0$) the \emph{type B} chains. As shown in Figure~\ref{fig:chain-types}, the type B chain $C_{2k}$ will not exist if $a_k = 1$ or $a_{k+1}=1$ while all type A chains will exist. Our notation implicitly assumes that all of $C_2, \ldots, C_{2m-2}$ exist but all of our arguments hold with reference to the chains that actually exist for any given $\lam$. For $0 \le k \le 2m-4$, let $X_{2,k}$ denote the set of facets in chains $C_0, C_2, C_3, \ldots, C_k$ and let $X_2 := X_{2,2m-2}$ denote the set of facets in chains $C_0, C_2, C_3, \ldots, C_{2m-2}$.

\begin{figure}[h!]
\includegraphics[scale=0.12]{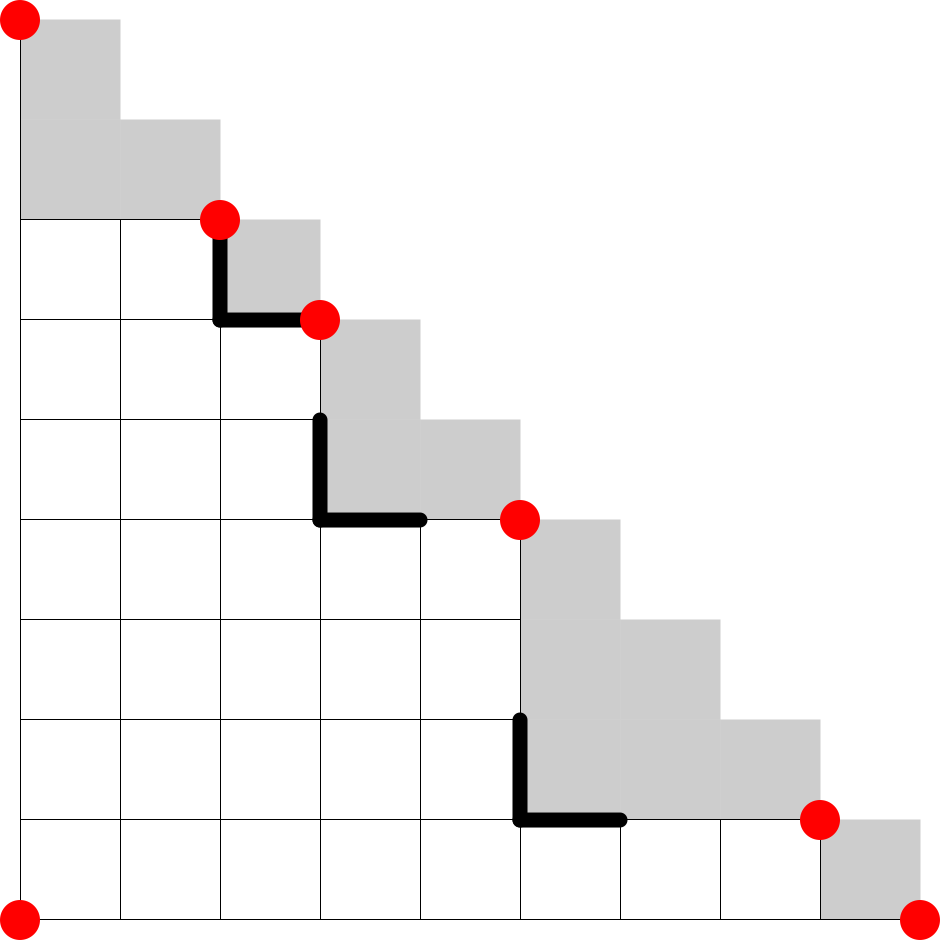}
\qquad
\includegraphics[scale=0.12]{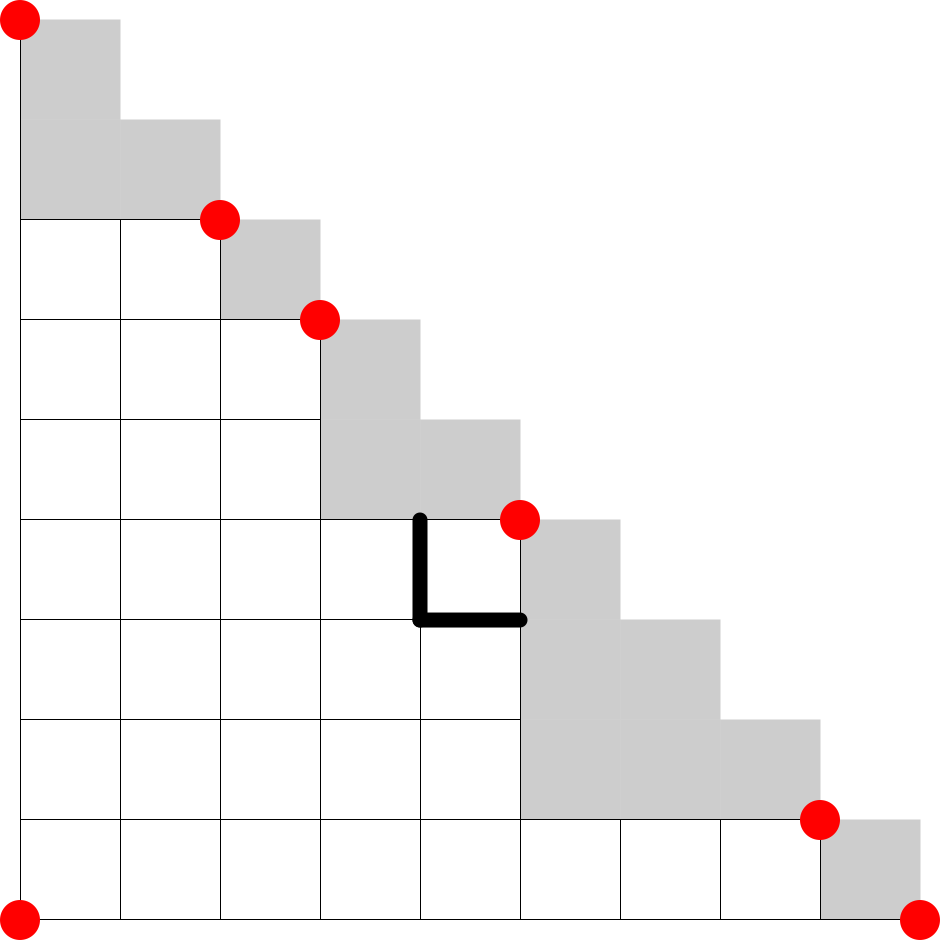}
\caption{Left: all type A chains. Right: all type B chains.}
\label{fig:chain-types}
\end{figure}

\begin{lem} 
The chains $D_1$, $D_2$, $C_2$, $C_3$, $\ldots$ , $C_{2m-2}$ can be ordered into a sequence $D_1$, $C_2$, $C_3$, $\ldots$ , $C_{2m-2}$, $D_2$ such that any $\phi \in \Aut(\GT_\lam)$ preserves the ordering within this sequence. In particular, any $\phi$ sends each chain in this sequence to itself or $\phi$ flips the order of the sequence.
\label{lem:seq-of-chains}
\end{lem}

Figure~\ref{fig:sequence} shows a picture of the sequence of chains mentioned in Lemma~\ref{lem:seq-of-chains}.

\begin{figure}[h!]
\includegraphics[scale=0.12]{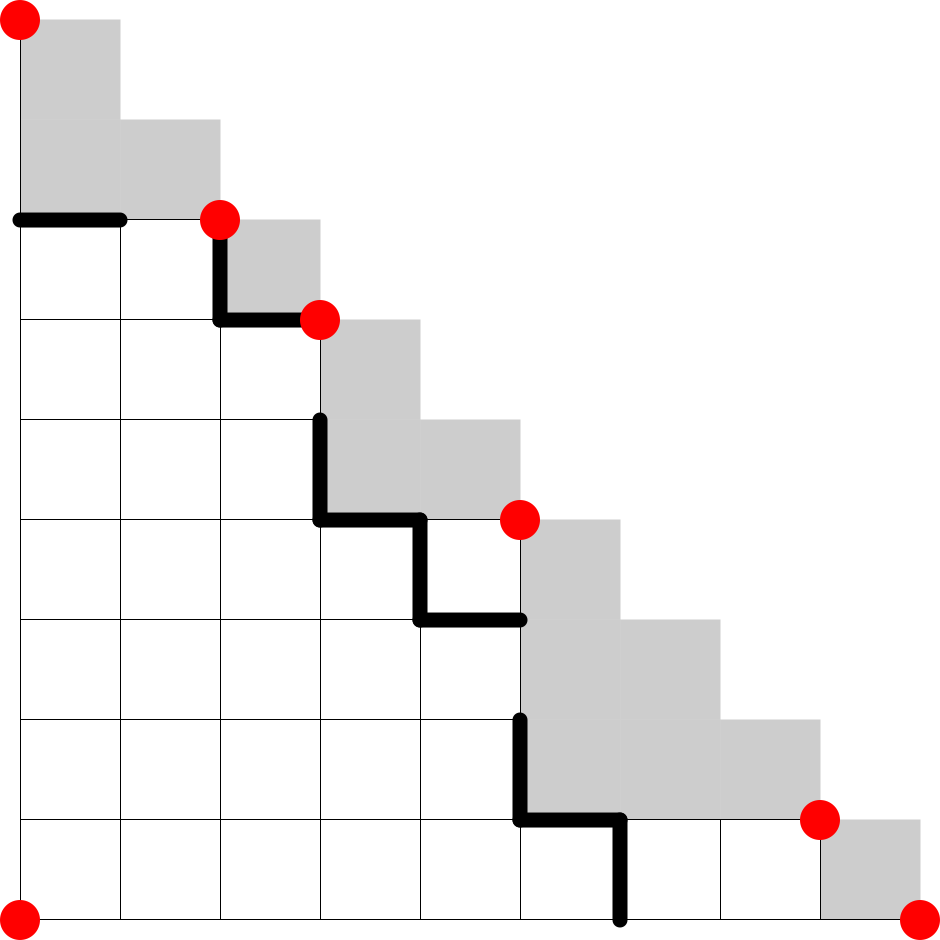}
\caption{The sequence of chains $D_1,C_2,\ldots,C_{2m-2},D_2$.}
\label{fig:sequence}
\end{figure}

\begin{proof}
By Lemma~\ref{lem:fcmapstofc}, $\phi$ preserves lengths of chains. We say two length $1$ chains are incompatible if there does not exist a vertex containing both of them and we say they are compatible otherwise. Note that all chains in $\cal{C}_1$ are incompatible with each other but compatible with any chain in $\cal{C}_2$. The set of chains that is the image of $\cal{C}_1$ under $\phi$ must all be incompatible with each other, so $\phi$ must map all chains of $\cal{C}_1$ to chains in $\cal{C}_1$ or it must all chains of $\cal{C}_1$ to chains in $\cal{C}_2$. The latter can only occur if the sizes of $\cal{C}_1$ and $\cal{C}_2$ are the same. The rest of this proof and the proof of Lemma~\ref{lem:seq-flip-lam-lam'} will show that this can only occur if $\lam = \lam'$.

Thus $|O_G(X_1)| \le (a_2)!^{\delta_{a_1,1}}\cdot(a_{m-1})!^{\delta_{a_m,1}}\cdot 2^{t}$ where $t = 1$ if $\lam = \lam'$ and $t = 0$ otherwise. Any $\phi \in G(X_1)$ must send $C_0$ to itself because it is the only length $2$ chain such that for either of its facets, there exist vertices containing this facet and any facet in a type A chain. Note that $\phi$ may send the facets of $C_0$ to each other. 

Chain $D_i$ and a type A chain are incompatible if there does not exist a vertex containing the facet in the length $1$ chain and both facets of the length $2$ chain. Visually, the type A chain closest to the length $1$ chain will be the only type A chain incompatible with it. Two type A chains are incompatible if there does not exist a vertex containing all four facets in these chains. Visually, two type A chains are incompatible iff they occur at the corners of adjacent shaded triangular subgrids (these are the subgrids corresponding to the fixed entries of $\cal{I}_n$). We can form a sequence starting with $D_1$ where two adjacent chains in this sequence are incompatible. Visually, this sequence corresponds to reading $D_1,D_2$ and the type A chains from left to right. 

Now we insert the type B chains into this sequence. A type B chain and $D_i$ are incompatible if there does not exist a vertex containing both facets in the type B chain and the facet in $D_i$. A type B chain and type A chain are incompatible if there does not exist a vertex containing all four facets in these chains. Visually, a type B chain is incompatible only with the chains adjacent to it. So we can insert the type B chains into the sequence so that the sequence corresponds to reading $D_1,D_2$ and all length $2$ chains from left to right.

Note that $\phi \in G(X_1)$ must preserve these relations of incompatibility so $\phi$ fixes all length $2$ chains. Finally, note that $\phi$ cannot map the two facets in a type A chain to each other: once $D_1$ is fixed, there does not exist a vertex containing the facet in $D_1$ and the vertical facet in the type A chain incompatible with $D_1$, but there does exist such a vertex for the horizontal facet. This fixes the image of the facets in this type A chain. Similarly, if the facets in a type A chain are fixed, then the facets in a type A chain incompatible with it will also be fixed. Of all length $2$ chains, $\phi$ can only swap the facets in type B chains and chain $C_0$ and hence $|O_{G(X_1)}(X_2)| \le 2^{r_1+1}$ where $r_1$ is the number of type B chains.
\end{proof}

\begin{lem}
If the sequence of chains $D_1,C_2,C_3,\ldots,C_{2m-2},D_2$ is flipped by some $\phi \in \Aut(\GT_\lam)$, then $\lam = \lam'$.
\label{lem:seq-flip-lam-lam'}
\end{lem}
\begin{proof}
Assume this sequence of chains is flipped. Define the distance between $D_1$ and $C_3$ (which always exists) to be the cardinality of the minimal set of facets $\cal{S}$, containing no other facets from this sequence, such that the intersection of facet $D_1$, one facet from $C_3$, and the facets in $\cal{S}$ is empty. We give examples of such sets in Figures~\ref{fig:distance}. This distance is $a_1 + a_2-2$ since such a set must visually separate the coordinates that are fixed as 1 and 2 (the top-leftmost gray squares in Fig. \ref{fig:distance}). Similarly, the distance between $D_2$ and $C_{2m-3}$ is $a_{m-1} + a_m-2$. Note that $\phi$ must preserve this distance, as this property can be realized in the face lattice of the polytope. So if $\phi$ flips the sequence of chains, then $a_1 + a_2 = a_{m-1} + a_m$.

\begin{figure}[h!]
\includegraphics[scale=0.12]{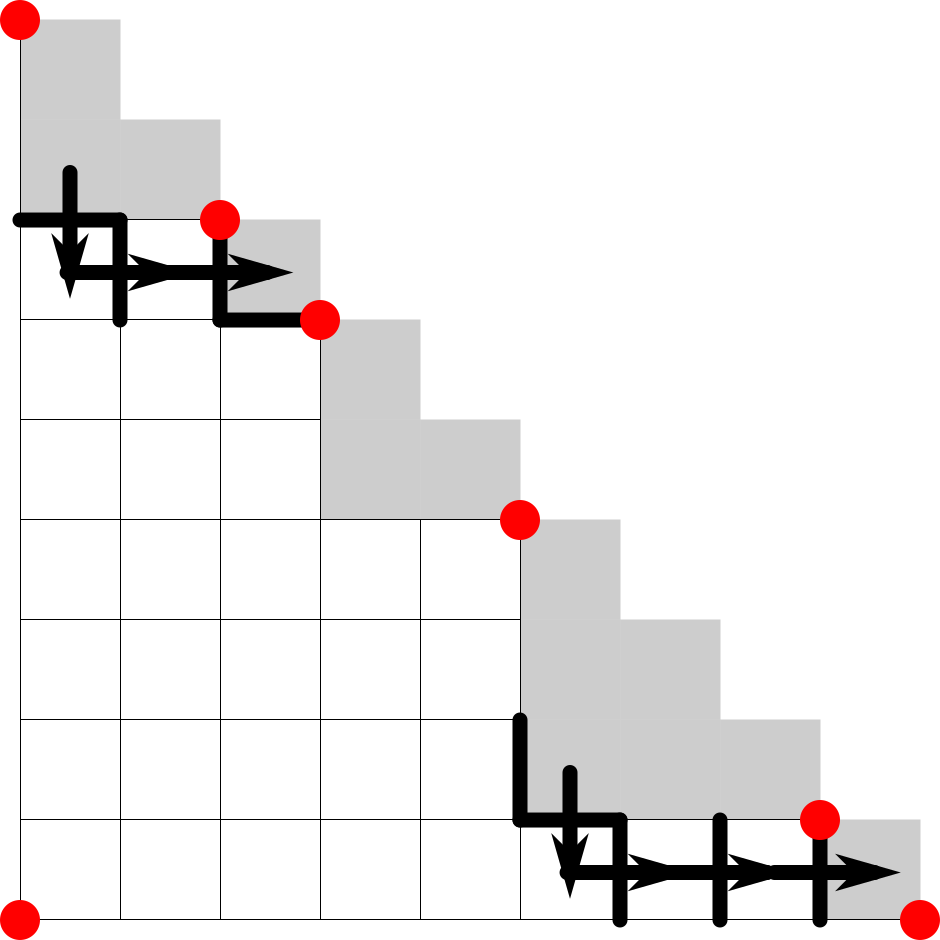}
\caption{The distance between $D_1$ and $C_3$ is 1 and the distance between $D_2$ and $C_{2m-3}$ is 2.}
\label{fig:distance}
\end{figure}

Recalling the definition of $\cal{G}_\lam$, we can define the depth of a node as one plus the number of edges between it and the root. Then the depth of $C_{2k-1}$ is 
$$a_1 + \ldots + a_{k-1} + a_{k+1} + \ldots + a_m = \sum_{i = 1}^m a_i - a_{k}.$$

Note that $\phi$ must send type A chains to type A chains. 
Since $\phi$ is an automorphism of $\cal{G}_\lam$, it must preserve depth so if $\phi$ flips the sequence of chains, then it sends $C_{2k-1}$ to chain $C_{2m-2k+1}$ so $a_{k} = a_{m+1-k}$ for $2 \le k \le m-1$. Since $a_1 + a_2 = a_{m-1} + a_m$, we have shown $a_{k} = a_{m+1-k}$ for $1 \le k \le m$ so $\lam = \lam'$.
\end{proof}

Now we are in a position to finish the proof of Theorem~\ref{thm:autom-m>=3}.

\begin{proof}[Proof of Thm. \ref{thm:autom-m>=3} cont.]
In Section~\ref{subsec:sym}, we showed that $(S_{a_2}^{\delta_{1,a_1}} \times S_{a_{m-1}}^{\delta_{1,a_{m}}} \times \Z_2^{r_1+1}) \ltimes_\varphi \Z_2^{r_2} \subseteq \Aut(\GT_\lambda)$. Now we show that the order of $\Aut(\GT_\lambda)$ is at most the order of this group. 

By Lemmas~\ref{lem:seq-of-chains} and \ref{lem:seq-flip-lam-lam'}, it suffices to bound $|G(X_1 \cup X_2)|$. Consider any automorphism $\phi \in G(X_1 \cup X_2)$. We can argue up the rooted tree $\cal{G}_\lam$ to show that all of its vertices are fixed under $\phi$. Specifically, $\phi$ must preserve adjacencies between chains so it is an isomorphism of the tree fixing the leaves of $\cal{G}_\lam$ and hence must act trivially on the nodes. Furthermore, Lemma~\ref{lem:adj-chains-same-orientation} implies that the orientation of any two non-leaves must be the same. If $\cal{G}_\lam$ has at least two leaves, then the common ancestor of these two leaves cannot be flipped by $\phi$ or else the direct children of this node cannot be fixed. Otherwise, $\cal{G}_\lam$ is a chain of nodes. If $m \ge 4$, then there are always at least two type A chains and hence at least 2 leaves. Similarly, if $m = 3$ and either $a_1,a_2 \ge 2$ or $a_2, a_3 \ge 2$, then there will be one type A chain and at least one type B chain. So $\cal{G}_\lam$ can only be a chain of nodes if $\lam = (1, 2^{a_2}, 3)$ or $\lam = (1^{a_1}, 2, 3^{a_3})$. In the first case, $\cal{G}_\lam$ contains two nodes, both of which have length 2. In the second case, WLOG assume that $a_1\geq2$. Consider the chain ending at a facet $F$ that shares a vertex of $\gamlam$ with facet in $D_1$. An example of this chain is shown in Figure~\ref{fig:a_2=1}. If this chain were flipped, then there exists a vertex containing the image of $F$ and the facet in $D_1$, but there does not exist a vertex containing $F$ and the facet in $D_1$. Hence all of the non-leaves in $\cal{G}_\lam$ are not flipped. 

\begin{figure}[h!]
\includegraphics[scale=0.12]{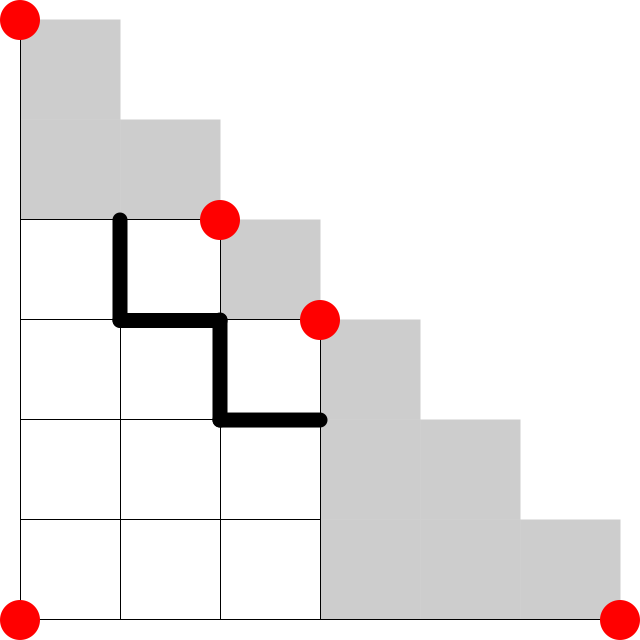}
\caption{An example of the chain containing facet $F$.}
\label{fig:a_2=1}
\end{figure}

Therefore, $|G(X_1 \cup X_2)| = 1$ and we have 
\begin{align*}
|G| &= |O_{G}(X_1)||O_{G(X_1)}(X_2)||G(X_1 \cup X_2)| \\
&\le a_2!^{\delta_{a_1,1}} \cdot a_{m-1}!^{\delta_{a_{m},1}} \cdot 2^{r_1+1} \cdot 2^{r_2} \\ 
&= |(S_{a_2}^{\delta_{1,a_1}} \times S_{a_{m-1}}^{\delta_{1,a_{m}}} \times \Z_2^{r_1+1}) \ltimes_\varphi \Z_2^{r_2}|.
\end{align*}
Since $(S_{a_2}^{\delta_{1,a_1}} \times S_{a_{m-1}}^{\delta_{1,a_{m}}} \times \Z_2^{r_1+1}) \ltimes_\varphi \Z_2^{r_2}$ is contained in $\Aut(\GT_\lam)$, we have $\Aut(\GT_\lam) \cong (S_{a_2}^{\delta_{1,a_1}} \times S_{a_{m-1}}^{\delta_{1,a_{m}}} \times \Z_2^{r_1+1}) \ltimes_\varphi \Z_2^{r_2}$, as desired.
\end{proof}


\section*{Acknowledgements}
This research was carried out as part of the 2016 REU program at the School of Mathematics at University of Minnesota, Twin Cities, and was supported by NSF RTG grant DMS-1148634 and NSF grant DMS-1351590. The authors are especially grateful to Victor Reiner for his mentorship and support, and for many fruitful conversations. The authors would also like to thank Elise delMas and Craig Corsi for their valuable advice and comments. 

\bibliography{main}
\bibliographystyle{alpha}
\end{document}